\documentclass[reqno]{amsart}
\usepackage{amsmath}
\usepackage{amsfonts}
\usepackage{amssymb}
\usepackage{color}
\usepackage{graphicx}
\usepackage{appendix}
\allowdisplaybreaks[4]

\newtheorem{thm}{Theorem}[section]
\newtheorem{cor}[thm]{Corollary}
\newtheorem{lem}[thm]{Lemma}

\newtheorem{rem}[thm]{Remark}
\newtheorem{cjt}[thm]{Conjecture}
\theoremstyle{definition}

\newcommand{\be}{\begin{equation}}
\newcommand{\ee}{\end{equation}}
\newcommand{\bea}{\begin{eqnarray}}
\newcommand{\eea}{\end{eqnarray}}
\newcommand{\ben}{\begin{eqnarray*}}
\newcommand{\een}{\end{eqnarray*}}
\newcommand{\bt}{\begin{split}}
\newcommand{\et}{\end{split}}
\newcommand{\bet}{\begin{equation}
\begin{split}}
\newcommand{\eet}{\end{split}
\end{equation}}

\definecolor{green}{rgb}{0,1,0}
\definecolor{yellow}{rgb}{1,1,0}
\definecolor{orange}{rgb}{1,.7,0}
\definecolor{red}{rgb}{1,0,0}
\definecolor{white}{rgb}{1,1,1}

\usepackage{pinlabel}

\begin{document}

\title[]
{Getzler relation and Virasoro Conjecture for genus one}
\date{}
\author{Yijie Lin}

\address{Yijie Lin: School of Mathematics (Zhuhai)\\Sun Yat-Sen University\\Zhuhai, 519082, China}
\email{yjlin12@163.com}

\maketitle

\begin{abstract}
We derive explicit universal equations for primary Gromov-Witten invariants by applying Getzler's genus one relation to quantum powers of Euler vector field. As an application, we provide some evidences for the genus-1 Virasoro conjecture.
\end{abstract}

%{\bf MSC(2010):} 14N35.

%{\bf Keywords:} Getzler's genus one relation, Virasoro conjecture, quantum product, Gromov-Witten invariants

%\tableofcontents
\section{Introduction}
The Virasoro conjecture of Eguchi-Hori-Xiong \cite{EHX} predicts some mysterious relations between Gromov-Witten invariants in all genera, that is, a sequence of operators are conjectured to annihilate the generating functions of Gromov-Witten invariants of smooth projective varieties. It is equivalent to Witten conjecture \cite{Wi1} proved by Kontsevich \cite{Kon} when the underlying mainfold is one point. For manifolds with semisimple quantum product, this conjecture was proved in lower genera \cite{DZ1,DZ2,Get1,LT,L1,L3,L4}, and completely solved by Teleman \cite{Tel}.
The genus-0 part of the Virasoro conjecture has been  proved firstly in \cite{LT} without assumption of semisimplicity, and later by other authors \cite{DZ1,Get1}. It is still open  for the general case of the genus-1 Virasoro conjecture. \par
\par

In \cite{L1}, it is proved that the genus-1 Virasoro conjecture can be reduced to some simple equation derived by restricting the genus-1 $L_{1}$-constraint on the small phase space. This equation is motivated by using Getzler's genus one relation \cite{Get2} as follows
\be
\textbf{G}(E^{k_{1}},E^{k_{2}},E^{k_{3}},E^{k_{4}})=0,\label{Liu1}
\ee
where $E^{k}$ is $k^{th}$ quantum power of Euler vector field, and the definition of $\textbf{G}$ is presented in section 2.
Furtherly, the author in \cite{L2}  shows that to prove the genus-1 Virasoro conjecture, it is enough to  prove the property that the derivative of
that simple equation along the direction of any vector field vanishes. Besides the identity of the ordinary cohomology ring satisfying this property, he also find one more  vector field called quantum volume element  by the computation of
\be\label{Liu2}
\sum_{\alpha}\textbf{G}(E,E,\gamma^{\alpha},\gamma_{\alpha})=0,
\ee
where $\gamma^{\alpha}$ and $\gamma_{\alpha}$ belong to the space of cohomology classes.\par
In this paper, we will consider the following form
\be\label{explicitformula1}
\textbf{G}(E^{k},\gamma_{\alpha},\gamma_{\beta},\gamma_{\sigma})=0.
\ee
The explicit formula of equation \eqref{explicitformula1} is computed in section 3 (cf. Theorem \ref{Getzler-1E}), which implies the explicit expressions (cf. Theorems \ref{Getzler-2E} and \ref{Getzler-3E}) for the following two equations
\bea
&&\textbf{G}(E^{k_{1}},E^{k_{2}},\gamma_{\alpha},\gamma_{\beta})=0,\label{explicitformula2}\\
&&\textbf{G}(E^{k_{1}},E^{k_{2}},E^{k_{3}},\gamma_{\alpha})=0.\label{explicitformula3}
\eea
The explicit formula \eqref{explicitformula3} generalizes  equation \eqref{Liu1}, and its equivalent result, i.e.,
Theorem \ref{mainresult1} implies some  generalized version (cf. Corollaries \ref{mainresult2} and \ref{mainresult3}) of the Virasoro type relation for $\{\Phi_{k}\}$ in \cite{L1}. And formula \eqref{explicitformula2} implies equation \eqref{Liu2}, and provides more evidences (cf. Theorem \ref{mainresult6}) for the genus-1 Virasoro conjecture
by computing
\ben
&&\sum_{\mu}\textbf{G}(E^{k_{1}},E^{k_{2}},\gamma_{\mu},\gamma^{\mu}\circ\gamma_{\alpha})=0.
\een
These evidences give an alternative proof of the genus-1 Virasoro conjecture for any manifold with semisimple quantum cohomology (cf. Corollary \ref{mainresult61}).
We also derive some new relation (cf. Theorem \ref{mainresult7}) from the general equation \eqref{explicitformula1}.

An outline of this paper is as follows. In section 2, we recall some basic definitions, present some known facts for the genus-1 Virasoro conjecture and
recollect universal equations for primary Gromov-Witten invariants which will be used later.
In section 3, we firstly obtain  an explicit formula for the first derivative of $\Phi_{k}$, and then derive explicit universal equations from Getzler's genus one relation involving some quantum powers of Euler
vector field.
In section 4, we consider applications to the genus-1 Virasoro conjecture.

{\bf Acknowledgements.}
The author would like to thank Professor Xiaobo Liu for many helpful suggestions, and Professor Jian Zhou for
his encouragement.

\section{Preliminaries}
In this section, we recall Gromov-Witten invariants, quantum product, and some known facts for the genus-1 Virasoro conjecture. We will also recollect some universal equations for primary Gromov-Witten invariants, and fix some notation.
\subsection{Gromov-Witten invariants, quantum product and  Virasoro conjecture}
Let $X$ be a smooth projective variety of dimension $d$, and denote by $N$  the dimension of the space of cohomology classes $H^{*}(X,\mathbb{C})$. We assume $H^{odd}(X;\mathbb{C})=0$ for simplicity, and fix a basis $\{\gamma_{1},\cdot\cdot\cdot,\gamma_{N}\}$ of $H^{*}(X,\mathbb{C})$. Let $\gamma_{1}$ be the identity of the cohomology ring of $X$ and $\gamma_{\alpha}\in H^{p_{\alpha},q_{\alpha}}(X,\mathbb{C})$ for each $\alpha$.  Let $\eta_{\alpha\beta}=\int_{X}\gamma_{\alpha}\cup\gamma_{\beta}$ be the intersection
form on $H^*(X,\mathbb{C})$, and $\mathcal{C}=(\mathcal{C}_{\alpha}^{\beta})$ be the matrix satisfying
\ben
c_{1}(X)\cup\gamma_{\alpha}=\sum_{\beta}\mathcal{C}_{\alpha}^{\beta}\gamma_{\beta}.
\een
The symmetric matrices $\eta=(\eta_{\alpha\beta})$ and $\eta^{-1}=(\eta^{\alpha\beta})$ are used to lower and raise indices. For example,
$\gamma^{\alpha}=\sum_{\mu}\gamma_{\mu}\eta^{\mu\alpha}$. Let
$b_{\alpha}=p_{\alpha}-\frac{1}{2}(d-1)$. It is easy to verify that if $\eta^{\alpha\beta}\neq0$ or $\eta_{\alpha\beta}\neq0$, then
$b_{\alpha}=1-b_{\beta}$. For
$\Lambda\in H_{2}(X,\mathbb{Z})$, let $\overline{\mathcal{M}}_{g,k}(X,\Lambda)$ be the moduli space of stable map with $[\overline{\mathcal{M}}_{g,k}(X,\Lambda)]^{vir}$ as its virtual fundamental class, and $\mathbb{L}_{i}$  the tautological line bundle over $\overline{\mathcal{M}}_{g,k}(X,\Lambda)$.  Let $\widetilde{\gamma}_{1},\cdot\cdot\cdot,\widetilde{\gamma}_{k}\in H^{*}(X,\mathbb{C})$, the  genus-$g$ descendant Gromov-Witten invariants are defined by
\bea
&&\langle\tau_{n_{1}}(\widetilde{\gamma}_{1})\cdot\cdot\cdot\tau_{n_{k}}(\widetilde{\gamma}_{k})\rangle_{g}\nonumber\\
&=&\sum_{\Lambda\in H_{2}(X,\mathbb{Z})}q^{\Lambda}\int_{[\overline{\mathcal{M}}_{g,k}(X,\Lambda)]^{vir}}\prod_{i=1}^k c_{1}(\mathbb{L}_{1})^{n_{i}}\cup\mathbf{ev}_{i}^{*}(\widetilde{\gamma}_{i}),\label{descendentGW}
\eea
where $q^{\Lambda}$ is in Novikov ring with the product defined by $q^{\Lambda_{1}}q^{\Lambda_{2}}=q^{\Lambda_{1}+\Lambda_{2}}$, and $\mathbf{ev}_{i}: \overline{\mathcal{M}}_{g,k}(X,\Lambda)\longrightarrow X$ is the $i$-th evaluation map $(C;x_{1},\cdot\cdot\cdot,x_{k};f)\mapsto f(x_{i})$. One can refer \cite{CK} for more details. As in \cite{L1}, for any $\tau_{n}(\gamma_{\alpha})$, one associates a parameter $t_{n}^{\alpha}$, and call
the space of all $T=\{t_{n}^{\alpha}: n\in\mathbb{Z}^{\geq0}, \alpha=1,\cdot\cdot\cdot,N\}$ the big phase space and its subspace spanned by $\{T|t_{n}^{\alpha}=0 \mbox{ if } n>0\}$ the small phase space. For simplicity, we identify $\tau_{n}(\gamma_{\alpha})$ with tangent vector field $\frac{\partial}{\partial t_{n}^{\alpha}}$ on the big phase space, and on the small phase space, we write $t_{0}^{\alpha}$ as $t^{\alpha}$ and identify
$\gamma_{\alpha}$ with $\frac{\partial}{\partial t^{\alpha}}$. If we restrict the invariant \eqref{descendentGW} on the small phase space, i.e., setting $n_{i}=0$
for all $1\leq i\leq k$, the resulting invariants are called primary Gromov-Witten invariants.

The generating function of genus-$g$ Gromov-Witten invariants is defined as
\ben
F_{g}(T):=\sum_{k\geq0}\frac{1}{k!}\sum_{n_{1},\alpha_{1},\cdot\cdot\cdot, n_{k},\alpha_{k}}t_{n_{1}}^{\alpha_{1}}\cdot\cdot\cdot t_{n_{k}}^{\alpha_{k}}
\langle\tau_{n_{1}}(\gamma_{\alpha_{1}})\cdot\cdot\cdot\tau_{n_{k}}(\gamma_{\alpha_{k}})\rangle_{g}.
\een
And the generating function for Gromov-Witten invariants involving all genera is defined to be
\ben
Z(T;\lambda):=\exp\sum_{g\geq0}\lambda^{2g-2}F_{g}(T).
\een
As in \cite{L4}, we define $k$-point (correlation) function
\ben
&&\langle\langle\mathcal{W}_{1}\mathcal{W}_{2}\cdot\cdot\cdot\mathcal{W}_{k}\rangle\rangle_{g}:
=\sum_{n_{1},\alpha_{1},\cdot\cdot\cdot,n_{k},\alpha_{k}}f_{n_{1},\alpha_{1}}^{1}\cdot\cdot\cdot f_{n_{k},\alpha_{k}}^{k}
\frac{\partial^k}{\partial t_{n_{1}}^{\alpha_{1}}\partial t_{n_{2}}^{\alpha_{2}}\cdot\cdot\cdot\partial t_{n_{k}}^{\alpha_{k}}}F_{g}
\een
for vector fields $\mathcal{W}_{i}=\sum_{n,\alpha}f_{n,\alpha}^{i}\frac{\partial}{\partial t_{n}^{\alpha}}$.
Let $\nabla$ be the covariant derivative defined by
\be
\nabla_{\mathcal{W}}\mathcal{V}=\sum_{n,\alpha}(\mathcal{W}f_{n,\alpha})\tau_{n}(\gamma_{\alpha})\label{cderivative}
\ee
for any vector fields $\mathcal{W}$ and $\mathcal{V}=\sum_{n,\alpha}f_{n,\alpha}\tau_{n}(\gamma_{\alpha})$. It is simple to show that
\be
[\mathcal{V},\mathcal{W}]=\nabla_{\mathcal{V}}\mathcal{W}-\nabla_{\mathcal{W}}\mathcal{V}\label{bracketrelation1}
\ee
and
\be\label{derivative1}
\mathcal{W}\langle\langle \mathcal{V}_{1}\cdot\cdot\cdot \mathcal{V}_{k}\rangle\rangle_{g}=\langle\langle \mathcal{W}\mathcal{V}_{1}\cdot\cdot\cdot \mathcal{V}_{k}\rangle\rangle_{g}+\sum_{i=1}^{k}\langle\langle \mathcal{V}_{1}\cdot\cdot\cdot(\nabla_{\mathcal{W}}\mathcal{V}_{i})\cdot\cdot\cdot \mathcal{V}_{k}\rangle\rangle_{g}
\ee
for any vector fields $\mathcal{V}$, $\mathcal{W}$ and $\mathcal{V}_{i}$.\par
Next, as in \cite{L4}, we define the quantum product of  any two vector fields $\mathcal{V}$ and $\mathcal{W}$ by
\ben
\mathcal{V}\circ\mathcal{W}=\sum_{\alpha}\langle\langle\mathcal{V}\mathcal{W}\gamma^{\alpha}\rangle\rangle_{0}\gamma_{\alpha}.
\een
Obviously, this product is commutative by definition and associative by the following well known generalized WDVV equation
\be\label{WDVV1}
\langle\langle\{\mathcal{V}_{1}\circ\mathcal{V}_{2}\}\mathcal{V}_{3}\mathcal{V}_{4}\rangle\rangle_{0}
=\langle\langle\{\mathcal{V}_{1}\circ\mathcal{V}_{3}\}\mathcal{V}_{2}\mathcal{V}_{4}\rangle\rangle_{0}.
\ee
It follows easily from formula \eqref{derivative1} that
\be\label{derivativeqp}
\nabla_{\mathcal{U}}(\mathcal{V}\circ \mathcal{W})=(\nabla_{\mathcal{U}}\mathcal{V})\circ \mathcal{W}+\mathcal{V}\circ(\nabla_{\mathcal{U}}\mathcal{W})+\sum_{\alpha}\langle\langle \mathcal{U}\mathcal{V}\mathcal{W}\gamma^\alpha\rangle\rangle_{0}\gamma_{\alpha}
\ee
for any vector fields $\mathcal{U}$, $\mathcal{V}$, and $\mathcal{W}$.
Define
\be
G(\mathcal{V}):=\sum_{n,\alpha}(n+b_{\alpha})f_{n,\alpha}\tau_{n}(\gamma_{\alpha})\label{Goperator}
\ee
for any vector field $\mathcal{V}=\sum\limits_{n,\alpha}f_{n,\alpha}\tau_{n}(\gamma_{\alpha}$).
We have the following easy observations
\bea
&&\sum_{\mu}\langle\langle\{\mathcal{V}_{1}\circ\gamma^{\mu}\}\{\mathcal{V}_{2}\circ\mathcal{V}_{3}\circ\gamma_{\mu}\}\mathcal{V}_{4}\cdot\cdot\cdot\mathcal{V}_{k}\rangle\rangle_{g}\nonumber\\
&=&\sum_{\mu}\langle\langle\{\mathcal{V}_{1}\circ\gamma_{\mu}\}\{\mathcal{V}_{2}\circ\mathcal{V}_{3}\circ\gamma^{\mu}\}\mathcal{V}_{4}\cdot\cdot\cdot\mathcal{V}_{k}\rangle\rangle_{g}\nonumber\\
&=&\sum_{\mu}\langle\langle\{\mathcal{V}_{1}\circ\mathcal{V}_{2}\circ\gamma^{\mu}\}\{\mathcal{V}_{3}\circ\gamma_{\mu}\}\mathcal{V}_{4}\cdot\cdot\cdot\mathcal{V}_{k}\rangle\rangle_{g}\label{Observation1}
\eea
and
\bea
&&\sum_{\mu}\langle\langle G(\mathcal{V}_{1}\circ\mathcal{V}_{2}\circ\gamma^{\mu})\{\mathcal{V}_{3}\circ\mathcal{V}_{4}\circ\gamma_{\mu}\}\mathcal{V}_{5}\cdot\cdot\cdot\mathcal{V}_{k}\rangle\rangle_{g}\nonumber\\
&=&\sum_{\mu}\langle\langle G(\mathcal{V}_{1}\circ\mathcal{V}_{2}\circ\gamma_{\mu})\{\mathcal{V}_{3}\circ\mathcal{V}_{4}\circ\gamma^{\mu}\}\mathcal{V}_{5}\cdot\cdot\cdot\mathcal{V}_{k}\rangle\rangle_{g}\nonumber\\
&=&\sum_{\mu}\langle\langle G(\mathcal{V}_{1}\circ\mathcal{V}_{2}\circ\mathcal{V}_{3}\circ\gamma^{\mu})\{\mathcal{V}_{4}\circ\gamma_{\mu}\}\mathcal{V}_{5}\cdot\cdot\cdot\mathcal{V}_{k}\rangle\rangle_{g}\nonumber\\
&=&\sum_{\mu}\langle\langle G(\mathcal{V}_{1}\circ\gamma^{\mu})\{\mathcal{V}_{2}\circ\mathcal{V}_{3}\circ\mathcal{V}_{4}\circ\gamma_{\mu}\}\mathcal{V}_{5}\cdot\cdot\cdot\mathcal{V}_{k}\rangle\rangle_{g}\label{Observation2}
\eea
for any vector fields $\mathcal{V}_{i}$ ($1\leq i\leq k$).\par
We recollect the following Virasoro operators defined in \cite{LT}
\ben
L_{-1}:&=&\sum_{m,\alpha}\tilde{t}_{m}^{\alpha}\frac{\partial}{\partial t_{m-1}^{\alpha}}+\frac{1}{2\lambda^2}\sum_{\alpha,\beta}\eta_{\alpha\beta}t_{0}^{\alpha}t_{0}^{\beta},
\\
L_{0}:&=&\sum_{m,\alpha}(m+b_{\alpha})\tilde{t}_{m}^{\alpha}\frac{\partial}{\partial t_{m}^{\alpha}}
+\sum_{m,\alpha,\beta}\mathcal{C}_{\alpha}^{\beta}\tilde{t}_{m}^{\alpha}\frac{\partial}{\partial t_{m-1}^{\beta}}
+\frac{1}{2\lambda^2}\sum_{\alpha,\beta}\mathcal{C}_{\alpha\beta}t_{0}^{\alpha}t_{0}^{\beta}\\
&&+\frac{1}{24}\bigg(\frac{3-d}{2}\chi(X)-\int_{X}c_{1}(X)\cup c_{d-1}(X)\bigg),
\een
and for $n\geq1$,
\ben
L_{n}:&=&\sum_{m,\alpha,\beta}\sum_{j=0}^{m+n}A_{\alpha}^{(j)}(m,n)(\mathcal{C}^j)_{\alpha}^{\beta}\tilde{t}_{m}^{\alpha}\frac{\partial}{\partial t_{m+n-j}^{\beta}}\\
&&+\frac{\lambda^2}{2}\sum_{\alpha,\beta,\gamma}\sum_{j=0}^{n-1}\sum_{k=0}^{n-j-1}B_{\alpha}^{(j)}(k,n)(\mathcal{C}^j)_{\alpha}^{\beta}\eta^{\alpha\gamma}
\frac{\partial}{\partial t_{k}^{\gamma}}\frac{\partial}{\partial t_{n-k-1-j}^{\beta}}\\
&&+\frac{1}{2\lambda^2}\sum_{\alpha,\beta}(\mathcal{C}^{n+1})_{\alpha\beta}t_{0}^{\alpha}t_{0}^{\beta},
\een
where
\ben
&&A_{\alpha}^{(j)}(m,n):=\frac{\Gamma(b_{\alpha}+m+n+1)}{\Gamma(b_{\alpha}+m)}\sum_{m\leq l_{1}<l_{2}<\cdot\cdot\cdot<l_{j}\leq m+n}\bigg(\prod_{i=1}^{j}\frac{1}{b_{\alpha}+l_{i}}\bigg),\\
&&B_{\alpha}^{(j)}(m,n):=\frac{\Gamma(m+2-b_{\alpha})\Gamma(n-m+b_{\alpha})}{\Gamma(1-b_{\alpha})\Gamma(b_{\alpha})}\sum_{-m-1\leq l_{1}<l_{2}<\cdot\cdot\cdot<l_{j}\leq n-m-1}\bigg(\prod_{i=1}^{j}\frac{1}{b_{\alpha}+l_{i}}\bigg).
\een
These operators satisfy the following bracket relation
\be\label{Virasoro relation}
[L_{m},L_{n}]=(m-n)L_{m+n},
\ee
for $m,n\geq -1$.
The Virasoro conjecture is
\begin{cjt}\label{Virasoro conjecture}
$L_{n}Z\equiv0$ (called the $L_{n}$-constraint) for all $n\geq-1$.
\end{cjt}
It is well known that the $L_{-1}$-constraint and the $L_{0}$-constraint hold for all manifolds. In fact, the $L_{-1}$-constraint is equivalent to the string equation, and $L_{0}Z=0$ is true by Hori \cite{Hori}.
 %Then combined with \eqref{Virasoro relation}, Conjecture \ref{Virasoro conjecture} is reduced to prove that $L_{n}Z=0$ for $n=1,2$.
Assume that
\ben
L_{n}Z(T;\lambda)=\bigg\{\sum_{g\geq0}\Omega_{g,n}\lambda^{2g-2}\bigg\}Z(T;\lambda).
\een
Then the $L_{n}$-constraint is equivalent to $\Omega_{g,n}=0$ for all $g\geq0$.
As in \cite{LT}, equation $\Omega_{g,n}=0$ is called the genus-$g$ $L_{n}$-constraint, and the so called  genus-$g$ Virasoro conjecture predicts that the genus-$g$ $L_{n}$-constraint is true for all $n\geq-1$.\par
We will concentrate on the study of the genus-1 Virasoro conjecture. As mentioned in Introduction, it suffices to consider some simple equation on
the small phase space. Therefore, in the rest of this paper, everything will be considered  on the small phase space. Notice that all the above definitions and results hold when restricted on the small phase space. The
notation $\langle\langle\cdot\cdot\cdot\rangle\rangle_{g}$ is used again to denote some correlation function on the small phase space which is written as $\langle\langle\cdot\cdot\cdot\rangle\rangle_{g,s}$ in \cite{L1}.
 We start with Euler vector field on the small phase space, which is  defined by
\ben
E:=c_{1}(X)+\sum_{\alpha}(b_{1}+1-b_{\alpha})t^\alpha\gamma_{\alpha}.
\een
It is showed in \cite{LT} that the following quasi-homogeneity equation holds
\ben
\langle\langle E\rangle\rangle_{g}=(3-d)(1-g)F_{g}+\frac{1}{2}\delta_{g,0}\sum_{\alpha,\beta}\mathcal{C}_{\alpha\beta}t^{\alpha}t^{\beta}
-\frac{1}{24}\delta_{g,1}\int_{X}c_{1}(X)\cup c_{d-1}(X).
\een
And its derivatives are
\bea\label{quasihom}
&&\langle\langle E\upsilon_{1}\cdot\cdot\cdot\upsilon_{k}\rangle\rangle_{g}\nonumber\\
&=&\sum_{i=1}^{k}\langle\langle\upsilon_{1}\cdot\cdot\cdot G(\upsilon_{i})\cdot\cdot\cdot\upsilon_{k}\rangle\rangle_{g}-(2g+k-2)(b_{1}+1)\langle\langle\upsilon_{1}\cdot\cdot\cdot\upsilon_{k}\rangle\rangle_{g}\nonumber\\
&&+\delta_{g,0}\nabla^{k}_{\upsilon_{1},\cdot\cdot\cdot,\upsilon_{k}}\bigg(\frac{1}{2}\sum_{\alpha,\beta}\mathcal{C}_{\alpha\beta}t^{\alpha}t^{\beta}\bigg)
\eea
for any vector fields $\upsilon_{i}$ ($1\leq i\leq k$) on the small phase space.
In particular, we have
\be\label{4pointeq1}
\sum_{\alpha}\langle\langle E\upsilon_{1}\upsilon_{2}\gamma^\alpha\rangle\rangle_{0}\gamma_{\alpha}=G(\upsilon_{1})\circ\upsilon_{2}+\upsilon_{1}\circ G(\upsilon_{2})
-G(\upsilon_{1}\circ\upsilon_{2})-b_{1}\upsilon_{1}\circ\upsilon_{2}.
\ee
It is shown in \cite{L2} that
\be\label{derivative2}
\nabla_{\upsilon}E=-G(\upsilon)+(b_{1}+1)\upsilon.
\ee
More generally, we have
\begin{lem}
For all $k\geq0$ and any vector field $\upsilon$ on the small phase space, let $E^k$ be the $k^{th}$ quantum power of $E$, we have
\be\label{derivative3}
\nabla_{\upsilon}E^k
=\sum_{i=1}^{k}G(E^{i-1})\circ \upsilon\circ E^{k-i}-\sum_{i=1}^{k}G(\upsilon\circ E^{k-i})\circ E^{i-1}+k\upsilon\circ E^{k-1}.
\ee
\end{lem}
\begin{proof}
We will prove this lemma by induction on $k$. By the definition \eqref{cderivative} and $E^0=\gamma_{1}$, it holds for $k=0$ and also  for $k=1$ due to formula \eqref{derivative2}.
Assume equation \eqref{derivative3} holds for $k\leq n$. For $k=n+1$, by equations \eqref{derivativeqp} and \eqref{4pointeq1}, we have
\ben
&&\nabla_{\upsilon}E^{n+1}\\
&=&\nabla_{\upsilon}(E\circ E^{n})\\
&=&(\nabla_{\upsilon}E)\circ E^n+E\circ(\nabla_{\upsilon}E^n)+\sum_{\alpha}\langle\langle EE^n\upsilon\gamma^\alpha\rangle\rangle_{0}\gamma_{\alpha}\\
&=&-G(\upsilon)\circ E^n+(b_{1}+1)\upsilon\circ E^n+\sum_{i=1}^nG(E^{i-1})\circ\upsilon\circ E^{n-i+1}\\
&&-\sum_{i=1}^{n}G(\upsilon\circ E^{n-i})\circ E^i+
n\upsilon\circ E^n+G(E^n)\circ\upsilon+G(\upsilon)\circ E^{n}\\
&&-G(E^n\circ\upsilon)-b_{1}E^n\circ\upsilon\\
&=&\sum_{i=1}^{n+1}G(E^{i-1})\circ\upsilon\circ E^{n+1-i}-\sum_{i=1}^{n+1}G(\upsilon\circ E^{n+1-i})\circ E^{i-1}+(n+1)\upsilon\circ E^{n}.
\een
The prooof is completed.
\end{proof}
Then by equations \eqref{bracketrelation1} and \eqref{derivative3}, we have
\begin{cor}[\cite{DZ1,HM,L1}]\label{Virasorotype1}For $k,m\geq0$,
\ben
[E^{k},E^m]=(m-k)E^{m+k-1}.
\een
\end{cor}
Secondly, functions $\Phi_{k}$ are defined in terms of genus-0 data in \cite{L1}. They are
\bea
\Phi_{0}&=&0,\nonumber\\
\Phi_{1}&=&-\frac{1}{24}\int_{X}c_{1}(X)\cup c_{d-1}(X),\nonumber\\
\Phi_{k}&=&-\frac{1}{24}\sum\limits_{m=0}^{k-1}\sum\limits_{\alpha,\beta,\sigma}b_{\alpha}\langle\langle\gamma_{1}E^{m}\gamma^{\alpha}\rangle\rangle_{0}
\langle\langle\gamma_{\alpha}E^{k-1-m}\gamma^{\beta}\rangle\rangle_{0}\langle\langle\gamma_{\beta}\gamma_{\sigma}\gamma^{\sigma}\rangle\rangle_{0}\nonumber\\
&&-\frac{1}{4}\sum\limits_{m=0}^{k-1}\sum\limits_{\alpha,\beta}b_{\alpha}b_{\beta}\langle\langle\gamma_{\alpha}E^{m}\gamma^{\beta}\rangle\rangle_{0}
\langle\langle\gamma_{\beta}E^{k-1-m}\gamma^{\alpha}\rangle\rangle_{0}\nonumber\\
&&+\frac{k}{12}\sum\limits_{\sigma}\langle\langle\gamma_{\sigma}E^{k-1}\gamma^{\sigma}\rangle\rangle_{0},\label{phiformula}
\eea
for $k\geq2$.
It follows from the definitions of quantum product and $G$ that
\bea
24\Phi_{k}
&=&-\sum_{i=0}^{k-1}\langle\langle G(E^i)\Delta E^{k-i-1}\rangle\rangle_{0}-k\sum_{\mu}\langle\langle E^{k-1}\gamma_{\mu}\gamma^{\mu}\rangle\rangle_{0}\nonumber\\
&&+6\sum_{i=0}^{k-1}\sum_{\mu}\langle\langle G(E^i\circ\gamma^{\mu})G(\gamma_{\mu})E^{k-i-1}\rangle\rangle_{0},\label{phialternative}
\eea
where $\Delta=\sum_{\alpha}\gamma^{\alpha}\circ\gamma_{\alpha}$ and $k\geq2$.
\begin{rem}
It is easy to check that the expression \eqref{phiformula} holds for $k=0$. It also holds for $k=1$ by the string equation
\eqref{stringequ1} and the following equality \cite{Bori} derived by Borisov
\ben
\sum_{\alpha}b_{\alpha}(1-b_{\alpha})-\frac{b_{1}+1}{6}\chi(X)=-\frac{1}{6}\int_{X}c_{1}(X)\cup c_{d-1}(X).
\een
The same argument is applied in the proof of $\gamma_{1}\Phi_{2}=2\Phi_{1}$ in \cite{L1}. Therefore,  we adopt
equation \eqref{phiformula} or \eqref{phialternative} as the definition of $\Phi_{k}$  for all $k\geq0$. And notice that
it is shown in \cite{L1} that $\langle\langle E^{k}\rangle\rangle_{1}=\Phi_{k}$ for $k=0,1$.
\end{rem}
One important result is
\begin{thm}[\cite{L1}]
For any manifold $X$, the genus-1 Virasoro conjecture holds if and only if $\langle\langle E^{2}\rangle\rangle_{1}=\Phi_{2}$.
% where
%\ben
%\Phi_{2}=-\frac{1}{24}\sum\limits_{\alpha}\langle\langle EE\gamma_{\alpha}\gamma^{\alpha}\rangle\rangle_{0}
%+\frac{1}{2}\sum\limits_{\alpha}\left( b_{\alpha}(1-b_{\alpha})-\frac{b_{1}+1}{6}\right)\langle\langle\gamma_{\alpha}\gamma^{\alpha}\rangle\rangle_{0}.
%\een
\end{thm}
It is shown in \cite{L11} that for any given $k\geq2$, the genus-1 Virasoro conjecture holds if and only if $\langle\langle E^{k}\rangle\rangle_{1}=\Phi_{k}$.
Due to the equation $E(\langle\langle E^{2}\rangle\rangle_{1}-\Phi_{2})=\langle\langle E^{2}\rangle\rangle_{1}-\Phi_{2}$ proved in \cite{L1}, this conjecture is reduced to prove that $\upsilon(\langle\langle E^{2}\rangle\rangle_{1}-\Phi_{2})=0$ for any $\upsilon\in H^{*}(X,\mathbb{C})$. It is verified in \cite{L1} that $\gamma_{1}(\langle\langle E^{2}\rangle\rangle_{1}-\Phi_{2})=0$. And one more evidence was discovered  in the following
\begin{thm}[\cite{L2}]
For all smooth projective varieties,
we have
\ben
\Delta(\langle\langle E^2\rangle\rangle_{1}-\Phi_{2})=0.
\een
\end{thm}
We will find more evidences in Section 4.
By the fact $\gamma_{1}(\langle\langle E^{k}\rangle\rangle_{1}-\Phi_{k})=k(\langle\langle E^{k-1}\rangle\rangle_{1}-\Phi_{k-1})$ proved in Lemma 6.3 of \cite{L1}, it is easy to show that for any given $k\geq2$, the genus-1 Virasoro conjecture holds
 if and only if $\upsilon(\langle\langle E^{k}\rangle\rangle_{1}-\Phi_{k})=0$ for any $\upsilon\in H^{*}(X,\mathbb{C})$.
\begin{rem}
By the above argument, to prove the
genus-1 Virasoro conjecture, it is also reduced to verify that for any given $k,l\geq2$, $\upsilon_{l}\upsilon_{l-1}\cdot\cdot\cdot\upsilon_{1}(\langle\langle E^{k}\rangle\rangle_{1}-\Phi_{k})=0$ holds for any $\upsilon_{1},\cdot\cdot\cdot,\upsilon_{l}\in H^{*}(X,\mathbb{C})$. Notice that the computation in Section 3
only yields the first derivative of $\{\langle\langle E^{k}\rangle\rangle_{1}-\Phi_{k}\}$, hence to derive its l-$th$ derivative, one may take derivatives of our results (cf. Remark \ref{addedrem}) in section 3 or by applying the pull-back of Getzler's genus one relation to quantum powers of Euler vector field.
\end{rem}
\subsection{Universal equations for Gromov-Witten invariants}
In this subsection, we breifly recall some universal equations for primary Gromov-Witten invariants.
Firstly,  the string equation on the small phase space shows that
\begin{lem}[\cite{L1}]
\bea
&&\langle\langle\gamma_{1}\gamma_{\alpha}\gamma_{\beta}\rangle\rangle_{0}=\eta_{\alpha\beta},\label{stringequ1}\\
&&\langle\langle\gamma_{1}\gamma_{\alpha_{1}}\cdot\cdot\cdot\gamma_{\alpha_{k}}\rangle\rangle_{0}=0 \mbox{      if $k\geq3$},\label{stringequ2}\\
&&\langle\langle\gamma_{1}\gamma_{\alpha_{1}}\cdot\cdot\cdot\gamma_{\alpha_{k}}\rangle\rangle_{g}=0 \mbox{  if $g\geq1$ and  $k\geq0$}.\label{stringequ3}
\eea
\end{lem}
Next, the important universal equation for genus-0 primary Gromov-Witten invariants is the WDVV equation \eqref{WDVV1}, which is implicitly used for computation throughout the paper. Its first derivative has the form
\bea\label{WDVV2}
&&\langle\langle\{\upsilon_{1}\circ\upsilon_{2}\}\upsilon_{3}\upsilon_{4}\upsilon_{5}\rangle\rangle_{0}\nonumber\\
&=&\langle\langle\{\upsilon_{1}\circ\upsilon_{3}\}\upsilon_{2}\upsilon_{4}\upsilon_{5}\rangle\rangle_{0}
+\langle\langle\{\upsilon_{2}\circ\upsilon_{5}\}\upsilon_{1}\upsilon_{3}\upsilon_{4}\rangle\rangle_{0}
-\langle\langle\{\upsilon_{3}\circ\upsilon_{5}\}\upsilon_{1}\upsilon_{2}\upsilon_{4}\rangle\rangle_{0}.
\eea
It is well known that equation \eqref{WDVV2} is the topological recursion relation derived from Keel relation \cite{Keel} on the moduli space of stable curves $\overline{\mathcal{M}}_{0,5}$. And it implies
\begin{lem}[\cite{L1}]\label{WDVV3}
For any vector field $\upsilon$ on the small phase space, let $\upsilon^{k}$ be the $k^{th}$ quantum power of $\upsilon$. Then for any
$\alpha,\beta$ and $\mu$, we have
\ben
\langle\langle\upsilon^k\gamma_{\alpha}\gamma_{\beta}\gamma_{\mu}\rangle\rangle_{0}&=&
-\sum_{i=1}^{k-1}\langle\langle\upsilon^{k-i}(\gamma_{\alpha}\circ\gamma_{\beta}\circ \upsilon^{i-1})\upsilon\gamma_{\mu}\rangle\rangle_{0}\\
&&+\sum_{i=1}^{k}\langle\langle(\upsilon^{k-i}\circ\gamma_{\alpha})(\gamma_{\beta}\circ\upsilon^{i-1})\upsilon\gamma_{\mu}\rangle\rangle_{0}.
\een
\end{lem}
It follows easily from Keel relation on $\overline{\mathcal{M}}_{0,6}$ that
\bea
&&\langle\langle\{\upsilon_{1}\circ\upsilon_{2}\}\upsilon_{3}\upsilon_{4}\upsilon_{5}\upsilon_{6}\rangle\rangle_{0}\nonumber\\
&=&\sum_{\rho}\langle\langle\upsilon_{1}\upsilon_{3}\upsilon_{6}\gamma^{\rho}\rangle\rangle_{0}
\langle\langle\gamma_{\rho}\upsilon_{2}\upsilon_{4}\upsilon_{5}\rangle\rangle_{0}
+\sum_{\rho}\langle\langle\upsilon_{1}\upsilon_{3}\upsilon_{5}\gamma^{\rho}\rangle\rangle_{0}
\langle\langle\gamma_{\rho}\upsilon_{2}\upsilon_{4}\upsilon_{6}\rangle\rangle_{0}\nonumber\\
&&+\langle\langle\{\upsilon_{2}\circ\upsilon_{4}\}\upsilon_{1}\upsilon_{3}\upsilon_{5}\upsilon_{6}\rangle\rangle_{0}
+\langle\langle\{\upsilon_{1}\circ\upsilon_{3}\}\upsilon_{2}\upsilon_{4}\upsilon_{5}\upsilon_{6}\rangle\rangle_{0}\nonumber\\
&&-\sum_{\rho}\langle\langle\upsilon_{1}\upsilon_{2}\upsilon_{6}\gamma^{\rho}\rangle\rangle_{0}
\langle\langle\gamma_{\rho}\upsilon_{3}\upsilon_{4}\upsilon_{5}\rangle\rangle_{0}
-\sum_{\rho}\langle\langle\upsilon_{1}\upsilon_{2}\upsilon_{5}\gamma^{\rho}\rangle\rangle_{0}
\langle\langle\gamma_{\rho}\upsilon_{3}\upsilon_{4}\upsilon_{6}\rangle\rangle_{0}\nonumber
\\
&&-\langle\langle\{\upsilon_{3}\circ\upsilon_{4}\}\upsilon_{1}\upsilon_{2}\upsilon_{5}\upsilon_{6}\rangle\rangle_{0},\label{WDVVd2}
\eea
which is also  derived from the second derivative of the WDVV equation \eqref{WDVV1}.
We will use the following lemma for the computation in Section 3.
\begin{lem}\label{WDVV4}
For any vector field $\upsilon$ on the small phase space, let $\upsilon^{k}$ be the $k^{th}$ quantum power of $\upsilon$. Then for any
$\alpha,\beta$, $\mu$, and $\sigma$, we have
\ben
&&\langle\langle\upsilon^k\gamma_{\alpha}\gamma_{\beta}\gamma_{\mu}\gamma_{\sigma}\rangle\rangle_{0}\\
&=&-\sum_{i=1}^{k-1}\sum_{\rho}\langle\langle\upsilon^{k-i}\gamma_{\sigma}\upsilon\gamma^{\rho}\rangle\rangle_{0}
\langle\langle \gamma_{\rho}\gamma_{\alpha}\{\gamma_{\beta}\circ\upsilon^{i-1}\}\gamma_{\mu}\rangle\rangle_{0}\\
&&-\sum_{i=1}^{k-1}\langle\langle\upsilon^{k-i}
\{\gamma_{\alpha}\circ\gamma_{\beta}\circ\upsilon^{i-1}\}\upsilon\gamma_{\mu}\gamma_{\sigma}\rangle\rangle_{0}\\
&&-\sum_{i=1}^{k-1}\sum_{\rho}\langle\langle\upsilon^{k-i}\gamma_{\mu}\upsilon\gamma^{\rho}\rangle\rangle_{0}
\langle\langle \gamma_{\rho}\gamma_{\alpha}\{\gamma_{\beta}\circ\upsilon^{i-1}\}\gamma_{\sigma}\rangle\rangle_{0}\\
&&+\sum_{i=1}^{k-1}\sum_{\rho}\langle\langle\upsilon^{k-i}\gamma_{\alpha}\gamma_{\sigma}\gamma^{\rho}\rangle\rangle_{0}
\langle\langle \gamma_{\rho}\{\gamma_{\beta}\circ\upsilon^{i-1}\}\upsilon\gamma_{\mu}\rangle\rangle_{0}\\
&&+\sum_{i=1}^{k-1}\sum_{\rho}\langle\langle\upsilon^{k-i}\gamma_{\alpha}\gamma_{\mu}\gamma^{\rho}\rangle\rangle_{0}
\langle\langle \gamma_{\rho}\{\gamma_{\beta}\circ\upsilon^{i-1}\}\upsilon\gamma_{\sigma}\rangle\rangle_{0}\\
&&+\sum_{i=1}^{k}\langle\langle\{\upsilon^{k-i}
\circ\gamma_{\alpha}\}\{\gamma_{\beta}\circ\upsilon^{i-1}\}\upsilon\gamma_{\mu}\gamma_{\sigma}\rangle\rangle_{0}
\een
\end{lem}
\begin{proof}
As in the proof of Lemma 4.2 in \cite{L1}, the proof follows by choosing $\upsilon_{1}=\upsilon^{k-1}$, $\upsilon_{2}=\upsilon$,  $\upsilon_{3}=\gamma_{\alpha}$,
 $\upsilon_{4}=\gamma_{\beta}$, $\upsilon_{5}=\gamma_{\mu}$, and $\upsilon_{6}=\gamma_{\sigma}$ in equation \eqref{WDVVd2} and repeatedly using
 the resulting formula.
\end{proof}
The essential universal equation for genus-1 primary Gromov-Witten invariants
is derived from Getzler's genus one relation \cite{Get2}. Adopting the presentation in \cite{L1}, it has the following form
(we call it Getzler's genus one relation again)
\be\label{Getzlertrr}
\textbf{G}(\upsilon_{1},\upsilon_{2},\upsilon_{3},\upsilon_{4})=G_{0}(\upsilon_{1},\upsilon_{2},\upsilon_{3},\upsilon_{4})
+G_{1}(\upsilon_{1},\upsilon_{2},\upsilon_{3},\upsilon_{4})=0,
\ee
where
\ben
&&G_{0}(\upsilon_{1},\upsilon_{2},\upsilon_{3},\upsilon_{4})\\
&=&\sum\limits_{h\in S_{4}}\sum\limits_{\alpha,\beta}\bigg\{\frac{1}{6}\langle\langle \upsilon_{h{(1)}}\upsilon_{h{(2)}}\upsilon_{h{(3)}}\gamma^{\alpha}\rangle\rangle_{0}
\langle\langle \gamma_{\alpha}\upsilon_{h{(4)}}\gamma_{\beta}\gamma^{\beta}\rangle\rangle_{0}\\
&&+\frac{1}{24}\langle\langle \upsilon_{h{(1)}}\upsilon_{h{(2)}}\upsilon_{h{(3)}}\upsilon_{h{(4)}}\gamma^{\alpha}\rangle\rangle_{0}\langle\langle \gamma_{\alpha}\gamma_{\beta}\gamma^{\beta}\rangle\rangle_{0}\\
&&-\frac{1}{4}\langle\langle \upsilon_{h{(1)}}\upsilon_{h{(2)}}\gamma^{\alpha}\gamma^{\beta}\rangle\rangle_{0}
\langle\langle \gamma_{\alpha}\gamma_{\beta}\upsilon_{h{(3)}}\upsilon_{h{(4)}}\rangle\rangle_{0}\bigg\},
\een
and
\ben
&&G_{1}(\upsilon_{1},\upsilon_{2},\upsilon_{3},\upsilon_{4})\\
&=&3\sum\limits_{h\in S_{4}}\langle\langle\{\upsilon_{h{(1)}}\circ \upsilon_{h{(2)}}\}\{\upsilon_{h{(3)}}\circ \upsilon_{h{(4)}}\}\rangle\rangle_{1}\\
&&-4\sum\limits_{h\in S_{4}}\langle\langle\{\upsilon_{h{(1)}}\circ \upsilon_{h{(2)}}\circ \upsilon_{h(3)}\} \upsilon_{h{(4)}}\rangle\rangle_{1}\\
&&-\sum\limits_{h\in S_{4}}\sum\limits_{\alpha}\langle\langle\{\upsilon_{h{(1)}}\circ \upsilon_{h{(2)}}\}\upsilon_{h{(3)}}\upsilon_{h{(4)}}\gamma^{\alpha}\rangle\rangle_{0}\langle\langle\gamma_{\alpha}\rangle\rangle_{1}\\
&&+2\sum\limits_{h\in S_{4}}\sum\limits_{\alpha}\langle\langle \upsilon_{h{(1)}}\upsilon_{h{(2)}}\upsilon_{h{(3)}}\gamma^{\alpha}\rangle\rangle_{0}
\langle\langle\{\gamma_{\alpha}\circ \upsilon_{h{(4)}}\}\rangle\rangle_{1},
\een
for any vector fields $\upsilon_{1},\upsilon_{2},\upsilon_{3},\upsilon_{4}$ on the small phase space.
\begin{rem}By string equations \eqref{stringequ2} and \eqref{stringequ3}, we have
$G_{1}(\gamma_{1},\upsilon_{1},\upsilon_{2},\upsilon_{3})\equiv0$ and
$G_{0}(\gamma_{1},\upsilon_{1},\upsilon_{2},\upsilon_{3})\equiv0$ for any vector fields
$\upsilon_{1},\upsilon_{2},\upsilon_{3}$ on the small phase space.
\end{rem}
For simplicity, we will use the following notational conventions: Repeated indices are summed over their entire meaningful range unless otherwise indicated.
\section{Universal equations from Getzler's genus one relation}
In this section, we will derive an explicit formula for the first derivative of $\Phi_{k}$, and then obtain an explicit universal
equation, i.e., Theorem \ref{Getzler-1E}, by computing Getzler's genus one relation \eqref{Getzlertrr} when one of vector fields is substituted by quantum power of Euler
vector field. With this theorem, we can  derive other explicit universal equations for the cases involving more quantum powers of Euler vector field. These explicit universal equations from Getzler's genus one relation will be called Getzler equations.
To see how much information of Getzler equations can be applied in the genus-1 Virasoro conjecture, we always reduce higher $k$-point correlation functions involving some quantum power of Euler vector field to lower $l$-point correlation functions, and use the first derivative of $\Phi_{k}$ for possible simplification.

\subsection{An explicit formula for the first derivative of $\Phi_{k}$}
 The following lemma is  useful for reducing 4-point functions to 3-point functions.
\begin{lem}\label{4point-3point}
For all $m\geq0$,
\ben
&&\langle\langle E^{m}\gamma_{\alpha}\gamma_{\beta}\gamma_{\mu}\rangle\rangle_{0}\\
&=&-\sum_{i=1}^{m}\langle\langle G(E^{m-i})\{E^{i-1}\circ \gamma_{\alpha}\circ\gamma_{\beta}\}\gamma_{\mu}\rangle\rangle_{0}
-\sum_{i=1}^{m}\langle\langle G(E^{i-1}\circ\gamma_{\alpha}\circ\gamma_{\beta}) E^{m-i}\gamma_{\mu}\rangle\rangle_{0}\\
&&+\sum_{i=1}^{m}\langle\langle G(E^{m-i}\circ\gamma_{\alpha})\{E^{i-1}\circ\gamma_{\beta}\}\gamma_{\mu}\rangle\rangle_{0}
+\sum_{i=1}^{m} \langle\langle G(E^{i-1}\circ\gamma_{\beta})\{E^{m-i}\circ\gamma_{\alpha}\}\gamma_{\mu}\rangle\rangle_{0}.
\een
\end{lem}
\begin{proof}
Using Lemma \ref{WDVV3}, equations \eqref{quasihom} and \eqref{WDVV1}, we have
\ben
&&\langle\langle E^{m}\gamma_{\alpha}\gamma_{\beta}\gamma_{\mu}\rangle\rangle_{0}\\
&=&-\sum_{i=1}^{m-1}\langle\langle G(E^{m-i})\{E^{i-1}\circ \gamma_{\alpha}\circ\gamma_{\beta}\}\gamma_{\mu}\rangle\rangle_{0}
-\sum_{i=1}^{m-1}\langle\langle G(E^{i-1}\circ\gamma_{\alpha}\circ\gamma_{\beta}) E^{m-i}\gamma_{\mu}\rangle\rangle_{0}\\
&&+\sum_{i=1}^{m}\langle\langle G(E^{m-i}\circ\gamma_{\alpha})\{E^{i-1}\circ\gamma_{\beta}\}\gamma_{\mu}\rangle\rangle_{0}
+\sum_{i=1}^{m} \langle\langle G(E^{i-1}\circ\gamma_{\beta})\{E^{m-i}\circ\gamma_{\alpha}\}\gamma_{\mu}\rangle\rangle_{0}\\
&&-\sum_{i=1}^{m-1}\langle\langle E^{m-1}\{\gamma_{\alpha}\circ\gamma_{\beta}\}G(\gamma_{\mu})\rangle\rangle_{0}
+\sum_{i=1}^{m}\langle\langle E^{m-1}\{\gamma_{\alpha}\circ\gamma_{\beta}\}G(\gamma_{\mu})\rangle\rangle_{0}
\\
&&+(b_{1}+1)\bigg\{\sum_{i=1}^{m-1}\langle\langle E^{m-1}\{\gamma_{\alpha}\circ\gamma_{\beta}\}\gamma_{\mu}\rangle\rangle_{0}
-\sum_{i=1}^{m} \langle\langle E^{m-1}\{\gamma_{\alpha}\circ\gamma_{\beta}\}\gamma_{\mu}\rangle\rangle_{0}\bigg\}.
\een
The proof follows by using the following equation \eqref{simplication1}.
\end{proof}
In particular, it follows from Lemma \ref{4point-3point} that
\bea
&&\langle\langle E^{m}\gamma_{\alpha}\gamma_{\beta}\gamma^{\mu}\rangle\rangle_{0}\gamma_{\mu}\nonumber\\
&=&-\sum_{i=1}^{m} G(E^{m-i})\circ E^{i-1}\circ\gamma_{\alpha}\circ\gamma_{\beta}
-\sum_{i=1}^{m} G(E^{i-1}\circ\gamma_{\alpha}\circ\gamma_{\beta})\circ E^{m-i}\nonumber\\
&&+\sum_{i=1}^{m} G(E^{m-i}\circ\gamma_{\alpha})\circ E^{i-1}\circ\gamma_{\beta}
+\sum_{i=1}^{m} G(E^{i-1}\circ\gamma_{\beta})\circ E^{m-i}\circ\gamma_{\alpha}.\label{4point-3point1}
\eea
The following lemma is useful for simplification.
\begin{lem}\label{simplication}
For any vector field $\upsilon_{i}$ ($1\leq i\leq4$) on the small phase space, we have
\be
\langle\langle G(\upsilon_{1}\circ\upsilon_{2})\upsilon_{3}\upsilon_{4}\rangle\rangle_{0}
=\langle\langle\upsilon_{1}\upsilon_{2}\{\upsilon_{3}\circ\upsilon_{4}\}\rangle\rangle_{0}
-\langle\langle\upsilon_{1}\upsilon_{2}G(\upsilon_{3}\circ\upsilon_{4})\rangle\rangle_{0},\label{simplication1}\ee
\be
G(\upsilon_{1}\circ\gamma^{\mu})\circ\gamma_{\mu}=G(\upsilon_{1}\circ\gamma_{\mu})\circ\gamma^{\mu}
=\frac{1}{2}\Delta\circ\upsilon_{1},\label{simplication2}\ee
\be\langle\langle G(\upsilon_{1}\circ\upsilon_{2}\circ\upsilon_{3}\circ\gamma^{\mu})G(\gamma_{\mu})\upsilon_{4}\rangle\rangle_{0}
=\langle\langle \{G(\upsilon_{4}\circ\gamma^{\mu})\circ G(\gamma_{\mu})\circ\upsilon_{1}\}\upsilon_{2}\upsilon_{3}\rangle\rangle_{0},\label{simplication3}\ee
\be
\langle\langle\{G(\upsilon_{1}\circ\gamma^{\mu})\circ G(\upsilon_{2}\circ\gamma_{\mu})\}\upsilon_{3}\upsilon_{4}\rangle\rangle_{0}
=\langle\langle\{G(\upsilon_{1}\circ\upsilon_{2}\circ\gamma^{\mu})\circ G(\gamma_{\mu})\}\upsilon_{3}\upsilon_{4}\rangle\rangle_{0}.\label{simplication4}
\ee
\end{lem}

\begin{proof}
As $\eta_{\alpha\beta}\neq0$ requires that $b_{\alpha}=1-b_{\beta}$,
\ben
\langle\langle G(\upsilon_{1}\circ\upsilon_{2})\upsilon_{3}\upsilon_{4}\rangle\rangle_{0}&=&b_{\alpha}\langle\langle\upsilon_{1}\upsilon_{2}\gamma^{\alpha}\rangle\rangle\langle\langle \gamma_{\alpha}\upsilon_{3}\upsilon_{4}\rangle\rangle_{0}\\
&=&b_{\alpha}\eta_{\alpha\beta}\langle\langle\upsilon_{1}\upsilon_{2}\gamma^{\alpha}\rangle\rangle_{0}\langle\langle \gamma^{\beta}\upsilon_{3}\upsilon_{4}\rangle\rangle_{0}\\
&=&(1-b_{\beta})\eta_{\alpha\beta}\langle\langle\upsilon_{1}\upsilon_{2}\gamma^{\alpha}\rangle\rangle_{0}\langle\langle \gamma^{\beta}\upsilon_{3}\upsilon_{4}\rangle\rangle_{0}\\
&=&(1-b_{\beta})\langle\langle\upsilon_{1}\upsilon_{2}\gamma_{\beta}\rangle\rangle_{0}\langle\langle \gamma^{\beta}\upsilon_{3}\upsilon_{4}\rangle\rangle_{0}.
\een
The equation \eqref{simplication1} follows. By the same argument, the equation \eqref{simplication2} also holds.
The equation \eqref{simplication3} follows from equations \eqref{simplication1} and \eqref{simplication2}. The last equation
can be proved as follows
\ben
&&\langle\langle\{G(\upsilon_{1}\circ\gamma^{\mu})\circ G(\upsilon_{2}\circ\gamma_{\mu})\}\upsilon_{3}\upsilon_{4}\rangle\rangle_{0}\\
&=&b_{\alpha}b_{\beta}\langle\langle\upsilon_{1}\gamma^{\mu}\gamma^{\alpha}\rangle\rangle_{0}
\langle\langle\upsilon_{2}\gamma_{\mu}\gamma^{\beta}\rangle\rangle_{0}\langle\langle\{\gamma_{\alpha}\circ \gamma_{\beta}\}\upsilon_{3}\upsilon_{4}\rangle\rangle_{0}\\
&=&b_{\alpha}b_{\beta}
\langle\langle\upsilon_{2}\{\upsilon_{1}\circ\gamma^{\alpha}\}\gamma^{\beta}\rangle\rangle_{0}\langle\langle\{\gamma_{\alpha}\circ \gamma_{\beta}\}\upsilon_{3}\upsilon_{4}\rangle\rangle_{0}\\
&=&\langle\langle\{G(\upsilon_{1}\circ\upsilon_{2}\circ\gamma^{\mu})\circ G(\gamma_{\mu})\}\upsilon_{3}\upsilon_{4}\rangle\rangle_{0}.\een
\end{proof}
Now, an explicit formula for the first derivative of $\Phi_{k}$ is presented below.
\begin{thm}\label{explicit-dphi}For all $k\geq0$, and any $\alpha$,
\ben
24\gamma_{\alpha}\Phi_{k}
&=&-\sum_{i=1}^{k-1}\langle\langle G(E^i)\{E^{k-i-1}\circ\gamma_{\alpha}\}\gamma_{\mu}\gamma^{\mu}\rangle\rangle_{0}\\
&&-\sum_{i=0}^{k-2}(k-i-1)\langle\langle \{G(E^i)\circ E^{k-i-2}\}\Delta\gamma_{\alpha}\rangle\rangle_{0}\\
&&+\sum_{i=0}^{k-2}(-k+2i+2)\langle\langle G(\Delta\circ E^{k-i-2})E^i\gamma_{\alpha}\rangle\rangle_{0}\\
&&-\sum_{i=1}^{k-1}\sum_{j=1}^{i}\langle\langle G(E^{k-i-1}\circ\Delta) G(E^{i-j}\circ\gamma_{\alpha}) E^{j-1}\rangle\rangle_{0}\\
&&-\sum_{i=1}^{k-1}\sum_{j=1}^{i}\langle\langle \{G(E^{k-i-1})\circ \Delta\}G(E^{i-j}\circ\gamma_{\alpha}) E^{j-1}\rangle\rangle_{0}\\
&&+12\sum_{i=1}^{k-1}i\langle\langle \{G(E^{k-i-1}\circ \gamma^{\mu})\circ G(\gamma_{\mu})\}E^{i-1}\gamma_{\alpha}\rangle\rangle_{0}\\
&&-k(k-1)\langle\langle \Delta E^{k-2}\gamma_{\alpha}\rangle\rangle_{0}.
\een
\end{thm}

\begin{proof}It is trivial for $k=0,1$. It remains to prove this theorem when $k\geq2$.
According to  equation \eqref{phiformula}, using equations \eqref{derivative1} and \eqref{stringequ2}, we have
\bea
&&24\gamma_{\alpha}\Phi_{k}\nonumber\\
&=&-\sum_{i=0}^{k-1}\langle\langle G(\nabla_{\gamma_{\alpha}}E^i)E^{k-i-1}\Delta\rangle\rangle_{0}-\sum_{i=0}^{k-1}\langle\langle E^{k-i-1}G(E^i)\Delta\gamma_{\alpha}\rangle\rangle_{0}\nonumber\\
&&-\sum_{i=0}^{k-1}\langle\langle G(E^i)\{\nabla_{\gamma_{\alpha}}E^{k-i-1}\}\Delta\rangle\rangle_{0}-\sum_{i=0}^{k-1}\langle\langle \{G(E^i)\circ E^{k-i-1}\}\gamma_{\alpha}\gamma_{\sigma}\gamma^{\sigma}\rangle\rangle_{0}\nonumber\\
&&-12\sum_{i=0}^{k-1}b_{\mu}b_{\beta}\langle\langle E^i\gamma_{\alpha}\gamma_{\mu}\gamma^{\beta}\rangle\rangle_{0}\langle\langle\gamma_{\beta}E^{k-i-1}\gamma^{\mu}\rangle\rangle_{0}\nonumber\\
&&-12\sum_{i=0}^{k-1}b_{\mu}b_{\beta}\langle\langle \{\nabla_{\gamma_{\alpha}}E^i\}\gamma_{\mu}\gamma^{\beta}\rangle\rangle_{0}\langle\langle\gamma_{\beta}E^{k-i-1}\gamma^{\mu}\rangle\rangle_{0}\nonumber\\
&&+2k\langle\langle E^{k-1}\gamma_{\sigma}\gamma^{\sigma}\gamma_{\alpha}\rangle\rangle_{0}
+2k\langle\langle\{\nabla_{\gamma_{\alpha}}E^{k-1}\}\gamma_{\sigma}\gamma^{\sigma}\rangle\rangle_{0}.\label{phiderivative1}
\eea
By equation \eqref{WDVV2}, we have
\bea
&&\langle\langle \{G(E^i)\circ E^{k-i-1}\}\gamma_{\alpha}\gamma_{\sigma}\gamma^{\sigma}\rangle\rangle_{0}\nonumber\\
&=&\langle\langle \{E^{k-i-1}\circ\gamma_{\alpha}\}\gamma_{\sigma}\gamma^{\sigma}G(E^i)\rangle\rangle_{0}
+\langle\langle E^{k-i-1}\gamma_{\sigma}\gamma_{\alpha}\{G(E^i)\circ\gamma^{\sigma}\}\rangle\rangle_{0}\nonumber\\
&&-\langle\langle E^{k-i-1}\{\gamma_{\alpha}\circ\gamma^{\sigma}\}\gamma_{\sigma}G(E^i)\rangle\rangle_{0}.\label{phireduction1}
\eea
Notice that
\bea
&&b_{\mu}b_{\beta}\langle\langle E^i\gamma_{\alpha}\gamma_{\mu}\gamma^{\beta}\rangle\rangle_{0}\langle\langle\gamma_{\beta}E^{k-i-1}\gamma^{\mu}\rangle\rangle_{0}\nonumber\\
&=&b_{\mu}(1-b_{\beta})\langle\langle E^i\gamma_{\alpha}\gamma_{\mu}\gamma_{\beta}\rangle\rangle_{0}\langle\langle\gamma^{\beta}E^{k-i-1}\gamma^{\mu}\rangle\rangle_{0}\nonumber\\
&=&\langle\langle E^iG(\gamma_{\mu})\{E^{k-i-1}\circ\gamma^{\mu}\}\gamma_{\alpha}\rangle\rangle_{0}-\langle\langle E^iG(\gamma_{\mu})G(E^{k-i-1}\circ\gamma^{\mu})\gamma_{\alpha}\rangle\rangle_{0}\label{phireduction2}
\eea
%\ben
%&&b_{\mu}b_{\beta}\langle\langle E^i\gamma_{\alpha}\gamma_{\mu}\gamma^{\beta}\rangle\rangle_{0}\langle\langle\gamma_{\beta}E^{k-i-1}\gamma^{\mu}\rangle\rangle_{0}\\
%&=&\sum_{j=1}^{i-1}\langle\langle\{G(E^{i-j})\circ E^{j-1}\}\{G(E^{k-i-1}\circ\gamma^{\mu})\}\gamma_{\alpha}\rangle\rangle_{0}\\
%&&+\sum_{j=1}^{i-1}\langle\langle E^{i-j}G(G(E^{k-i-1}\circ\gamma^{\mu})\circ G(\gamma_{\mu})\circ E^{j-1})\gamma_{\alpha}\rangle\rangle_{0}\\
%&&-(b_{\alpha}-b_{1}-1)\langle\langle E^{i-1}\{G(E^{k-i-1}\circ\gamma^{\mu})\circ G(\gamma_{\mu})\}\gamma_{\alpha}\rangle\rangle_{0}\\
%&&-\sum_{j=1}^{i}\langle\langle E^{j-1}\{G(E^{k-i-1}\circ\gamma^{\mu})\circ G(G(\gamma_{\mu})\circ E^{i-j})\}\gamma_{\alpha}\rangle\rangle_{0}\\
%&&-\sum_{j=1}^{i}\langle\langle E^{i-j}\{G(G(E^{k-i-1}\circ\gamma^{\mu})\circ E^{j-1})\circ G(\gamma_{\mu})\}\gamma_{\alpha}\rangle\rangle_{0}\\
%&&-\frac{1}{2}\sum_{j=1}^{i-1}\langle\langle\{G(E^{i-j})\circ E^{k+j-i-2}\}\Delta\gamma_{\alpha}\rangle\rangle_{0}-\frac{1}{2}\sum_{j=1}^{i-1}\langle\langle\{E^{i-j}G(\Delta\circ E^{k+j-i-2})\gamma_{\alpha}\rangle\rangle_{0}\\
%&&+\frac{1}{2}(b_{\alpha}-b_{1}+i-1)\langle\langle\Delta E^{k-2}\gamma_{\alpha}\rangle\rangle_{0}\\
%&&-\sum_{j=1}^i\langle\langle E^{k-i+j-2}\{G(E^{i-j}\circ\gamma^{\mu})\circ G(\gamma_{\mu})\}\gamma_{\alpha}\rangle\rangle_{0}\\
%&&+\sum_{j=1}^i\langle\langle E^{i-j}\{G(E^{k-i+j-2}\circ\gamma^{\mu})\circ G(\gamma_{\mu})\}\gamma_{\alpha}\rangle\rangle_{0}
%\een
and
\bea
&&b_{\mu}b_{\beta}\langle\langle \{\nabla_{\gamma_{\alpha}}E^i\}\gamma_{\mu}\gamma^{\beta}\rangle\rangle_{0}\langle\langle\gamma_{\beta}E^{k-i-1}\gamma^{\mu}\rangle\rangle_{0}\nonumber\\
&=&[1-(1-b_{\mu})]\langle\langle G((\nabla_{\gamma_{\alpha}}E^i)\circ \gamma_{\mu})E^{k-i-1}\gamma^{\mu}\rangle\rangle_{0}\nonumber\\
&=&\langle\langle G((\nabla_{\gamma_{\alpha}}E^i)\circ \gamma_{\mu})E^{k-i-1}\gamma^{\mu}\rangle\rangle_{0}-\langle\langle G((\nabla_{\gamma_{\alpha}}E^i)\circ \gamma_{\mu})E^{k-i-1}G(\gamma^{\mu})\rangle\rangle_{0}.\label{phireduction3}
\eea
The proof follows by firstly substituting equations \eqref{phireduction1}, \eqref{phireduction2} and \eqref{phireduction3} into the equality \eqref{phiderivative1},
and secondly applying Lemma \ref{4point-3point} and formula \eqref{derivative3} to the resulting equality, and then using Lemma \ref{simplication} for
 simplification.
\end{proof}
Therefore, we have an alternative direct proof of
\begin{cor}[\cite{L1}]\label{Virasorotype2} For any smooth projective variety $X$, we have
\ben
E^k\Phi_{m}-E^m\Phi_{k}=(m-k)\Phi_{k+m-1}.
\een
\end{cor}
\begin{proof}
We only consider the cases for $k+m\geq2$, since other cases are trivial.
Using Theorem \ref{explicit-dphi}, Lemma \ref{4point-3point} and Lemma \ref{simplication}, we have
\ben
&&24E^m\Phi_{k}\\
&=&\sum_{i=1}^k\sum_{j=1}^m\langle\langle G(E^{i-1})G(E^{j-1})\{\Delta\circ E^{k+m-i-j}\}\rangle\rangle_{0}\\
&&-b_{1}\sum_{i=1}^{k+m-1}\langle\langle G(E^{i-1})\Delta E^{k+m-i-1}\rangle\rangle_{0}\\
&&+\sum_{i=1}^{k-1}\sum_{j=1}^{k+m-i-1}\langle\langle G(E^i)G(\Delta\circ E^{j-1})E^{k+m-i-j-1}\rangle\rangle_{0}\\
&&-\sum_{i=m}^{k+m-2}\sum_{j=1}^{k+m-i-1}\langle\langle G(E^i)G(\Delta\circ E^{j-1})E^{k+m-i-j-1}\rangle\rangle_{0}\\
&&-\sum_{i=1}^{k-1}(2k+m-2i-2)\langle\langle G(E^i)\Delta E^{k+m-i-2}\rangle\rangle_{0}\\
&&+\sum_{i=m}^{k+m-2}(2m+k-2i-2)\langle\langle G(E^i)\Delta E^{k+m-i-2}\rangle\rangle_{0}\\
&&+6\sum_{i=1}^{k-1}(k-i)\langle\langle G(E^{i-1}\circ\gamma^{\mu})G(\gamma_{\mu})E^{k+m-i-1}\rangle\rangle_{0}\\
&&-6\sum_{i=m+1}^{k+m-1}(m-i)\langle\langle G(E^{i-1}\circ\gamma^{\mu})G(\gamma_{\mu})E^{k+m-i-1}\rangle\rangle_{0}\\
&&-(k-1)(k+b_{1}\cdot\delta_{k\geq2})\langle\langle E^{m+k-2}\gamma^{\mu}\gamma_{\mu}\rangle\rangle_{0}.
\een
where if $k\geq2$, then $\delta_{k\geq2}=1$, otherwise $\delta_{k\geq2}=0$.
Then it follows that
\ben
&&24E^m\Phi_{k}-24E^k\Phi_{m}\\
&=&\bigg\{-\sum_{i=0}^{k+m-2}\langle\langle G(E^i)\Delta E^{k+m-i-2}\rangle\rangle_{0}-(k+m-1)\langle\langle E^{m+k-2}\gamma^{\mu}\gamma_{\mu}\rangle\rangle_{0}\\
&&+6\sum_{i=1}^{k+m-1}\langle\langle G(E^{i-1}\circ\gamma^{\mu})G(\gamma_{\mu})E^{k+m-i-1}\rangle\rangle_{0}\bigg\}\times(k-m).
\een
The proof is completed by  equation \eqref{phialternative}.
\end{proof}
\subsection{Getzler equation $\textbf{G}(E^{k},\gamma_{\alpha},\gamma_{\beta},\gamma_{\sigma})=0$}
We start with the computation of the function $G_{1}$.
\begin{lem}\label{G1-1E}For any $\alpha,\beta,\sigma$,
\ben
&&G_{1}(E^k,\gamma_{\alpha},\gamma_{\beta},\gamma_{\sigma})\\
&=&-24\{\gamma_{\alpha}\circ\gamma_{\beta}\circ\gamma_{\sigma}\}\langle\langle E^k\rangle\rangle_{1}+24k\langle\langle E^{k-1}\circ\gamma_{\alpha}\circ\gamma_{\beta}\circ\gamma_{\sigma}\rangle\rangle_{1}\\
&&+12\sum_{g\in S_{3}}\langle\langle \{E^k\circ\gamma_{\varsigma_{g(1)}}\}\{\gamma_{\varsigma_{g(2)}}\circ\gamma_{\varsigma_{g(3)}}\}\rangle\rangle_{1}
-12\sum_{g\in S_{3}}\langle\langle \{E^k\circ\gamma_{\varsigma_{g(1)}}\circ\gamma_{\varsigma_{g(2)}}\}\gamma_{\varsigma_{g(3)}}\rangle\rangle_{1}\\
&&+12\sum_{g\in S_{3}}\sum_{i=1}^k\langle\langle G(E^{i-1}\circ\gamma_{\varsigma_{g(1)}})\circ E^{k-i}\circ\gamma_{\varsigma_{g(2)}}\circ\gamma_{\varsigma_{g(3)}}\rangle\rangle_{1}\\
&&-12\sum_{g\in S_{3}}\sum_{i=1}^k\langle\langle G(E^{i-1}\circ\gamma_{\varsigma_{g(1)}}\circ\gamma_{\varsigma_{g(2)}})\circ E^{k-i}\circ\gamma_{\varsigma_{g(3)}}\rangle\rangle_{1}
\een
where $k\geq0$ and $\{\varsigma_{1},\varsigma_{2},\varsigma_{3}\}=\{\alpha,\beta,\sigma\}$.
\end{lem}

\begin{proof}
By the definition of $G_{1}$, we have
\bea
&&G_{1}(E^k,\gamma_{\alpha},\gamma_{\beta},\gamma_{\sigma})\nonumber\\
&=&12\sum_{g\in S_{3}}\langle\langle \{E^k\circ\gamma_{\varsigma_{g(1)}}\}\{\gamma_{\varsigma_{g(2)}}\circ\gamma_{\varsigma_{g(3)}}\}\rangle\rangle_{1}
-12\sum_{g\in S_{3}}\langle\langle \{E^k\circ\gamma_{\varsigma_{g(1)}}\circ\gamma_{\varsigma_{g(2)}}\}\gamma_{\varsigma_{g(3)}}\}\rangle\rangle_{1}\nonumber\\
&&-24\langle\langle\{\gamma_{\alpha}\circ\gamma_{\beta}\circ\gamma_{\sigma}\}E^k\rangle\rangle_{1}-2\sum_{g\in S_{3}}\langle\langle \{E^k\circ\gamma_{\varsigma_{g(1)}}\}\gamma_{\varsigma_{g(2)}}\gamma_{\varsigma_{g(3)}}\gamma^{\mu}\rangle\rangle_{0}\langle\langle\gamma_{\mu}\rangle\rangle_{1}\nonumber\\
&&-2\sum_{g\in S_{3}}\langle\langle E^k\{\gamma_{\varsigma_{g(1)}}\circ\gamma_{\varsigma_{g(2)}}\}\gamma_{\varsigma_{g(3)}}\gamma^{\mu}\rangle\rangle_{0}\langle\langle\gamma_{\mu}\rangle\rangle_{1}\nonumber\\
&&+6\sum_{g\in S_{3}}\langle\langle E^k\gamma_{\varsigma_{g(1)}}\gamma_{\varsigma_{g(2)}}\gamma^{\mu}\rangle\rangle_{0}\langle\langle\gamma_{\mu}\circ\gamma_{\varsigma_{g(3)}}\rangle\rangle_{1}\nonumber\\
&&+12\langle\langle\gamma_{\alpha}\gamma_{\beta}\gamma_{\sigma}\gamma^{\mu}\rangle\rangle_{0}\langle\langle\gamma_{\mu}\circ E^k\rangle\rangle_{1}.\label{G1forumula1}
\eea
Firstly, it follows from equations \eqref{derivative1} and \eqref{derivative3} that
\bea
&&\langle\langle\{\gamma_{\alpha}\circ\gamma_{\beta}\circ\gamma_{\sigma}\}E^k\rangle\rangle_{1}\nonumber\\
&=&\{\gamma_{\alpha}\circ\gamma_{\beta}\circ\gamma_{\sigma}\}\langle\langle E^k\rangle\rangle_{1}-\sum_{i=0}^{k-1}\langle\langle G(E^i)\circ E^{k-i-1}\circ\gamma_{\alpha}\circ\gamma_{\beta}\circ\gamma_{\sigma}\rangle\rangle_{1}\nonumber\\
&&+\sum_{i=1}^{k}\langle\langle G(E^{k-i}\circ\gamma_{\alpha}\circ\gamma_{\beta}\circ\gamma_{\sigma})\circ E^{i-1}\rangle\rangle_{1}-k\langle\langle E^{k-1}\circ\gamma_{\alpha}\circ\gamma_{\beta}\circ\gamma_{\sigma}\rangle\rangle_{1}. \label{G1simplication1}
\eea
Using equation \eqref{WDVV2}, we have
\ben
&&\langle\langle\gamma_{\alpha}\gamma_{\beta}\gamma_{\sigma}\gamma^{\mu}\rangle\rangle_{0}\langle\langle\gamma_{\mu}\circ E^k\rangle\rangle_{1}\\
&=&\langle\langle\{E^k\circ\gamma^{\mu}\}\gamma_{\alpha}\gamma_{\beta}\gamma_{\sigma}\rangle\rangle_{0}\langle\langle\gamma_{\mu}\rangle\rangle_{1}\\
&=&\langle\langle E^k\gamma_{\varsigma_{g(1)}}\gamma_{\varsigma_{g(2)}}\{\gamma^{\mu}\circ\gamma_{\varsigma_{g(3)}}\}\rangle\rangle_{0}\langle\langle\gamma_{\mu}\rangle\rangle_{1}
+\langle\langle \{E^k\circ\gamma_{\varsigma_{g(1)}}\}\gamma_{\varsigma_{g(2)}}\gamma_{\varsigma_{g(3)}}\gamma^{\mu}\rangle\rangle_{0}\langle\langle\gamma_{\mu}\rangle\rangle_{1}\\
&&-\langle\langle \{\gamma_{\varsigma_{g(1)}}\circ\gamma_{\varsigma_{g(3)}}\}E^k\gamma_{\varsigma_{g(2)}}\gamma^{\mu}
\rangle\rangle_{0}\langle\langle\gamma_{\mu}\rangle\rangle_{1}
\een
for any $g\in S_{3}$. Hence it is easy to show that
\bea
&&12\langle\langle\gamma_{\alpha}\gamma_{\beta}\gamma_{\sigma}\gamma^{\mu}\rangle\rangle_{0}\langle\langle\gamma_{\mu}\circ E^k\rangle\rangle_{1}
-2\sum_{g\in S_{3}}\langle\langle \{E^k\circ\gamma_{\varsigma_{g(1)}}\}\gamma_{\varsigma_{g(2)}}\gamma_{\varsigma_{g(3)}}\gamma^{\mu}\rangle\rangle_{0}
\langle\langle\gamma_{\mu}\rangle\rangle_{1}\nonumber\\
&=&2\sum_{g\in S_{3}}\langle\langle E^k\gamma_{\varsigma_{g(1)}}\gamma_{\varsigma_{g(2)}}\gamma^{\mu}\rangle\rangle_{0}\langle\langle\gamma_{\mu}\circ\gamma_{\varsigma_{g(3)}}\rangle\rangle_{1}\nonumber\\
&&-2\sum_{g\in S_{3}}\langle\langle E^k\{\gamma_{\varsigma_{g(1)}}\circ\gamma_{\varsigma_{g(2)}}\}\gamma_{\varsigma_{g(3)}}\gamma^{\mu}\rangle\rangle_{0}
\langle\langle\gamma_{\mu}\rangle\rangle_{1}.\label{G1simplication2}
\eea
By substituting equations \eqref{G1simplication1} and \eqref{G1simplication2} into equation \eqref{G1forumula1},
the proof is completed by using Lemma \ref{4point-3point}.
\end{proof}
\begin{rem}\label{remG1-1E}
Although it is simple to show that there are other equivalent expressions for $G_{1}(E^k,\gamma_{\alpha},\gamma_{\beta},\gamma_{\sigma})$ by equations \eqref{derivative1} and \eqref{derivative3}, for
example,
\bea
&&G_{1}(E^k,\gamma_{\alpha},\gamma_{\beta},\gamma_{\sigma})\nonumber\\
&=&-24\{\gamma_{\alpha}\circ\gamma_{\beta}\circ\gamma_{\sigma}\}\langle\langle E^k\rangle\rangle_{1}+24k\langle\langle E^{k-1}\circ\gamma_{\alpha}\circ\gamma_{\beta}\circ\gamma_{\sigma}\rangle\rangle_{1}\nonumber\\
&&+12\sum_{g\in S_{3}}\{\gamma_{\varsigma_{g(1)}}\circ\gamma_{\varsigma_{g(2)}}\}\langle\langle E^k\circ\gamma_{\varsigma_{g(3)}}\rangle\rangle_{1}
-12\sum_{g\in S_{3}}\gamma_{\varsigma_{g(1)}}\langle\langle E^k\circ\gamma_{\varsigma_{g(2)}}\circ\gamma_{\varsigma_{g(3)}}\rangle\rangle_{1}\nonumber\\
&&+72\langle\langle\gamma_{\alpha}\gamma_{\beta}\gamma_{\sigma}\gamma^{\mu}\rangle\rangle_{0}\langle\langle\gamma_{\mu}\circ E^k\rangle\rangle_{1},\label{equi1}
\eea
we adopt the expression of $G_{1}$ in Lemma \ref{G1-1E} which is useful for later applications in Section 3 and 4. Notice that equation \eqref{equi1}
can not be used to derive $G_{1}(E^k,\gamma_{\alpha},\gamma_{\beta},E^m)$ by substituting  $E^m$ into the position of $\gamma_{\sigma}$ directly.
\end{rem}
By the definition of $G_{0}$, we have for any $\alpha, \beta, \sigma$,
\bea
&&G_{0}(E^{k},\gamma_{\alpha},\gamma_{\beta},\gamma_{\sigma})\nonumber\\
&=&\langle\langle E^k\gamma_{\alpha}\gamma_{\beta}\gamma_{\sigma}\Delta\rangle\rangle_{0}+\frac{1}{2}\sum_{g\in S_{3}}\langle\langle E^k\gamma_{\varsigma_{g(1)}}\gamma_{\varsigma_{g(2)}}\gamma^{\mu}\rangle\rangle_{0}
\langle\langle\gamma_{\mu}\gamma_{\varsigma_{g(3)}}\gamma_{\rho}\gamma^{\rho}\rangle\rangle_{0}\nonumber\\
&&-\sum_{g\in S_{3}}\langle\langle E^k\gamma_{\varsigma_{g(1)}}\gamma^{\mu}\gamma^{\rho}\rangle\rangle_{0}
\langle\langle \gamma_{\mu}\gamma_{\rho}\gamma_{\varsigma_{g(2)}}\gamma_{\varsigma_{g(3)}}\rangle\rangle_{0}\nonumber\\
&&+\langle\langle E^k\gamma_{\rho}\gamma^{\rho}\gamma^{\mu}\rangle\rangle_{0}\langle\langle \gamma_{\mu}\gamma_{\alpha}\gamma_{\beta}\gamma_{\sigma}\rangle\rangle_{0},\label{G0formula1}
\eea
where $k\geq0$ and $\{\varsigma_{1},\varsigma_{2},\varsigma_{3}\}=\{\alpha,\beta,\sigma\}$.
Since the computation is quite involved, we put some intermediate results in Appendix A. With these preparation, we have the
following explicit formula.

\begin{lem}\label{G0-1E}For any $\alpha,\beta,\sigma$,
\ben
&&G_{0}(E^k,\gamma_{\alpha},\gamma_{\beta},\gamma_{\sigma})\\
&=&24\{\gamma_{\alpha}\circ\gamma_{\beta}\circ\gamma_{\sigma}\}\Phi_{k}-4k\langle\langle\{\Delta\circ E^{k-1}\}\gamma_{\alpha}\gamma_{\beta}\gamma_{\sigma}\rangle\rangle_{0}\\
&&+\frac{1}{2}\sum_{g\in S_{3}}\sum_{i=1}^k\langle\langle G(E^{k-i}\circ\gamma_{\varsigma_{g(1)}})\gamma_{\mu}\gamma^{\mu}\{E^{i-1}\circ\gamma_{\varsigma_{g(2)}}\circ\gamma_{\varsigma_{g(3)}}\}\rangle\rangle_{0}\\
&&-\frac{1}{2}\sum_{g\in S_{3}}\sum_{i=1}^k\langle\langle G(E^{k-i}\circ\gamma_{\varsigma_{g(1)}}\circ\gamma_{\varsigma_{g(2)}})\gamma_{\mu}\gamma^{\mu}\{E^{i-1}\circ\gamma_{\varsigma_{g(3)}}\}\rangle\rangle_{0}\\
&&+2k\sum_{g\in S_{3}}\langle\langle G(E^{k-1}\circ\gamma_{\varsigma_{g(1)}}\circ\gamma^{\mu})\gamma_{\varsigma_{g(2)}}\gamma_{\varsigma_{g(3)}}\gamma_{\mu}\rangle\rangle_{0}\\
&&-\frac{1}{2}k\sum_{g\in S_{3}}\langle\langle \{E^{k-1}\circ\gamma_{\varsigma_{g(1)}}\}\gamma_{\mu}\gamma^{\mu}\{\gamma_{\varsigma_{g(2)}}\circ\gamma_{\varsigma_{g(3)}}\}\rangle\rangle_{0}\\
&&+2\sum_{g\in S_{3}}\sum_{i=1}^{k-1}(k-3i)\langle\langle\{G(E^{k-i-1}\circ\gamma_{\varsigma_{g(1)}}\circ\gamma^{\mu})\circ G(\gamma_{\mu})\}E^{i-1}\{\gamma_{\varsigma_{g(2)}}\circ\gamma_{\varsigma_{g(3)}}\}\rangle\rangle_{0}\\
&&-\frac{1}{2}\sum_{g\in S_{3}}\sum_{i=1}^{k-1}i\langle\langle\{\Delta\circ E^{i-1}\}G(E^{k-i-1}\circ\gamma_{\varsigma_{g(1)}}\circ\gamma_{\varsigma_{g(2)}})\gamma_{\varsigma_{g(3)}}\rangle\rangle_{0}\\
&&-\frac{1}{2}\sum_{g\in S_{3}}\sum_{i=1}^{k-1}(k-i)\langle\langle\{\Delta\circ E^{i-1}\}G(E^{k-i-1}\circ\gamma_{\varsigma_{g(1)}})\{\gamma_{\varsigma_{g(2)}}\circ\gamma_{\varsigma_{g(3)}}\}\rangle\rangle_{0}\\
&&-4k\sum_{i=1}^{k-1}\langle\langle G(\Delta\circ E^{k-i-1})E^{i-1}\{\gamma_{\alpha}\circ\gamma_{\beta}\circ\gamma_{\sigma}\}\rangle\rangle_{0}\\
&&+5k(k-1)\langle\langle\Delta E^{k-2}\{\gamma_{\alpha}\circ\gamma_{\beta}\circ\gamma_{\sigma}\}\rangle\rangle_{0},
\een
where $k\geq0$ and $\{\varsigma_{1},\varsigma_{2},\varsigma_{3}\}=\{\alpha,\beta,\sigma\}$.
\end{lem}

\begin{proof}
Firstly, by substituting Lemmas \ref{G0part11}, \ref{G0part12}, \ref{G0part13}, \ref{G0part14}, \ref{G0part15} into Lemma \ref{G0part1},
and substituting Lemma \ref{G0part21} into Lemma \ref{G0part2}, and combining with Lemmas \ref{G0part3} and \ref{G0part4}, we obtain an
expression for the function $G_{0}$ by equation \eqref{G0formula1}. Then subsituting Lemma \ref{G0part41} to the resulting expression, and using Theorem \ref{explicit-dphi} and Lemma \ref{simplication} for simplification, a tedious computation shows that
\ben
&&G_{0}(E^k,\gamma_{\alpha},\gamma_{\beta},\gamma_{\sigma})\\
&=&24\{\gamma_{\alpha}\circ\gamma_{\beta}\circ\gamma_{\sigma}\}\Phi_{k}-(b_{1}+k)\langle\langle\{\Delta\circ E^{k-1}\}\gamma_{\alpha}\gamma_{\beta}\gamma_{\sigma}\rangle\rangle_{0}\\
&&-2\sum_{i=1}^{k}\langle\langle\{G(E^{k-i})\circ E^{i-1}\}\gamma_{\mu}\gamma^{\mu}\{\gamma_{\alpha}\circ\gamma_{\beta}\circ\gamma_{\sigma}\}\rangle\rangle_{0}\\
&&-\frac{1}{2}\sum_{g\in S_{3}}\sum_{i=1}^k\langle\langle G(E^{k-i}\circ\gamma_{\varsigma_{g(1)}}\circ\gamma_{\varsigma_{g(2)}})\gamma_{\mu}\gamma^{\mu}\{E^{i-1}\circ\gamma_{\varsigma_{g(3)}}\}\rangle\rangle_{0}\\
&&+\sum_{g\in S_{3}}\sum_{i=1}^k\langle\langle\{G(E^{k-i}\circ\gamma_{\varsigma_{g(1)}})\circ\gamma_{\varsigma_{g(2)}}\}
\gamma_{\mu}\gamma^{\mu}\{E^{i-1}\circ\gamma_{\varsigma_{g(3)}}\}\rangle\rangle_{0}\\
&&-\frac{1}{6}\sum_{g\in S_{3}}\sum_{i=1}^k\langle\langle\{G(E^{k-i})\circ E^{i-1}\circ\gamma^{\mu}\}\{\gamma_{\varsigma_{g(1)}}\circ\gamma_{\varsigma_{g(2)}}\}\gamma_{\varsigma_{g(3)}}\gamma_{\mu}\rangle\rangle_{0}\\
&&+\frac{1}{2}\sum_{g\in S_{3}}\sum_{i=1}^k\langle\langle G(E^{k-i}) \{\gamma_{\varsigma_{g(1)}}\circ\gamma_{\varsigma_{g(2)}}\}\{\gamma_{\varsigma_{g(3)}}\circ\gamma^{\mu}\}\{E^{i-1}\circ\gamma_{\mu}\}\rangle\rangle_{0}\\
&&+\frac{1}{3}\sum_{g\in S_{3}}\sum_{i=1}^k\langle\langle\{G(E^{k-i})\circ E^{i-1}\circ\gamma_{\varsigma_{g(1)}}\circ\gamma^{\mu}\}\gamma_{\varsigma_{g(2)}}\gamma_{\varsigma_{g(3)}}\gamma_{\mu}\rangle\rangle_{0}\\
&&-\frac{1}{6}\sum_{g\in S_{3}}\sum_{i=1}^{k-1}\langle\langle\{\Delta\circ E^{i-1}\}\gamma_{\varsigma_{g(1)}}\gamma_{\varsigma_{g(2)}}\{G(E^{k-i})\circ\gamma_{\varsigma_{g(3)}}\}\rangle\rangle_{0}\\
&&-\frac{1}{6}\sum_{g\in S_{3}}\sum_{i=1}^{k-1}\langle\langle\{\Delta\circ E^{i-1}\}\{\gamma_{\varsigma_{g(1)}}\circ\gamma_{\varsigma_{g(2)}}\}\gamma_{\varsigma_{g(3)}}G(E^{k-i})\rangle\rangle_{0}\\
&&+\frac{1}{2}\sum_{g\in S_{3}}\sum_{i=1}^{k}\langle\langle\{\Delta\circ E^{i-1}\}\gamma_{\varsigma_{g(1)}}\gamma_{\varsigma_{g(2)}}G(E^{k-i}\circ\gamma_{\varsigma_{g(3)}})\rangle\rangle_{0}\\
&&+k\sum_{g\in S_{3}}\langle\langle G(E^{k-1}\circ\gamma_{\varsigma_{g(1)}}\circ\gamma^{\mu})\gamma_{\varsigma_{g(2)}}\gamma_{\varsigma_{g(3)}}\gamma_{\mu}\rangle\rangle_{0}\\
&&-\sum_{g\in S_{3}}\sum_{i=1}^{k}\langle\langle \{G(E^{k-i}\circ\gamma_{\varsigma_{g(1)}})\circ E^{i-1}\circ\gamma^{\mu}\}\gamma_{\varsigma_{g(2)}}\gamma_{\varsigma_{g(3)}}\gamma_{\mu}\rangle\rangle_{0}\\
&&-k\sum_{g\in S_{3}}\langle\langle \{G(E^{k-1}\circ\gamma^{\mu})\circ\gamma_{\varsigma_{g(1)}}\}\gamma_{\varsigma_{g(2)}}\gamma_{\varsigma_{g(3)}}\gamma_{\mu}\rangle\rangle_{0}\\
&&-6k\sum_{i=1}^{k-1}\langle\langle\{G(E^{k-i-1}\circ\gamma^{\mu})\circ G(\gamma_{\mu})\}E^{i-1}\{\gamma_{\alpha}\circ\gamma_{\beta}\circ\gamma_{\sigma}\}\rangle\rangle_{0}\\
&&+3\sum_{g\in S_{3}}\sum_{i=1}^{k-1}(k-2i)\langle\langle\{G(E^{k-i-1}\circ\gamma_{\varsigma_{g(1)}}\circ\gamma^{\mu})\circ G(\gamma_{\mu})\}E^{i-1}\{\gamma_{\varsigma_{g(2)}}\circ\gamma_{\varsigma_{g(3)}}\}\rangle\rangle_{0}\\
&&-2\sum_{i=1}^k\sum_{j=1}^{i-1}\langle\langle\{G(\Delta\circ E^{k-i})\circ G(E^{j-1})\}E^{i-j-1}\{\gamma_{\alpha}\circ\gamma_{\beta}\circ\gamma_{\sigma}\}\rangle\rangle_{0}\\
&&+2\sum_{i=1}^k\sum_{j=1}^{i-1}\langle\langle G(E^{k-i})E^{i-j-1}G(\Delta\circ E^{j-1}\circ \gamma_{\alpha}\circ\gamma_{\beta}\circ\gamma_{\sigma})\rangle\rangle_{0}\\
&&+\frac{1}{2}\sum_{g\in S_{3}}\sum_{i=1}^k\sum_{j=1}^{i-1}\langle\langle\{G(\Delta\circ E^{k-i})\circ E^{j-1}\}G(E^{i-j-1}\circ\gamma_{\varsigma_{g(1)}})\{\gamma_{\varsigma_{g(2)}}\circ\gamma_{\varsigma_{g(3)}}\}\rangle\rangle_{0}\\
&&-\frac{1}{2}\sum_{g\in S_{3}}\sum_{i=1}^{k-1}i\langle\langle\{\Delta\circ E^{i-1}\}G(E^{k-i-1}\circ\gamma_{\varsigma_{g(1)}}\circ\gamma_{\varsigma_{g(2)}})\gamma_{\varsigma_{g(3)}}\rangle\rangle_{0}\\
&&-\frac{1}{2}\sum_{g\in S_{3}}\sum_{i=1}^k\sum_{j=1}^{i-1}\langle\langle\{\Delta\circ E^{k-i}\}G(E^{j-1}\circ\gamma_{\varsigma_{g(1)}}\circ\gamma_{\varsigma_{g(2)}})G(E^{i-j-1}\circ\gamma_{\varsigma_{g(3)}})\rangle\rangle_{0}\\
&&+\sum_{g\in S_{3}}\sum_{i=1}^k\sum_{j=1}^{i-1}\langle\langle\{\Delta\circ E^{k-i}\circ\gamma_{\varsigma_{g(1)}}\}G(E^{j-1}\circ\gamma_{\varsigma_{g(2)}})G(E^{i-j-1}\circ\gamma_{\varsigma_{g(3)}})\rangle\rangle_{0}\\
&&-k\sum_{i=0}^{k-2}\langle\langle G(\Delta\circ E^{k-i-2})E^i\{\gamma_{\alpha}\circ\gamma_{\beta}\circ\gamma_{\sigma}\}\rangle\rangle_{0}\\
&&+2k(k-1)\langle\langle\Delta E^{k-2}\{\gamma_{\alpha}\circ\gamma_{\beta}\circ\gamma_{\sigma}\}\rangle\rangle_{0}.
\een
The proof follows from some tedious manipulation again, that is, substituting Lemmas \ref{4pointreduction1}, \ref{4pointreduction2},
\ref{4pointreduction4} into the above expression, and using Lemma \ref{4pointreduction3} to cancel the redundant terms in the resulting expression,
and then using Lemma \ref{simplication}  for possible simplification.
\end{proof}
The main result of this subsection is the following
\begin{thm}\label{Getzler-1E}For any $\alpha,\beta,\sigma$,
\ben
&&24\{\gamma_{\alpha}\circ\gamma_{\beta}\circ\gamma_{\sigma}\}(\langle\langle E^k\rangle\rangle_{1}-\Phi_{k})\\
&=&12\sum_{g\in S_{3}}\langle\langle \{E^k\circ\gamma_{\varsigma_{g(1)}}\}\{\gamma_{\varsigma_{g(2)}}\circ\gamma_{\varsigma_{g(3)}}\}\rangle\rangle_{1}
-12\sum_{g\in S_{3}}\langle\langle \{E^k\circ\gamma_{\varsigma_{g(1)}}\circ\gamma_{\varsigma_{g(2)}}\}\gamma_{\varsigma_{g(3)}}\}\rangle\rangle_{1}\\
&&+12\sum_{g\in S_{3}}\sum_{i=1}^k\langle\langle G(E^{i-1}\circ\gamma_{\varsigma_{g(1)}})\circ E^{k-i}\circ\gamma_{\varsigma_{g(2)}}\circ\gamma_{\varsigma_{g(3)}}\rangle\rangle_{1}\\
&&-12\sum_{g\in S_{3}}\sum_{i=1}^k\langle\langle G(E^{i-1}\circ\gamma_{\varsigma_{g(1)}}\circ\gamma_{\varsigma_{g(2)}})\circ E^{k-i}\circ\gamma_{\varsigma_{g(3)}}\rangle\rangle_{1}\\
&&+24k\langle\langle E^{k-1}\circ\gamma_{\alpha}\circ\gamma_{\beta}\circ\gamma_{\sigma}\rangle\rangle_{1}-4k\langle\langle\{\Delta\circ E^{k-1}\}\gamma_{\alpha}\gamma_{\beta}\gamma_{\sigma}\rangle\rangle_{0}\\
&&+\frac{1}{2}\sum_{g\in S_{3}}\sum_{i=1}^k\langle\langle G(E^{k-i}\circ\gamma_{\varsigma_{g(1)}})\gamma_{\mu}\gamma^{\mu}\{E^{i-1}\circ\gamma_{\varsigma_{g(2)}}\circ\gamma_{\varsigma_{g(3)}}\}\rangle\rangle_{0}\\
&&-\frac{1}{2}\sum_{g\in S_{3}}\sum_{i=1}^k\langle\langle G(E^{k-i}\circ\gamma_{\varsigma_{g(1)}}\circ\gamma_{\varsigma_{g(2)}})\gamma_{\mu}\gamma^{\mu}\{E^{i-1}\circ\gamma_{\varsigma_{g(3)}}\}\rangle\rangle_{0}\\
&&+2k\sum_{g\in S_{3}}\langle\langle G(E^{k-1}\circ\gamma_{\varsigma_{g(1)}}\circ\gamma^{\mu})\gamma_{\varsigma_{g(2)}}\gamma_{\varsigma_{g(3)}}\gamma_{\mu}\rangle\rangle_{0}\\
&&-\frac{1}{2}k\sum_{g\in S_{3}}\langle\langle \{E^{k-1}\circ\gamma_{\varsigma_{g(1)}}\}\gamma_{\mu}\gamma^{\mu}\{\gamma_{\varsigma_{g(2)}}\circ\gamma_{\varsigma_{g(3)}}\}\rangle\rangle_{0}\\
&&+2\sum_{g\in S_{3}}\sum_{i=1}^{k-1}(k-3i)\langle\langle\{G(E^{k-i-1}\circ\gamma_{\varsigma_{g(1)}}\circ\gamma^{\mu})\circ G(\gamma_{\mu})\}E^{i-1}\{\gamma_{\varsigma_{g(2)}}\circ\gamma_{\varsigma_{g(3)}}\}\rangle\rangle_{0}\\
&&-\frac{1}{2}\sum_{g\in S_{3}}\sum_{i=1}^{k-1}i\langle\langle\{\Delta\circ E^{i-1}\}G(E^{k-i-1}\circ\gamma_{\varsigma_{g(1)}}\circ\gamma_{\varsigma_{g(2)}})\gamma_{\varsigma_{g(3)}}\rangle\rangle_{0}\\
&&-\frac{1}{2}\sum_{g\in S_{3}}\sum_{i=1}^{k-1}(k-i)\langle\langle\{\Delta\circ E^{i-1}\}G(E^{k-i-1}\circ\gamma_{\varsigma_{g(1)}})\{\gamma_{\varsigma_{g(2)}}\circ\gamma_{\varsigma_{g(3)}}\}\rangle\rangle_{0}\\
&&-4k\sum_{i=1}^{k-1}\langle\langle G(\Delta\circ E^{k-i-1})E^{i-1}\{\gamma_{\alpha}\circ\gamma_{\beta}\circ\gamma_{\sigma}\}\rangle\rangle_{0}\\
&&+5k(k-1)\langle\langle\Delta E^{k-2}\{\gamma_{\alpha}\circ\gamma_{\beta}\circ\gamma_{\sigma}\}\rangle\rangle_{0},
\een
where $k\geq0$ and $\{\varsigma_{1},\varsigma_{2},\varsigma_{3}\}=\{\alpha,\beta,\sigma\}$.
\end{thm}
\begin{proof}
It follows easily from Lemma \ref{G1-1E}, Lemma \ref{G0-1E}  and equation \eqref{Getzlertrr}.
\end{proof}

\begin{rem}\label{addedrem}
In Lemma \ref{G1-1E}, Lemma \ref{G0-1E}, and Theorem \ref{Getzler-1E}, one can replace $\gamma_{\alpha}, \gamma_{\beta}$, and $\gamma_{\sigma}$
 by $\upsilon_{1}, \upsilon_{2}$, and $\upsilon_{3}$ respectively for any vector fields $\upsilon_{1}, \upsilon_{2}$, and $\upsilon_{3}$ on the small phase space.
The same holds for the following Lemmas \ref{G1-2E}, \ref{G0-2E}, \ref{G1-3E}, \ref{G0-3E}, and Theorems \ref{Getzler-2E}, \ref{Getzler-3E}.
\end{rem}

\subsection{Getzler equation $\textbf{G}(E^{k_{1}},E^{k_{2}},\gamma_{\alpha},\gamma_{\beta})=0$}
In this subsection, the following explicit formulas for functions $G_{1}$, $G_{0}$ and $\textbf{G}$ are derived from the above subsection.
\begin{lem}\label{G1-2E}For $k_{1}, k_{2}\geq0$, and any $\alpha,\beta$,
\ben
&&G_{1}(E^{k_{1}},E^{k_{2}},\gamma_{\alpha},\gamma_{\beta})\\
&=&24\{\gamma_{\alpha}\circ\gamma_{\beta}\}\langle\langle E^{\widehat{K}}\rangle\rangle_{1}-24\sum_{m=1}^{2}\{E^{\widehat{K}-k_{m}}\circ\gamma_{\alpha}\circ\gamma_{\beta}\}\langle\langle E^{k_{m}}\rangle\rangle_{1}\\
&&+24\sum_{h\in S_{2}}\langle\langle\{E^{k_{h(1)}}\circ\gamma_{\alpha}\}\{E^{k_{h(2)}}\circ\gamma_{\beta}\}\rangle\rangle_{1}
-24\sum_{g\in S_{2}}\langle\langle\{E^{\widehat{K}}\circ\gamma_{\varsigma_{g(1)}}\}\gamma_{\varsigma_{g(2)}}\rangle\rangle_{1}\\
&&+24\sum_{g\in S_{2}}\sum_{m=1}^2\sum_{i=1}^{k_{m}}\langle\langle G(E^{i-1}\circ\gamma_{\varsigma_{g(1)}})\circ E^{\widehat{K}-i}\circ\gamma_{\varsigma_{g(2)}}\rangle\rangle_{1}\\
&&-24\sum_{g\in S_{2}}\sum_{i=1}^{\widehat{K}}\langle\langle G(E^{i-1}\circ\gamma_{\varsigma_{g(1)}})\circ E^{\widehat{K}-i}\circ\gamma_{\varsigma_{g(2)}}\rangle\rangle_{1},
\een
where $\widehat{K}=k_{1}+k_{2}$ and $\{\varsigma_{1},\varsigma_{2}\}=\{\alpha,\beta\}$.
\end{lem}

\begin{proof}
The proof follows easily from Lemma \ref{G1-1E} and equations \eqref{derivative1} and  \eqref{derivative3}.
\end{proof}

\begin{rem}\label{remG1-2E}
 As in Remark \ref{remG1-1E}, we also have the following equivalent form
\ben
&&G_{1}(E^{k_{1}},E^{k_{2}},\gamma_{\alpha},\gamma_{\beta})\\
&=&24\{\gamma_{\alpha}\circ\gamma_{\beta}\}\langle\langle E^{\widehat{K}}\rangle\rangle_{1}-24\sum_{m=1}^{2}\{E^{\widehat{K}-k_{m}}\circ\gamma_{\alpha}\circ\gamma_{\beta}\}\langle\langle E^{k_{m}}\rangle\rangle_{1}\\
&&+12\sum_{g\in S_{2}}\sum_{h\in S_{2}}\{E^{k_{h(1)}}\circ\gamma_{\varsigma_{g(1)}}\}\langle\langle E^{k_{h(2)}}\circ\gamma_{\varsigma_{g(2)}}\rangle\rangle_{1}
-24\sum_{g\in S_{2}}\gamma_{\varsigma_{g(1)}}\langle\langle E^{\widehat{K}}\circ\gamma_{\varsigma_{g(2)}}\rangle\rangle_{1}\\
&&+12\sum_{g\in S_{2}}\sum_{m=1}^2\sum_{i=1}^{k_{m}}\langle\langle G(E^{i-1}\circ\gamma_{\varsigma_{g(1)}})\circ E^{\widehat{K}-i}\circ\gamma_{\varsigma_{g(2)}}\rangle\rangle_{1}\\
&&-24\sum_{m=1}^2\sum_{i=1}^{k_{m}}\langle\langle G(E^{i-1}\circ\gamma_{\alpha}\circ\gamma_{\beta})\circ E^{\widehat{K}-i}\rangle\rangle_{1}
+24\widehat{K}\langle\langle E^{\widehat{K}-1}\circ\gamma_{\alpha}\circ\gamma_{\beta}\rangle\rangle_{1},
\een
by equations \eqref{derivative1} and \eqref{derivative3}.
\end{rem}
%\begin{lem}
%\ben
%&&G_{0}(E^{k_{1}},E^{k_{2}},\gamma_{\alpha},\gamma_{\beta})\\
%&=&\langle\langle E^{k_{1}}E^{k_{2}}\gamma_{\alpha}\gamma^{\mu}\rangle\rangle_{0}\langle\langle \gamma_{\mu}\gamma_{\beta}\gamma_{\sigma}\gamma^{\sigma}\rangle\rangle_{0}+\langle\langle E^{k_{1}}E^{k_{2}}\gamma_{\beta}\gamma^{\mu}\rangle\rangle_{0}\langle\langle \gamma_{\mu}\gamma_{\alpha}\gamma_{\sigma}\gamma^{\sigma}\rangle\rangle_{0}\\
%&&+\langle\langle E^{k_{1}}\gamma_{\alpha}\gamma_{\beta}\gamma^{\mu}\rangle\rangle_{0}\langle\langle \gamma_{\mu}E^{k_{2}}\gamma_{\sigma}\gamma^{\sigma}\rangle\rangle_{0}+\langle\langle E^{k_{2}}\gamma_{\alpha}\gamma_{\beta}\gamma^{\mu}\rangle\rangle_{0}\langle\langle \gamma_{\mu}E^{k_{1}}\gamma_{\sigma}\gamma^{\sigma}\rangle\rangle_{0}\\
%&&+\langle\langle E^{k_{1}}E^{k_{2}}\gamma_{\alpha}\gamma_{\beta}\gamma^{\mu}\rangle\rangle_{0}\langle\langle \gamma_{\mu}\gamma_{\sigma}\gamma^{\sigma}\rangle\rangle_{0}-2\langle\langle E^{k_{1}}E^{k_{2}}\gamma^{\mu}\gamma^{\sigma}\rangle\rangle_{0}\langle\langle \gamma_{\mu}\gamma_{\sigma}\gamma_{\alpha}\gamma_{\beta}\rangle\rangle_{0}\\
%&&-2\langle\langle E^{k_{1}}\gamma_{\alpha}\gamma^{\mu}\gamma^{\sigma}\rangle\rangle_{0}\langle\langle \gamma_{\mu}\gamma_{\sigma}E^{k_{2}}\gamma_{\beta}\rangle\rangle_{0}-2\langle\langle E^{k_{1}}\gamma_{\beta}\gamma^{\mu}\gamma^{\sigma}\rangle\rangle_{0}\langle\langle \gamma_{\mu}\gamma_{\sigma}E^{k_{2}}\gamma_{\alpha}\rangle\rangle_{0}
%\een
%\end{lem}

\begin{lem}\label{G0-2E}For $k_{1}, k_{2}\geq0$, and any $\alpha,\beta$,
\ben
&&G_{0}(E^{k_{1}},E^{k_{2}},\gamma_{\alpha},\gamma_{\beta})\\
&=&-24\{\gamma_{\alpha}\circ\gamma_{\beta}\}\Phi_{\widehat{K}}+24\sum_{m=1}^2\{E^{\widehat{K}-k_{m}}\circ\gamma_{\alpha}\circ\gamma_{\beta}\}\Phi_{k_{m}}\\
%&&+2\sum_{i=1}^{\widehat{K}}\sum_{\mu}\langle\langle G(E^{i-1}\circ\gamma^{\mu})\gamma_{\alpha}\gamma_{\beta}\{E^{\widehat{K}-i}\circ\gamma_{\mu}\}\rangle\rangle_{0}\\
%&&-2\sum_{m=1}^2\sum_{i=1}^{k_{m}}\sum_{\mu}\langle\langle G(E^{i-1}\circ\gamma^{\mu})\gamma_{\alpha}\gamma_{\beta}\{E^{\widehat{K}-i}\circ\gamma_{\mu}\}\rangle\rangle_{0}\\
&&-\sum_{g\in S_{2}}\sum_{i=1}^{\widehat{K}}\langle\langle G(E^{i-1}\circ\gamma_{\varsigma_{g(1)}})\gamma^{\mu}\gamma_{\mu}\{E^{\widehat{K}-i}\circ\gamma_{\varsigma_{g(2)}}\}\rangle\rangle_{0}\\
&&+\sum_{g\in S_{2}}\sum_{m=1}^2\sum_{i=1}^{k_{m}}\langle\langle G(E^{i-1}\circ\gamma_{\varsigma_{g(1)}})\gamma^{\mu}\gamma_{\mu}\{E^{\widehat{K}-i}\circ\gamma_{\varsigma_{g(2)}}\}\rangle\rangle_{0}\\
&&-\sum_{g\in S_{2}}\sum_{i=1}^{k_{1}}\sum_{j=1}^{k_{2}}\langle\langle\{\Delta\circ E^{i+j-2}\}G(E^{\widehat{K}-i-j}\circ\gamma_{\varsigma_{g(1)}})\gamma_{\varsigma_{g(2)}}\rangle\rangle_{0}\\
&&+6\sum_{g\in S_{2}}\sum_{i=1}^{k_{1}}\sum_{j=1}^{k_{2}}\langle\langle \{G(E^{i+j-2}\circ\gamma_{\varsigma_{g(1)}}\circ\gamma^{\mu})\circ G(\gamma_{\mu})\}E^{\widehat{K}-i-j}\gamma_{\varsigma_{g(2)}}\rangle\rangle_{0}\\
&&-2k_{1}k_{2}\langle\langle\Delta E^{\widehat{K}-2}\{\gamma_{\alpha}\circ\gamma_{\beta}\}\rangle\rangle_{0},
\een
where $\widehat{K}=k_{1}+k_{2}$ and $\{\varsigma_{1},\varsigma_{2}\}=\{\alpha,\beta\}$.
\end{lem}

\begin{proof}
By Lemma \ref{G0-1E}, we get the expression of $G_{0}(E^{k_{1}},E^{k_{2}},\gamma_{\alpha},\gamma_{\beta})$ which contains
three terms: $\langle\langle G(E^{\widehat{K}-1}\circ\gamma^{\mu})\gamma_{\alpha}\gamma_{\beta}\gamma_{\mu}\rangle\rangle_{0}$,
$\langle\langle\{E^{k_{1}-1}\circ\gamma_{\alpha}\}\gamma_{\mu}\gamma^{\mu}\{\gamma_{\beta}\circ E^{k_{2}}\}\rangle\rangle_{0}$, and
$\langle\langle\{E^{k_{1}-1}\circ\gamma_{\beta}\}\gamma_{\mu}\gamma^{\mu}\{\gamma_{\alpha}\circ E^{k_{2}}\}\rangle\rangle_{0}$. They are
computed as follows.
\ben
&&\langle\langle G(E^{\widehat{K}-1}\circ\gamma^{\mu})\gamma_{\alpha}\gamma_{\beta}\gamma_{\mu}\rangle\rangle_{0}\\
&=&\langle\langle G(\gamma_{\mu})\gamma_{\alpha}\gamma_{\beta}\{E^{\widehat{K}-1}\circ\gamma^{\mu}\}\rangle\rangle_{0}\\
&=&\langle\langle E^{\widehat{K}-1}\{\gamma_{\beta}\circ\gamma^{\mu}\}\gamma_{\alpha}G(\gamma_{\mu})\rangle\rangle_{0}
+\langle\langle \{E^{\widehat{K}-1}\circ\gamma_{\alpha}\}\gamma^{\mu}G(\gamma_{\mu})\gamma_{\beta}\rangle\rangle_{0}\\
&&-\langle\langle E^{\widehat{K}-1}\{\gamma_{\alpha}\circ\gamma_{\beta}\}\gamma^{\mu}G(\gamma_{\mu})\rangle\rangle_{0},
\een
\ben
&&\langle\langle\{E^{k_{1}-1}\circ\gamma_{\alpha}\}\gamma_{\mu}\gamma^{\mu}\{\gamma_{\beta}\circ E^{k_{2}}\}\rangle\rangle_{0}\\
&=&\langle\langle E^{k_{2}}\{\gamma_{\beta}\circ\gamma^{\mu}\}\gamma_{\mu}\{E^{k_{1}-1}\circ\gamma_{\alpha}\}\rangle\rangle_{0}
+\langle\langle\{E^{\widehat{K}-1}\circ\gamma_{\alpha}\}\gamma_{\mu}\gamma^{\mu}\gamma_{\beta}\rangle\rangle_{0}\\
&&-\langle\langle E^{k_{2}}\gamma_{\beta}\gamma_{\mu}\{E^{k_{1}-1}\circ\gamma_{\alpha}\circ\gamma^{\mu}\}\rangle\rangle_{0},
\een
\ben
&&\langle\langle\{E^{k_{1}-1}\circ\gamma_{\beta}\}\gamma_{\mu}\gamma^{\mu}\{\gamma_{\alpha}\circ E^{k_{2}}\}\rangle\rangle_{0}\\
&=&\langle\langle E^{k_{1}-1}\{\gamma_{\beta}\circ\gamma^{\mu}\}\gamma_{\mu}\{E^{k_{2}}\circ\gamma_{\alpha}\}\rangle\rangle_{0}
+\langle\langle\{E^{\widehat{K}-1}\circ\gamma_{\alpha}\}\gamma_{\mu}\gamma^{\mu}\gamma_{\beta}\rangle\rangle_{0}\\
&&-\langle\langle E^{k_{1}-1}\gamma_{\beta}\gamma_{\mu}\{E^{k_{2}}\circ\gamma_{\alpha}\circ\gamma^{\mu}\}\rangle\rangle_{0}.
\een
The proof is completed by Lemmas \ref{4point-3point}, \ref{simplication} and Theorem \ref{explicit-dphi}.
\end{proof}

\begin{thm}\label{Getzler-2E}For $k_{1}, k_{2}\geq0$, and any $\alpha,\beta$,
\ben
&&24\{\gamma_{\alpha}\circ\gamma_{\beta}\}(\langle\langle E^{\widehat{K}}\rangle\rangle_{1}-\Phi_{\widehat{K}})-24\sum_{m=1}^2\{E^{\widehat{K}-k_{m}}\circ\gamma_{\alpha}\circ\gamma_{\beta}\}(\langle\langle E^{k_{m}}\rangle\rangle_{1}-\Phi_{k_{m}})\\
&=&-24\sum_{h\in S_{2}}\langle\langle\{E^{k_{h(1)}}\circ\gamma_{\alpha}\}\{E^{k_{h(2)}}\circ\gamma_{\beta}\}\rangle\rangle_{1}
+24\sum_{g\in S_{2}}\langle\langle\{E^{\widehat{K}}\circ\gamma_{\varsigma_{g(1)}}\}\gamma_{\varsigma_{g(2)}}\rangle\rangle_{1}\\
&&-24\sum_{g\in S_{2}}\sum_{m=1}^2\sum_{i=1}^{k_{m}}\langle\langle G(E^{i-1}\circ\gamma_{\varsigma_{g(1)}})\circ E^{\widehat{K}-i}\circ\gamma_{\varsigma_{g(2)}}\rangle\rangle_{1}\\
&&+24\sum_{g\in S_{2}}\sum_{i=1}^{\widehat{K}}\langle\langle G(E^{i-1}\circ\gamma_{\varsigma_{g(1)}})\circ E^{\widehat{K}-i}\circ\gamma_{\varsigma_{g(2)}}\rangle\rangle_{1}\\
&&+\sum_{g\in S_{2}}\sum_{i=1}^{\widehat{K}}\langle\langle G(E^{i-1}\circ\gamma_{\varsigma_{g(1)}})\gamma^{\mu}\gamma_{\mu}\{E^{\widehat{K}-i}\circ\gamma_{\varsigma_{g(2)}}\}\rangle\rangle_{0}\\
&&-\sum_{g\in S_{2}}\sum_{m=1}^2\sum_{i=1}^{k_{m}}\langle\langle G(E^{i-1}\circ\gamma_{\varsigma_{g(1)}})\gamma^{\mu}\gamma_{\mu}\{E^{\widehat{K}-i}\circ\gamma_{\varsigma_{g(2)}}\}\rangle\rangle_{0}\\
&&+\sum_{g\in S_{2}}\sum_{i=1}^{k_{1}}\sum_{j=1}^{k_{2}}\langle\langle\{\Delta\circ E^{i+j-2}\}G(E^{\widehat{K}-i-j}\circ\gamma_{\varsigma_{g(1)}})\gamma_{\varsigma_{g(2)}}\rangle\rangle_{0}\\
&&-6\sum_{g\in S_{2}}\sum_{i=1}^{k_{1}}\sum_{j=1}^{k_{2}}\langle\langle \{G(E^{i+j-2}\circ\gamma_{\varsigma_{g(1)}}\circ\gamma^{\mu})\circ G(\gamma_{\mu})\}E^{\widehat{K}-i-j}\gamma_{\varsigma_{g(2)}}\rangle\rangle_{0}\\
&&+2k_{1}k_{2}\langle\langle\Delta E^{\widehat{K}-2}\{\gamma_{\alpha}\circ\gamma_{\beta}\}\rangle\rangle_{0},
\een
where $\widehat{K}=k_{1}+k_{2}$ and $\{\varsigma_{1},\varsigma_{2}\}=\{\alpha,\beta\}$.
\end{thm}
\begin{proof}
It follows easily from Lemma \ref{G1-2E}, Lemma \ref{G0-2E} and equation \eqref{Getzlertrr}.
\end{proof}
\subsection{Getzler equation $\textbf{G}(E^{k_{1}},E^{k_{2}},E^{k_{3}},\gamma_{\alpha})=0$} In this subsection, we will derive the following results
from subsection 3.3.
\begin{lem}\label{G1-3E}For $k_{1}, k_{2}, k_{3}\geq0$, and any $\alpha$,
\ben
&&G_{1}(E^{k_{1}},E^{k_{2}},E^{k_{3}},\gamma_{\alpha})\\
&=&-24\gamma_{\alpha}\langle\langle E^{\widetilde{K}}\rangle\rangle_{1}+24\sum_{i=1}^{3}\{E^{k_{i}}\circ\gamma_{\alpha}\}\langle\langle E^{\widetilde{K}-k_{i}}\rangle\rangle_{1}-24\sum_{i=1}^{3}\{E^{\widetilde{K}-k_{i}}\circ\gamma_{\alpha}\}\langle\langle E^{k_{i}}\rangle\rangle_{1},
\een
where $\widetilde{K}=k_{1}+k_{2}+k_{3}$.
\end{lem}
\begin{proof}
It follows easily from Lemma \ref{G1-2E} and  equations \eqref{derivative1} and  \eqref{derivative3}.
\end{proof}
\begin{lem}\label{G0-3E}For $k_{1}, k_{2}, k_{3}\geq0$, and any $\alpha$,
\ben
&&G_{0}(E^{k_{1}},E^{k_{2}},E^{k_{3}},\gamma_{\alpha})\\
&=&24\gamma_{\alpha}\Phi_{\widetilde{K}}-24\sum_{i=1}^{3}\{E^{k_{i}}\circ\gamma_{\alpha}\}\Phi_{\widetilde{K}-k_{i}}
+24\sum_{i=1}^{3}\{E^{\widetilde{K}-k_{i}}\circ\gamma_{\alpha}\}\Phi_{k_{i}},
\een
where $\widetilde{K}=k_{1}+k_{2}+k_{3}$.
\end{lem}
\begin{proof}
It follows from Lemmas \ref{G0-2E}, \ref{4point-3point}, \ref{simplication} and Theorem \ref{explicit-dphi}.
\end{proof}

\begin{thm}\label{Getzler-3E}For $k_{1}, k_{2}, k_{3}\geq0$, and any $\alpha$,
\ben
\gamma_{\alpha}(\langle\langle E^{\widetilde{K}}\rangle\rangle_{1}-\Phi_{\widetilde{K}})&=&\sum_{i=1}^{3}\{E^{k_{i}}\circ\gamma_{\alpha}\}(\langle\langle E^{\widetilde{K}-k_{i}}\rangle\rangle_{1}-\Phi_{\widetilde{K}-k_{i}})\\
&&-\sum_{i=1}^{3}\{E^{\widetilde{K}-k_{i}}\circ\gamma_{\alpha}\}(\langle\langle E^{k_{i}}\rangle\rangle_{1}-\Phi_{k_{i}}),
\een
where $\widetilde{K}=k_{1}+k_{2}+k_{3}$.
\end{thm}
\begin{proof}
It follows easily from Lemma \ref{G1-3E}, Lemma \ref{G0-3E} and equation \eqref{Getzlertrr}.
\end{proof}

\subsection{Getzler equation $\textbf{G}(E^{k_{1}},E^{k_{2}},E^{k_{3}},E^{k_{4}})=0$}
The case in this subsection has been studied in \cite{L1}. We present alternative symmetric expressions as follows.
\begin{lem}\label{G1-4E}For $k_{1},k_{2},k_{3},k_{4}\geq0$,
\ben
&&G_{1}(E^{k_{1}},E^{k_{2}},E^{k_{3}},E^{k_{4}})\\
&=&36K\langle\langle E^{K-1}\rangle\rangle_{1}-24\sum_{i=1}^4E^{k_{i}}\langle\langle E^{K-k_{i}}\rangle\rangle_{1}+3\sum_{g\in S_{4}}E^{k_{g(1)}+k_{g(2)}}\langle\langle E^{k_{g(3)}+k_{g(4)}}\rangle\rangle_{1},
\een
where $K=k_{1}+k_{2}+k_{3}+k_{4}$.
\end{lem}
\begin{proof}
It follows from Lemma \ref{G1-3E} and Corollary \ref{Virasorotype1}.
\end{proof}
\begin{rem}
By Corollary \ref{Virasorotype1}, Lemma \ref{G1-4E} is equivalent to Corollary 3.2 of \cite{L1}.
%And it is easy to show that
%combined with formula \eqref{Getzlertrr},
%Lemma \ref{G1-4E} implies Theorem 3.4 in \cite{L1}.
\end{rem}
\begin{lem}\label{G0-4E}For $k_{1},k_{2},k_{3},k_{4}\geq0$,
\ben
&&G_{0}(E^{k_{1}},E^{k_{2}},E^{k_{3}},E^{k_{4}})\\
&=&-36K\Phi_{K-1}+24\sum_{i=1}^4E^{k_{i}}\Phi_{K-k_{i}}-3\sum_{g\in S_{4}}E^{k_{g(1)}+k_{g(2)}}\Phi_{k_{g(3)}+k_{g(4)}},
\een
where $K=k_{1}+k_{2}+k_{3}+k_{4}$.
\end{lem}
\begin{proof}
It follows from Lemma \ref{G0-3E} and Corollary \ref{Virasorotype2}.
\end{proof}
\begin{rem}
It is easy to show that Lemma \ref{G0-4E} implies the definition of $\Phi_{k}$ ($k\geq3$) given in \cite{L1}.
\end{rem}
\begin{thm}For $k_{1},k_{2},k_{3},k_{4}\geq0$,
\ben
&&12K(\langle\langle E^{K-1}\rangle\rangle_{1}-\Phi_{K-1})\\
&=&8\sum_{i=1}^4E^{k_{i}}(\langle\langle E^{K-k_{i}}\rangle\rangle_{1}-\Phi_{K-k_{i}})-\sum_{g\in S_{4}}E^{k_{g(1)}+k_{g(2)}}(\langle\langle E^{k_{g(3)}+k_{g(4)}}\rangle\rangle_{1}-\Phi_{k_{g(3)}+k_{g(4)}}),
\een
where $K=k_{1}+k_{2}+k_{3}+k_{4}$.
\end{thm}
\begin{proof}
It follows  from Lemma \ref{G1-4E}, Lemma \ref{G0-4E} and equation \eqref{Getzlertrr}.
\end{proof}
\section{Application to the genus-1 Virasoro Conjecture}
In this section, we apply Getzler equations derived in Section 3 to study the genus-1 Virasoro conjecture.
Since Getzler equation $\textbf{G}(E^{k_{1}},E^{k_{2}},E^{k_{3}},E^{k_{4}})=0$ has been analyzed in great detail in \cite{L1}, we begin with the application of Getzler equation $\textbf{G}(E^{k_{1}},E^{k_{2}},E^{k_{3}},\gamma_{\alpha})=0$. It provides the following
relations among $\{\langle\langle E^{k}\rangle\rangle_{1}-\Phi_{k}\}_{k\geq0}$.
\begin{thm}\label{mainresult1}For $k\geq0$, and any $\alpha$,
\ben
\gamma_{\alpha}(\langle\langle E^{k}\rangle\rangle_{1}-\Phi_{k})=\frac{1}{2}k(k-1)\{E^{k-2}\circ\gamma_{\alpha}\}(\langle\langle E^{2}\rangle\rangle_{1}-\Phi_{2}).
\een
\end{thm}
\begin{proof}
It is trivial for $k=0,1,2.$ For $k\geq3$,
let $k_{1}=k-2$, $k_{2}=k_{3}=1$ in Theorem \ref{Getzler-3E}, we have
\bea
&&\gamma_{\alpha}\left(\langle\langle E^{k}\rangle\rangle_{1}-\Phi_{k}\right)\nonumber\\
&=&2\{E\circ\gamma_{\alpha}\}\left(\langle\langle E^{k-1}\rangle\rangle_{1}-\Phi_{k-1}\right)-\{E^{2}\circ\gamma_{\alpha}\}\left(\langle\langle E^{k-2}\rangle\rangle_{1}-\Phi_{k-2}\right)\nonumber\\
&&+\{E^{k-2}\circ\gamma_{\alpha}\}\left(\langle\langle E^2\rangle\rangle_{1}-\Phi_{2}\right).\label{recursion}
\eea
Suppose that the theorem holds for $k\leq n$ ($n\geq2$). By equation \eqref{recursion}, we have
\ben
&&\gamma_{\alpha}\left(\langle\langle E^{n+1}\rangle\rangle_{1}-\Phi_{n+1}\right)\\
&=&2\{E\circ\gamma_{\alpha}\}\left(\langle\langle E^{n}\rangle\rangle_{1}-\Phi_{n}\right)-\{E^{2}\circ\gamma_{\alpha}\}\left(\langle\langle E^{n-1}\rangle\rangle_{1}-\Phi_{n-1}\right)\\
&&+\{E^{n-1}\circ\gamma_{\alpha}\}\left(\langle\langle E^2\rangle\rangle_{1}-\Phi_{2}\right)\\
&=&n(n-1)\{E^{n-1}\circ\gamma_{\alpha}\}\left(\langle\langle E^2\rangle\rangle_{1}-\Phi_{2}\right)+\{E^{n-1}\circ\gamma_{\alpha}\}\left(\langle\langle E^2\rangle\rangle_{1}-\Phi_{2}\right)\\
&&-\frac{1}{2}(n-1)(n-2)\{E^{n-1}\circ\gamma_{\alpha}\}\left(\langle\langle E^2\rangle\rangle_{1}-\Phi_{2}\right)\\
&=&\frac{1}{2}n(n+1)\{E^{n-1}\circ\gamma_{\alpha}\}\left(\langle\langle E^2\rangle\rangle_{1}-\Phi_{2}\right).
\een
The proof is completed by induction on $k$.
\end{proof}
\begin{rem}
By Theorem \ref{mainresult1}, we have
\ben
&&\{\gamma_{\alpha}\circ\gamma_{\beta}\}(\langle\langle E^{\widehat{K}}\rangle\rangle_{1}-\Phi_{\widehat{K}})-\sum_{m=1}^2\{E^{\widehat{K}-k_{m}}\circ\gamma_{\alpha}\circ\gamma_{\beta}\}(\langle\langle E^{k_{m}}\rangle\rangle_{1}-\Phi_{k_{m}})\\
&=&\frac{2k_{1}k_{2}}{\widehat{K}(\widehat{K}-1)}\{\gamma_{\alpha}\circ\gamma_{\beta}\}(\langle\langle E^{\widehat{K}}\rangle\rangle_{1}-\Phi_{\widehat{K}})
\een
which gives an equivalent equation in Theorem \ref{Getzler-2E}.
\end{rem}
%\begin{rem}
%For any $k\geq2$, Theorem \ref{mainresult1} implies that the genus-1 Virasoro conjecture holds for any variety without quantum correction, i.e.
%$\upsilon_{1}\circ\upsilon_{2}=0$ for any two vector fields $\upsilon_{1},\upsilon_{2}$ on the small phase space.
%\end{rem}
It follows easily from Theorem \ref{mainresult1} that
\begin{cor}\label{mainresult2}For any $\alpha$,
\ben
&&\frac{m-k}{m+k}\gamma_{\alpha}(\langle\langle E^{m+k}\rangle\rangle_{1}-\Phi_{m+k})\\
&=&\{E^{k}\circ\gamma_{\alpha}\}(\langle\langle E^{m}\rangle\rangle_{1}-\Phi_{m})-\{E^{m}\circ\gamma_{\alpha}\}(\langle\langle E^{k}\rangle\rangle_{1}-\Phi_{k}).
\een
where $k+m>0$.
\end{cor}
\begin{cor}\label{mainresult3}For any $k\geq0$ and $m>0$ satisfying $m+k\geq2$, and any $\alpha$,
\ben
\{E^{k}\circ\gamma_{\alpha}\}\frac{\langle\langle E^{m}\rangle\rangle_{1}-\Phi_{m}}{m}=\frac{m-1}{(m+k)(m+k-1)}\gamma_{\alpha}(\langle\langle E^{m+k}\rangle\rangle_{1}-\Phi_{m+k}).
\een
\end{cor}
It is easy to show that Theorem \ref{mainresult1} is equivalent to Theorem \ref{Getzler-3E}. Due to the fact $\gamma_{1}(\langle\langle E^{k}\rangle\rangle_{1}-\Phi_{k})=k(\langle\langle E^{k-1}\rangle\rangle_{1}-\Phi_{k-1})$,
Corollary \ref{mainresult2} implies Virasoro type relation for $\{\Phi_{k}\}$ in Theorem 6.1 of \cite{L1} by setting $\gamma_{\alpha}=\gamma_{1}$ and Corollary \ref{Virasorotype1}, while Corollary \ref{mainresult3} implies Lemma 6.3 in \cite{L1} if $\gamma_{\alpha}=\gamma_{1}$.

Next, we deal with
more general case,
i.e., Getzler equation $\textbf{G}(E^{k_{1}},E^{k_{2}},\gamma_{\alpha},\gamma_{\beta})=0$. It involves more complicated genus-1 data and  genus-0 4-point functions. But
thanks to formulas \eqref{Observation1} and \eqref{Observation2}, we may consider the following universal equation
\ben
\textbf{G}(E^{k_{1}},E^{k_{2}},\gamma_{\mu},\gamma^{\mu}\circ\gamma_{\alpha})=0.
\een
Actually, we have

\begin{thm}\label{mainresult5}For $k_{1},k_{2}\geq0$, and any $\alpha$,
\ben
&&\{\Delta\circ\gamma_{\alpha}\}(\langle\langle E^{k_{1}+k_{2}}\rangle\rangle_{1}-\Phi_{k_{1}+k_{2}})\\
&=&\{\Delta\circ E^{k_{1}}\circ\gamma_{\alpha}\}(\langle\langle E^{k_{2}}\rangle\rangle_{1}-\Phi_{k_{2}})+\{\Delta\circ E^{k_{2}}\circ\gamma_{\alpha}\}(\langle\langle E^{k_{1}}\rangle\rangle_{1}-\Phi_{k_{1}}).
\een
\end{thm}
\begin{proof}
By Lemmas \ref{G1-2E}, \ref{G0-2E} and equations \eqref{Observation1}, \eqref{Observation2}, \eqref{simplication2}, it is easy to show that
\ben
&&G_{1}(E^{k_{1}},E^{k_{2}},\gamma_{\mu},\gamma^{\mu}\circ\gamma_{\alpha})\\
&=&24\{\Delta\circ\gamma_{\alpha}\}\langle\langle E^{k_{1}+k_{2}}\rangle\rangle_{1}-24\{\Delta\circ E^{k_{1}}\circ\gamma_{\alpha}\}\langle\langle E^{k_{2}}\rangle\rangle_{1}-24\{\Delta\circ E^{k_{2}}\circ\gamma_{\alpha}\}\langle\langle E^{k_{1}}\rangle\rangle_{1}
\een
and
\ben
&&G_{0}(E^{k_{1}},E^{k_{2}},\gamma_{\mu},\gamma^{\mu}\circ\gamma_{\alpha})\\
&=&-24\{\Delta\circ\gamma_{\alpha}\}\Phi_{k_{1}+k_{2}}+24\{\Delta\circ E^{k_{1}}\circ\gamma_{\alpha}\}\Phi_{k_{2}}+24\{\Delta\circ E^{k_{2}}\circ\gamma_{\alpha}\}\Phi_{k_{1}}.
\een
The proof is completed by equation \eqref{Getzlertrr}.
\end{proof}
It is equivalent to the following evidences for the genus-1 Virasoro conjecture.
\begin{thm}\label{mainresult6}
For any smooth projective variety $X$,
we have
\ben
(\Delta\circ\gamma_{\alpha})(\langle\langle E^k\rangle\rangle_{1}-\Phi_{k})=0,
\een
for all $k\geq0$ and $\alpha\in\{1,2,\ldots,N\}$.
\end{thm}
\begin{proof}
Let $k_{1}=k_{2}=1$ in Theorem \ref{mainresult5}, we have
\ben
(\Delta\circ\gamma_{\alpha})(\langle\langle E^2\rangle\rangle_{1}-\Phi_{2})=0.
\een
Set $k_{1}=k$ ($k\geq2$) and $k_{2}=1$ in Theorem \ref{mainresult5}, we have
\ben
(\Delta\circ\gamma_{\alpha})(\langle\langle E^{k+1}\rangle\rangle_{1}-\Phi_{k+1})=(\Delta\circ E\circ\gamma_{\alpha})(\langle\langle E^{k}\rangle\rangle_{1}-\Phi_{k}).
\een
Since $\alpha$ is arbitrary, the proof is completed.
\end{proof}
By setting $k=2$ and $\gamma_{\alpha}=\gamma_{1}$, Theorem \ref{mainresult6} implies Theorem 1.1 of \cite{L2}. Hence it provides more evidences for
the genus-1 Virasoro conjecture.
Since by the WDVV equation \eqref{WDVV1}, we have
\ben
\Delta\circ\gamma_{\alpha}=\sum_{\mu}\gamma_{\mu}\circ\gamma^{\mu}\circ\gamma_{\alpha}
&=&\sum\limits_{\sigma,\mu,\beta}\langle\langle\gamma_{\alpha}\gamma_{\beta}\gamma^{\sigma}\rangle\rangle_{0}
\langle\langle\gamma_{\sigma}\gamma_{\mu}\gamma^{\mu}\rangle\rangle_{0}\gamma^{\beta}\\
&=&\sum\limits_{\sigma,\mu,\beta}\langle\langle\gamma_{\alpha}\gamma^{\sigma}\gamma^{\mu}\rangle\rangle_{0}
\langle\langle\gamma_{\sigma}\gamma_{\mu}\gamma_{\beta}\rangle\rangle_{0}\gamma^{\beta}.
\een
Hence by Theorem \ref{mainresult6}, we have for any $k\geq2$
\ben
\left( \begin{array}{ccccccccc}
  A_{11}&A_{12}&\dots&A_{1N}\\
  A_{21}&A_{22}&\dots&A_{2N}\\
  \vdots&\vdots&   &\vdots\\
  A_{N1}&A_{N2}&\dots&A_{NN}\\
\end{array} \right)\left( \begin{array}{ccccccccc}
  \gamma^{1}(\langle\langle E^k\rangle\rangle_{1}-\Phi_{k})\\
  \gamma^{2}(\langle\langle E^k\rangle\rangle_{1}-\Phi_{k})\\
  \vdots\\
  \gamma^{N}(\langle\langle E^k\rangle\rangle_{1}-\Phi_{k})\\
\end{array} \right)=\left( \begin{array}{ccccccccc}
   0\\
   0\\
  \vdots\\
   0\\
\end{array} \right)
\een
where
\ben
A_{\alpha,\beta}=\sum\limits_{\sigma,\mu}\langle\langle\gamma_{\alpha}\gamma^{\sigma}\gamma^{\mu}\rangle\rangle_{0}
\langle\langle\gamma_{\sigma}\gamma_{\mu}\gamma_{\beta}\rangle\rangle_{0}.
\een
It is obviously that $A_{\alpha,\beta}=A_{\beta,\alpha}$.
If the following symmetric matrix
\ben
\left( \begin{array}{ccccccccc}
  A_{11}&A_{12}&\dots&A_{1N}\\
  A_{21}&A_{22}&\dots&A_{2N}\\
  \vdots&\vdots&   &\vdots\\
  A_{N1}&A_{N2}&\dots&A_{NN}\\
\end{array} \right)=\sum\limits_{\sigma,\mu}\left( \begin{array}{ccccccccc}
 \langle\langle\gamma_{1}\gamma^{\sigma}\gamma^{\mu}\rangle\rangle_{0}\\
  \langle\langle\gamma_{2}\gamma^{\sigma}\gamma^{\mu}\rangle\rangle_{0}\\
  \vdots\\
  \langle\langle\gamma_{N}\gamma^{\sigma}\gamma^{\mu}\rangle\rangle_{0}\\
\end{array} \right)\left( \begin{array}{ccccccccc}
 \langle\langle\gamma_{\sigma}\gamma_{\mu}\gamma_{1}\rangle\rangle_{0}\\
  \langle\langle\gamma_{\sigma}\gamma_{\mu}\gamma_{2}\rangle\rangle_{0}\\
  \vdots\\
  \langle\langle\gamma_{\sigma}\gamma_{\mu}\gamma_{N}\rangle\rangle_{0}\\
\end{array} \right)^{T}
\een
is invertible, then $\gamma^{\alpha}(\langle\langle E^k\rangle\rangle_{1}-\Phi_{k})=0$ for any  $\alpha\in\{1,2,\cdot\cdot\cdot,N\}$, which
is equivalent to $\gamma_{\alpha}(\langle\langle E^k\rangle\rangle_{1}-\Phi_{k})=0$ for any  $\alpha\in\{1,2,\cdot\cdot\cdot,N\}$, and then
the genus-1 Virasoro conjecture holds.\par
The above argument gives an alternative proof of
\begin{cor}[\cite{DZ2,L1,L3,Tel}]\label{mainresult61}
For any compact symplectic manifold with semisimple quantum cohomology,
the genus-1 Virasoro conjecture holds.
\end{cor}
\begin{proof}
Follow the notation in \cite{L3} and restrict everything on the small phase space. Let $\{\mathcal{E}_{i}\}$ be the idempotents which span the space of primary vector fields on the small phase space.
Recall the following formulas derived in \cite{L3}:
\ben
\mathcal{E}_{i}\circ\mathcal{E}_{j}=\delta_{ij}\mathcal{E}_{i};\mbox{       } \gamma_{\alpha}=\sum_{i=1}^N\psi_{i\alpha}\frac{\mathcal{E}_{i}}{\sqrt{g_{i}}};\mbox{       }\Delta=\sum_{i=1}^N\frac{1}{g_{i}}\mathcal{E}_{i},
\een
where $g_{\alpha}=\|\mathcal{E}_{\alpha}\|^2$ and $(\psi_{\alpha\beta})$ is a invertible matrix (see \cite{L3} for more detail).
Hence, we have
\ben
&&\Delta\circ\gamma_{\alpha}=\sum_{i=1}^Ng_{i}^{-\frac{3}{2}}\psi_{i\alpha}\mathcal{E}_{i}.
\een
By Theorem \ref{mainresult6}, we have for any $k\geq2$
\ben
\left( \begin{array}{ccccccccc}
  g_{1}^{-\frac{3}{2}}\psi_{11}&g_{2}^{-\frac{3}{2}}\psi_{21}&\dots&g_{N}^{-\frac{3}{2}}\psi_{N1}\\
  g_{1}^{-\frac{3}{2}}\psi_{12}&g_{2}^{-\frac{3}{2}}\psi_{22}&\dots&g_{N}^{-\frac{3}{2}}\psi_{N2}\\
  \vdots&\vdots&   &\vdots\\
 g_{1}^{-\frac{3}{2}}\psi_{1N}&g_{2}^{-\frac{3}{2}}\psi_{2N}&\dots&g_{N}^{-\frac{3}{2}}\psi_{NN}\\
\end{array} \right)\left( \begin{array}{ccccccccc}
  \mathcal{E}_{1}(\langle\langle E^k\rangle\rangle_{1}-\Phi_{k})\\
  \mathcal{E}_{2}(\langle\langle E^k\rangle\rangle_{1}-\Phi_{k})\\
  \vdots\\
  \mathcal{E}_{N}(\langle\langle E^k\rangle\rangle_{1}-\Phi_{k})\\
\end{array} \right)=\left( \begin{array}{ccccccccc}
   0\\
   0\\
  \vdots\\
   0\\
\end{array} \right)
\een
The proof is completed due to the fact that the matrix $(\psi_{\alpha\beta})$ is invertible.
\end{proof}
%\subsection{Application of universal equation $\textbf{G}(E^{k},\gamma_{\alpha},\gamma_{\beta},\gamma_{\sigma})=0$}
Finally, we obtain one new relation from
 Getzler equation $\textbf{G}(E^{k},\gamma_{\alpha},\gamma_{\beta},\gamma_{\sigma})=0$.
The equation is $\textbf{G}(E^k,\gamma_{\alpha},\gamma^{\alpha}\circ\gamma^{\beta}\circ\gamma_{\mu},\gamma_{\beta})=0$  since in this case, using
equations \eqref{Observation1}, \eqref{Observation2} and \eqref{simplication2}, we have
\be\label{G1-1E1}
G_{1}(E^k,\gamma_{\alpha},\gamma^{\alpha}\circ\gamma^{\beta}\circ\gamma_{\mu},\gamma_{\beta})=-24\{\Delta^2\circ\gamma_{\mu}\}\langle\langle E^k\rangle\rangle_{1}+24k\langle\langle\Delta^2\circ E^{k-1}\circ\gamma_{\mu}\rangle\rangle_{1}
\ee
by Lemma \ref{G1-1E}. Actually, we have

\begin{thm}\label{mainresult7}
For all $k\geq1$ and any $\mu$,
\ben
&&24\langle\langle\Delta^2\circ E^{k-1}\circ\gamma_{\mu}\rangle\rangle_{1}\\
&=&5\langle\langle\{\Delta\circ E^{k-1}\}\gamma_{\alpha}\gamma^{\alpha}\{\Delta\circ\gamma_{\mu}\}\rangle\rangle_{0}\\
&&+2\langle\langle\{\Delta\circ E^{k-1}\}\gamma_{\alpha}\{\gamma^{\alpha}\circ\gamma_{\mu}\}\Delta\rangle\rangle_{0}\\
&&-6\langle\langle\{\Delta\circ E^{k-1}\}G(\gamma^{\alpha}\circ\gamma^{\beta}\circ\gamma_{\mu})\gamma_{\alpha}\gamma_{\beta}\rangle\rangle_{0}\\
&&-6\langle\langle\{\Delta\circ E^{k-1}\circ\gamma_{\mu}\}G(\gamma^{\alpha}\circ\gamma^{\beta})\gamma_{\alpha}\gamma_{\beta}\rangle\rangle_{0}\\
&&-6\sum_{i=1}^{k-1}\langle\langle\{G(\Delta\circ E^{i-1}\circ\gamma^{\alpha})\circ G(\gamma_{\alpha})\} E^{k-i-1}\{\Delta\circ\gamma_{\mu}\}\rangle\rangle_{0}\\
&&+4\sum_{i=1}^{k-1}\langle\langle G(\Delta\circ E^{i-1})E^{k-i-1}\{\Delta^2\circ\gamma_{\mu}\}\rangle\rangle_{0}\\
&&-3\sum_{i=1}^{k-1}\langle\langle G(\Delta\circ E^{i-1}\circ\gamma_{\mu})E^{k-i-1}\Delta^2\rangle\rangle_{0}\\
&&+(k-1)\langle\langle\Delta^3E^{k-2}\gamma_{\mu}\rangle\rangle_{0}.
\een
\end{thm}
\begin{proof}
By Lemma \ref{G0-1E}, and using formulas \eqref{Observation1}, \eqref{Observation2}, \eqref{WDVV2} and Lemmas \ref{4point-3point}, \ref{simplication}, it
can be verified that
\bea
&&G_{0}(E^k,\gamma_{\alpha},\gamma^{\alpha}\circ\gamma^{\beta}\circ\gamma_{\mu},\gamma_{\beta})\nonumber\\
&=&24\{\Delta^2\circ\gamma_{\mu}\}\Phi_{k}-5k\langle\langle\{\Delta\circ E^{k-1}\}\gamma_{\alpha}\gamma^{\alpha}\{\Delta\circ\gamma_{\mu}\}\rangle\rangle_{0}\nonumber\\
&&-2k\langle\langle\{\Delta\circ E^{k-1}\}\gamma_{\alpha}\{\gamma^{\alpha}\circ\gamma_{\mu}\}\Delta\rangle\rangle_{0}\nonumber\\
&&+6k\langle\langle\{\Delta\circ E^{k-1}\}G(\gamma^{\alpha}\circ\gamma^{\beta}\circ\gamma_{\mu})\gamma_{\alpha}\gamma_{\beta}\rangle\rangle_{0}\nonumber\\
&&+6k\langle\langle\{\Delta\circ E^{k-1}\circ\gamma_{\mu}\}G(\gamma^{\alpha}\circ\gamma^{\beta})\gamma_{\alpha}\gamma_{\beta}\rangle\rangle_{0}\nonumber\\
&&+6k\sum_{i=1}^{k-1}\langle\langle\{G(\Delta\circ E^{i-1}\circ\gamma^{\alpha})\circ G(\gamma_{\alpha})\} E^{k-i-1}\{\Delta\circ\gamma_{\mu}\}\rangle\rangle_{0}\nonumber\\
&&-4k\sum_{i=1}^{k-1}\langle\langle G(\Delta\circ E^{i-1})E^{k-i-1}\{\Delta^2\circ\gamma_{\mu}\}\rangle\rangle_{0}\nonumber\\
&&+3k\sum_{i=1}^{k-1}\langle\langle G(\Delta\circ E^{i-1}\circ\gamma_{\mu})E^{k-i-1}\Delta^2\rangle\rangle_{0}\nonumber\\
&&-k(k-1)\langle\langle\Delta^3E^{k-2}\gamma_{\mu}\rangle\rangle_{0}.\label{G0-1E1}
\eea
By Theorem \ref{mainresult6}, we have
\ben
\{\Delta^2\circ\gamma_{\mu}\}(\langle\langle E^k\rangle\rangle_{1}-\Phi_{k})=0.
\een
Together with equations \eqref{G1-1E1}, \eqref{G0-1E1} and \eqref{Getzlertrr}, the proof is completed.
\end{proof}
In particular, we have
\begin{cor}
\ben
\langle\langle\Delta^2\rangle\rangle_{1}=\frac{7}{24}\langle\langle\Delta\gamma_{\mu}\gamma^{\mu}\Delta\rangle\rangle_{0}
-\frac{1}{2}\langle\langle\Delta G(\gamma^{\alpha}\circ\gamma^{\beta})\gamma_{\alpha}\gamma_{\beta}\rangle\rangle_{0}.
\een
\end{cor}
\begin{rem}
Theorem \ref{mainresult7} can be also obtained from
\ben
\textbf{G}(E^{k},\gamma_{\alpha},\gamma^{\alpha},\Delta\circ\gamma_{\mu})+\textbf{G}(E^{k},\gamma_{\alpha},\gamma^{\alpha}\circ\gamma_{\mu},\Delta)=0.
\een
\end{rem}
\rule{0pt}{0pt}

\appendix
\section{}
%----------The first stage for computing $G_{0}$
  In this appendix, we present the following results which are used in
Lemma \ref{G0-1E}. Notice that $\{\alpha,\beta,\sigma\}=\{\varsigma_{1},\varsigma_{2},\varsigma_{3}\}$ is implicit below.
We start with the computation of each term on the right-hand side of  equation \eqref{G0formula1} as follows.
%\begin{lem}\label{G0simplication}
%For any ring $\mathcal{R}$ and any map $\varphi:\mathbb{Z}^{3}\rightarrow\mathcal{R}$, we have
%\ben
%&&\sum_{i=1}^k\sum_{j=1}^{i-1}\varphi(k-i,j-1,i-j-1)=\sum_{i=1}^k\sum_{j=1}^{i-1}\varphi(k-i,i-j-1,j-1)\\
%&=&\sum_{i=1}^k\sum_{j=1}^{i-1}\varphi(j-1,k-i,i-j-1)=\sum_{i=1}^k\sum_{j=1}^{i-1}\varphi(j-1,i-j-1,k-i)\\
%&=&\sum_{i=1}^k\sum_{j=1}^{i-1}\varphi(i-j-1,j-1,k-i)=\sum_{i=1}^k\sum_{j=1}^{i-1}\varphi(i-j-1,k-i,j-1)\\
%&=&\sum_{i=1}^{k-1}\sum_{j=1}^{k-i}\varphi(k-i-j,i-1,j-1)=\sum_{i=1}^{k-1}\sum_{j=1}^{k-i}\varphi(k-i-j,j-1,i-1)\\
%&=&\sum_{i=1}^{k-1}\sum_{j=1}^{k-i}\varphi(i-1,k-i-j,j-1)=\sum_{i=1}^{k-1}\sum_{j=1}^{k-i}\varphi(i-1,j-1,k-i-j)\\
%&=&\sum_{i=1}^{k-1}\sum_{j=1}^{k-i}\varphi(j-1,k-i-j,i-1)=\sum_{i=1}^{k-1}\sum_{j=1}^{k-i}\varphi(j-1,i-1,k-i-j).
%\een
%\end{lem}
%\begin{proof}
%The proof is straightforward.
%\end{proof}
%With this lemma, we have the following results.
\begin{lem}\label{G0part1}For any $\alpha,\beta,\sigma$,
\ben
&&\langle\langle E^k\gamma_{\alpha}\gamma_{\beta}\gamma_{\sigma}\Delta\rangle\rangle_{0}\\
&=&-\frac{1}{6}\sum_{g\in S_{3}}\sum_{i=1}^{k-1}\langle\langle\{G(E^{k-i})\circ\Delta\}\gamma_{\varsigma_{g(1)}}\gamma_{\varsigma_{g(2)}}\{E^{i-1}\circ\gamma_{\varsigma_{g(3)}}\}\rangle\rangle_{0}\\
&&+\frac{1}{6}\sum_{g\in S_{3}}\sum_{i=1}^{k}\langle\langle\{G(\Delta\circ E^{k-i})\}\gamma_{\varsigma_{g(1)}}\gamma_{\varsigma_{g(2)}}\{E^{i-1}\circ\gamma_{\varsigma_{g(3)}}\}\rangle\rangle_{0}\\
&&-\frac{1}{3}b_{1}\sum_{g\in S_{3}}\langle\langle \Delta\gamma_{\varsigma_{g(1)}}\gamma_{\varsigma_{g(2)}}\{E^{k-1}\circ\gamma_{\varsigma_{g(3)}}\}\rangle\rangle_{0}\\
&&-\frac{1}{3}\sum_{g\in S_{3}}\sum_{i=1}^{k}\langle\langle\{\Delta\circ E^{k-i}\}\gamma_{\varsigma_{g(1)}}\gamma_{\varsigma_{g(2)}}\{E^{i-1}\circ\gamma_{\varsigma_{g(3)}}\}\rangle\rangle_{0}\\
&&-\frac{1}{6}\sum_{g\in S_{3}}\sum_{i=1}^{k-1}\langle\langle G(E^{k-i})\Delta \gamma_{\varsigma_{g(1)}}\{E^{i-1}\circ\gamma_{\varsigma_{g(2)}}\circ\gamma_{\varsigma_{g(3)}}\}\rangle\rangle_{0}\\
&&+\frac{1}{2}\sum_{g\in S_{3}}\sum_{i=1}^{k-1}\langle\langle\{\Delta\circ E^{i-1}\}\gamma_{\varsigma_{g(1)}}\gamma_{\varsigma_{g(2)}}G(E^{k-i}\circ\gamma_{\varsigma_{g(3)}})\rangle\rangle_{0}\\
&&+\frac{1}{3}\sum_{g\in S_{3}}\langle\langle\{\Delta\circ E^{k-1}\}\gamma_{\varsigma_{g(1)}}\gamma_{\varsigma_{g(2)}}G(\gamma_{\varsigma_{g(3)}})\rangle\rangle_{0}\\
&&+\frac{1}{6}\sum_{g\in S_{3}}\langle\langle\Delta\gamma_{\varsigma_{g(1)}} \{E^{k-1}\circ\gamma_{\varsigma_{g(2)}}\}G(\gamma_{\varsigma_{g(3)}})\rangle\rangle_{0}\\
&&-\frac{1}{6}\sum_{g\in S_{3}}\sum_{i=1}^{k-1}\langle\langle\{\Delta\circ E^{i-1}\}\gamma_{\varsigma_{g(1)}}\gamma_{\varsigma_{g(2)}}\{G(E^{k-i})\circ\gamma_{\varsigma_{g(3)}}\}\rangle\rangle_{0}\\
&&+\frac{1}{6}\sum_{g\in S_{3}}\sum_{i=1}^{k}\sum_{j=1}^{i-1}\langle\langle\{G(E^{k-i})\circ\Delta\circ E^{j-1}\}G(E^{i-j-1}\circ\gamma_{\varsigma_{g(1)}}\circ\gamma_{\varsigma_{g(2)}})\gamma_{\varsigma_{g(3)}}\rangle\rangle_{0}\\
&&+\frac{1}{6}b_{1}\sum_{g\in S_{3}}\sum_{i=1}^{k-1}\langle\langle\{\Delta\circ E^{k-i-1}\}G(E^{i-1}\circ\gamma_{\varsigma_{g(1)}}\circ\gamma_{\varsigma_{g(2)}})\gamma_{\varsigma_{g(3)}}\rangle\rangle_{0}\\
&&+\frac{1}{2}\sum_{g\in S_{3}}\sum_{i=1}^{k}\sum_{j=1}^{i-1}\langle\langle E^{k-i}G(E^{j-1}\circ\gamma_{\varsigma_{g(1)}}\circ\gamma_{\varsigma_{g(2)}})G(\Delta\circ E^{i-j-1}\circ\gamma_{\varsigma_{g(3)}})\rangle\rangle_{0}\\
&&-\frac{1}{6}\sum_{g\in S_{3}}\sum_{i=1}^{k-1}\langle\langle \{\Delta\circ E^{i-1}\}G( E^{k-i-1}\circ\gamma_{\varsigma_{g(1)}}\circ \gamma_{\varsigma_{g(2)}})G(\gamma_{\varsigma_{g(3)}})\rangle\rangle_{0}\\
&&-\frac{1}{3}\sum_{g\in S_{3}}\sum_{i=1}^{k}\sum_{j=1}^{i-1}\langle\langle \{G(\Delta\circ E^{k-i})\circ E^{j-1}\}G(E^{i-j-1}\circ\gamma_{\varsigma_{g(1)}}\circ\gamma_{\varsigma_{g(2)}})\gamma_{\varsigma_{g(3)}}\rangle\rangle_{0}\\
&&-2\sum_{i=1}^{k-1}\sum_{j=1}^{k-i}\langle\langle\{G(E^{k-i-j})\circ E^{i+j-2}\}\Delta\{\gamma_{\alpha}\circ\gamma_{\beta}\circ\gamma_{\sigma}\}\rangle\rangle_{0}\\
&&-\frac{1}{6}(b_{1}+2)\sum_{g\in S_{3}}\sum_{i=1}^{k-1}\sum_{j=1}^{k-i}\langle\langle\{\Delta\circ E^{k-j-1}\}G(E^{j-1}\circ\gamma_{\varsigma_{g(1)}}\circ\gamma_{\varsigma_{g(2)}})\gamma_{\varsigma_{g(3)}}\rangle\rangle_{0}\\
&&+\frac{1}{6}(b_{1}+4)\sum_{g\in S_{3}}\sum_{i=1}^{k-1}\sum_{j=1}^{k-i}\langle\langle\{\Delta\circ E^{k-j-1}\}G(E^{j-1}\circ\gamma_{\varsigma_{g(1)}})\{\gamma_{\varsigma_{g(2)}}\circ\gamma_{\varsigma_{g(3)}}\}\rangle\rangle_{0}\\
&&-\frac{1}{6}\sum_{g\in S_{3}}\sum_{i=1}^{k}\sum_{j=1}^{i-1}\langle\langle\{G(E^{k-i})\circ\Delta\circ E^{j-1}\}G(E^{i-j-1}\circ\gamma_{\varsigma_{g(1)}})\{\gamma_{\varsigma_{g(2)}}\circ\gamma_{\varsigma_{g(3)}}\}\rangle\rangle_{0}\\
&&-\frac{1}{6}b_{1}\sum_{g\in S_{3}}\sum_{i=1}^{k-1}\langle\langle\{\Delta\circ E^{i-1}\}G(E^{k-i-1}\circ\gamma_{\varsigma_{g(1)}})\{\gamma_{\varsigma_{g(2)}}\circ\gamma_{\varsigma_{g(3)}}\}\rangle\rangle_{0}\\
&&-\sum_{g\in S_{3}}\sum_{i=1}^{k}\sum_{j=1}^{i-1}\langle\langle\{E^{k-i}\circ \gamma_{\varsigma_{g(1)}}\}G(\Delta\circ E^{j-1}\circ\gamma_{\varsigma_{g(2)}})G(E^{i-j-1}\circ\gamma_{\varsigma_{g(3)}})\rangle\rangle_{0}\\
&&+\frac{1}{6}\sum_{g\in S_{3}}\sum_{i=1}^{k-1}\langle\langle\{E^{i-1}\circ \gamma_{\varsigma_{g(1)}}\}G(\Delta\circ E^{k-i-1}\circ\gamma_{\varsigma_{g(2)}})G(\gamma_{\varsigma_{g(3)}})\rangle\rangle_{0}\\
&&+\frac{2}{3}\sum_{g\in S_{3}}\sum_{i=1}^{k}\sum_{j=1}^{i-1}\langle\langle\{G(\Delta\circ E^{k-i})\circ E^{j-1}\}G(E^{i-j-1}\circ\gamma_{\varsigma_{g(1)}})\{\gamma_{\varsigma_{g(2)}}\circ\gamma_{\varsigma_{g(3)}}\}\rangle\rangle_{0}\\
&&-\frac{1}{6}\sum_{g\in S_{3}}\sum_{i=1}^{k-1}\langle\langle\{G(\Delta\circ E^{i-1})\circ E^{k-i-1}\}\{\gamma_{\varsigma_{g(1)}}\circ\gamma_{\varsigma_{g(2)}}\}G(\gamma_{\varsigma_{g(3)}})\rangle\rangle_{0}\\
&&+\frac{1}{6}\sum_{g\in S_{3}}\sum_{i=1}^{k-1}\langle\langle\{\Delta\circ E^{k-i-1}\circ \gamma_{\varsigma_{g(1)}}\}G(E^{i-1}\circ\gamma_{\varsigma_{g(2)}})G(\gamma_{\varsigma_{g(3)}})\rangle\rangle_{0}\\
&&-\frac{1}{6}b_{1}\sum_{g\in S_{3}}\sum_{i=1}^{k-1}\sum_{j=1}^{k-i}\langle\langle E^{k-i-j}G(\Delta\circ E^{i+j-2}\circ\gamma_{\varsigma_{g(1)}})\{\gamma_{\varsigma_{g(2)}}\circ \gamma_{\varsigma_{g(3)}}\}\rangle\rangle_{0}\\
&&+\frac{1}{3}\sum_{g\in S_{3}}\sum_{i=1}^{k}\sum_{j=1}^{i-1}\langle\langle \{G(E^{k-i})\circ E^{j-1}\}G(\Delta\circ E^{i-j-1}\circ\gamma_{\varsigma_{g(1)}})\{\gamma_{\varsigma_{g(2)}}\circ \gamma_{\varsigma_{g(3)}}\}\rangle\rangle_{0}\\
&&-\frac{1}{6}b_{1}\sum_{g\in S_{3}}\sum_{i=1}^{k-1}\langle\langle  E^{k-i-1}G(\Delta\circ E^{i-1}\circ\gamma_{\varsigma_{g(1)}})\{\gamma_{\varsigma_{g(2)}}\circ \gamma_{\varsigma_{g(3)}}\}\rangle\rangle_{0}\\
&&-\sum_{i=1}^{k-1}\sum_{j=1}^{i-1}\langle\langle\{G(\Delta\circ E^{i-j-1})\circ G(E^{k-i})\}E^{j-1}\{\gamma_{\alpha}\circ\gamma_{\beta}\circ\gamma_{\sigma}\}\rangle\rangle_{0}\\
&&+\frac{1}{6}b_{1}\sum_{g\in S_{3}}\sum_{i=1}^{k-1}\sum_{j=1}^{k-i}\langle\langle E^{k-i-j}G(\Delta\circ\gamma_{\varsigma_{g(1)}}\circ\gamma_{\varsigma_{g(2)}}\circ E^{i+j-2})\gamma_{\varsigma_{g(3)}}\rangle\rangle_{0},
\een
where $k\geq0$ and $\{\varsigma_{1},\varsigma_{2},\varsigma_{3}\}=\{\alpha,\beta,\sigma\}$.
\end{lem}
\begin{proof}
By Lemma \ref{WDVV4}, we have
\ben
&&\langle\langle E^k\gamma_{\alpha}\gamma_{\beta}\gamma_{\sigma}\Delta\rangle\rangle_{0}\\
&=&-\sum_{i=1}^{k-1}\langle\langle EE^{k-i}\Delta \gamma^{\rho}\rangle\rangle_{0}
\langle\langle \gamma_{\rho}\gamma_{\alpha}\{\gamma_{\beta}\circ E^{i-1}\}\gamma_{\sigma}\rangle\rangle_{0}\\
&&-\sum_{i=1}^{k-1}\langle\langle EE^{k-i}
\{\gamma_{\alpha}\circ\gamma_{\beta}\circ E^{i-1}\}\gamma_{\sigma}\Delta\rangle\rangle_{0}\\
&&-\sum_{i=1}^{k-1}\langle\langle EE^{k-i}\gamma_{\sigma}\gamma^{\rho}\rangle\rangle_{0}
\langle\langle \gamma_{\rho}\gamma_{\alpha}\{\gamma_{\beta}\circ E^{i-1}\}\Delta\rangle\rangle_{0}\\
&&+\sum_{i=1}^{k-1}\langle\langle E^{k-i}\gamma_{\alpha}\Delta\gamma^{\rho}\rangle\rangle_{0}
\langle\langle \gamma_{\rho}\{\gamma_{\beta}\circ E^{i-1}\}E\gamma_{\sigma}\rangle\rangle_{0}\\
&&+\sum_{i=1}^{k-1}\langle\langle E^{k-i}\gamma_{\alpha}\gamma_{\sigma}\gamma^{\rho}\rangle\rangle_{0}
\langle\langle \gamma_{\rho}\{\gamma_{\beta}\circ E^{i-1}\}E\Delta\rangle\rangle_{0}\\
&&+\sum_{i=1}^{k}\langle\langle\{E^{k-i}
\circ\gamma_{\alpha}\}\{\gamma_{\beta}\circ E^{i-1}\}E\gamma_{\sigma}\Delta\rangle\rangle_{0}.
\een
We compute each term of right-hand side of the above equation as follows.
By equation \eqref{4pointeq1}, we have
\ben
&&\langle\langle EE^{k-i}\Delta \gamma^{\rho}\rangle\rangle_{0}
\langle\langle \gamma_{\rho}\gamma_{\alpha}\{\gamma_{\beta}\circ E^{i-1}\}\gamma_{\sigma}\rangle\rangle_{0}\\
&=&\langle\langle\{G(E^{k-i})\circ\Delta\}\gamma_{\alpha}\gamma_{\sigma}\{E^{i-1}\circ\gamma_{\beta}\}\rangle\rangle_{0}
+\langle\langle\{G(\Delta)\circ E^{k-i}\}\gamma_{\alpha}\gamma_{\sigma}\{E^{i-1}\circ\gamma_{\beta}\}\rangle\rangle_{0}\\
&&-\langle\langle G(\Delta\circ E^{k-i})\gamma_{\alpha}\gamma_{\sigma}\{E^{i-1}\circ\gamma_{\beta}\}\rangle\rangle_{0}
-b_{1}\langle\langle \{\Delta\circ E^{k-i}\}\gamma_{\alpha}\gamma_{\sigma}\{E^{i-1}\circ\gamma_{\beta}\}\rangle\rangle_{0},
\een
\ben
&&\langle\langle EE^{k-i}\gamma_{\sigma}\gamma^{\rho}\rangle\rangle_{0}
\langle\langle \gamma_{\rho}\gamma_{\alpha}\{\gamma_{\beta}\circ E^{i-1}\}\Delta\rangle\rangle_{0}\\
&=&\langle\langle\{G(E^{k-i})\circ\gamma_{\sigma}\}\gamma_{\alpha}\{E^{i-1}\circ\gamma_{\beta}\}\Delta\rangle\rangle_{0}
+\langle\langle\{E^{k-i}\circ G(\gamma_{\sigma}) \}\gamma_{\alpha}\{E^{i-1}\circ\gamma_{\beta}\}\Delta\rangle\rangle_{0}\\
&&-\langle\langle G(E^{k-i}\circ\gamma_{\sigma})\gamma_{\alpha}\{E^{i-1}\circ\gamma_{\beta}\}\Delta\rangle\rangle_{0}
-b_{1}\langle\langle \{E^{k-i}\circ\gamma_{\sigma}\}\gamma_{\alpha}\{E^{i-1}\circ\gamma_{\beta}\}\Delta\rangle\rangle_{0},
\een
\ben
&&\langle\langle E^{k-i}\gamma_{\alpha}\Delta\gamma^{\rho}\rangle\rangle_{0}
\langle\langle \gamma_{\rho}\{\gamma_{\beta}\circ E^{i-1}\}E\gamma_{\sigma}\rangle\rangle_{0}\\
&=&\langle\langle E^{k-i}\Delta\gamma_{\alpha}\{G(E^{i-1}\circ\gamma_{\beta})\circ\gamma_{\sigma}\}\rangle\rangle_{0}
+\langle\langle E^{k-i}\Delta\gamma_{\alpha}\{G(\gamma_{\sigma})\circ E^{i-1}\circ\gamma_{\beta}\}\rangle\rangle_{0}\\
&&-\langle\langle E^{k-i}\Delta\gamma_{\alpha}G(E^{i-1}\circ\gamma_{\beta}\circ\gamma_{\sigma})\rangle\rangle_{0}
-b_{1}\langle\langle E^{k-i}\Delta\gamma_{\alpha}\{E^{i-1}\circ\gamma_{\beta}\circ\gamma_{\sigma}\}\rangle\rangle_{0},
\een
\ben
&&\langle\langle E^{k-i}\gamma_{\alpha}\gamma_{\sigma}\gamma^{\rho}\rangle\rangle_{0}
\langle\langle \gamma_{\rho}\{\gamma_{\beta}\circ E^{i-1}\}E\Delta\rangle\rangle_{0}\\
&=&\langle\langle E^{k-i}\gamma_{\alpha}\gamma_{\sigma}\{G(E^{i-1}\circ\gamma_{\beta})\circ\Delta\}\rangle\rangle_{0}
+\langle\langle E^{k-i}\gamma_{\alpha}\gamma_{\sigma}\{G(\Delta)\circ E^{i-1}\circ\gamma_{\beta}\}\rangle\rangle_{0}\\
&&-\langle\langle E^{k-i}\gamma_{\alpha}\gamma_{\sigma}G(E^{i-1}\circ\gamma_{\beta}\circ\Delta)\rangle\rangle_{0}
-b_{1}\langle\langle E^{k-i}\gamma_{\alpha}\gamma_{\sigma}\{E^{i-1}\circ\gamma_{\beta}\circ\Delta\}\rangle\rangle_{0}.
\een
By equation \eqref{quasihom}, we have
\ben
&&\langle\langle EE^{k-i}
\{\gamma_{\alpha}\circ\gamma_{\beta}\circ E^{i-1}\}\gamma_{\sigma}\Delta\rangle\rangle_{0}\\
&=&\langle\langle G(E^{k-i})\Delta\gamma_{\sigma}\{E^{i-1}\circ\gamma_{\alpha}\circ\gamma_{\beta}\}\rangle\rangle_{0}
+\langle\langle E^{k-i}\Delta\gamma_{\sigma}G(E^{i-1}\circ\gamma_{\alpha}\circ\gamma_{\beta})\rangle\rangle_{0}\\
&&+\langle\langle E^{k-i}\Delta\{E^{i-1}\circ\gamma_{\alpha}\circ\gamma_{\beta}\}G(\gamma_{\sigma})\rangle\rangle_{0}
+\langle\langle E^{k-i}\gamma_{\sigma}\{E^{i-1}\circ\gamma_{\alpha}\circ\gamma_{\beta}\}G(\Delta)\rangle\rangle_{0}\\
&&-2(b_{1}+1)\langle\langle E^{k-i}\gamma_{\sigma}\{E^{i-1}\circ\gamma_{\alpha}\circ\gamma_{\beta}\}\Delta\rangle\rangle_{0}
\een
and
\ben
&&\langle\langle\{E^{k-i}
\circ\gamma_{\alpha}\}\{\gamma_{\beta}\circ E^{i-1}\}E\gamma_{\sigma}\Delta\rangle\rangle_{0}\\
&=&\langle\langle G(E^{k-i}
\circ\gamma_{\alpha})\{\gamma_{\beta}\circ E^{i-1}\}\gamma_{\sigma}\Delta\rangle\rangle_{0}
+\langle\langle\{E^{k-i}
\circ\gamma_{\alpha}\}G(\gamma_{\beta}\circ E^{i-1})\gamma_{\sigma}\Delta\rangle\rangle_{0}\\
&&+\langle\langle\{E^{k-i}
\circ\gamma_{\alpha}\}\{\gamma_{\beta}\circ E^{i-1}\}G(\gamma_{\sigma})\Delta\rangle\rangle_{0}
+\langle\langle\{E^{k-i}
\circ\gamma_{\alpha}\}\{\gamma_{\beta}\circ E^{i-1}\}\gamma_{\sigma}G(\Delta)\rangle\rangle_{0}\\
&&-2(b_{1}+1)\langle\langle\{E^{k-i}
\circ\gamma_{\alpha}\}\{\gamma_{\beta}\circ E^{i-1}\}\gamma_{\sigma}\Delta\rangle\rangle_{0}.
\een
Next, using equation \eqref{WDVV2} to change the type of the following 4-point functions
\ben
&&\langle\langle\{G(E^{k-i})\circ\gamma_{\sigma}\}\gamma_{\alpha}\{E^{i-1}\circ\gamma_{\beta}\}\Delta\rangle\rangle_{0}\\
&=&\langle\langle\{\Delta\circ E^{i-1}\}\gamma_{\alpha}\gamma_{\beta}\{G(E^{k-i})\circ\gamma_{\sigma}\}\rangle\rangle_{0}
+\langle\langle E^{i-1}\Delta\{\gamma_{\alpha}\circ\gamma_{\beta}\}\{G(E^{k-i})\circ\gamma_{\sigma}\}\rangle\rangle_{0}\\
&&-\langle\langle E^{i-1}\gamma_{\beta}\{\Delta\circ\gamma_{\alpha}\}\{G(E^{k-i})\circ\gamma_{\sigma}\}\rangle\rangle_{0},
\een
\ben
&&\langle\langle\{G(\gamma_{\sigma})\circ E^{k-i}\}\gamma_{\alpha}\{E^{i-1}\circ\gamma_{\beta}\}\Delta\rangle\rangle_{0}\\
&=&\langle\langle \{\Delta\circ E^{k-i}\}\{E^{i-1}\circ\gamma_{\beta}\}\gamma_{\alpha} G(\gamma_{\sigma})\rangle\rangle_{0}
+\langle\langle E^{k-i}\Delta\{E^{i-1}\circ\gamma_{\beta}\}\{\gamma_{\alpha}\circ G(\gamma_{\sigma})\}\rangle\rangle_{0}\\
&&-\langle\langle E^{k-i}\{E^{i-1}\circ\gamma_{\beta}\}\{\Delta\circ\gamma_{\alpha}\}G(\gamma_{\sigma})\rangle\rangle_{0},
\een
\ben
&&\langle\langle G(E^{k-i}\circ\gamma_{\sigma})\gamma_{\alpha}\{E^{i-1}\circ\gamma_{\beta}\}\Delta\rangle\rangle_{0}\\
&=&\langle\langle\{\Delta\circ E^{i-1}\}\gamma_{\alpha}\gamma_{\beta}G(E^{k-i}\circ\gamma_{\sigma})\rangle\rangle_{0}
+\langle\langle E^{i-1}\Delta\{\gamma_{\alpha}\circ\gamma_{\beta}\}G(E^{k-i}\circ\gamma_{\sigma})\rangle\rangle_{0}\\
&&-\langle\langle E^{i-1}\gamma_{\beta}\{\Delta\circ\gamma_{\alpha}\}G(E^{k-i}\circ\gamma_{\sigma})\rangle\rangle_{0},
\een
\ben
&&\langle\langle\{ E^{k-i}\circ\gamma_{\sigma}\}\gamma_{\alpha}\{E^{i-1}\circ\gamma_{\beta}\}\Delta\rangle\rangle_{0}\\
&=&\langle\langle \{\Delta\circ E^{k-i}\}\{E^{i-1}\circ\gamma_{\beta}\}\gamma_{\alpha} \gamma_{\sigma}\rangle\rangle_{0}
+\langle\langle E^{k-i}\Delta\{E^{i-1}\circ\gamma_{\beta}\}\{\gamma_{\alpha}\circ \gamma_{\sigma}\}\rangle\rangle_{0}\\
&&-\langle\langle E^{k-i}\{E^{i-1}\circ\gamma_{\beta}\}\{\Delta\circ\gamma_{\alpha}\}\gamma_{\sigma}\rangle\rangle_{0},
\een
\ben
&&\langle\langle G(E^{k-i}
\circ\gamma_{\alpha})\{\gamma_{\beta}\circ E^{i-1}\}\gamma_{\sigma}\Delta\rangle\rangle_{0}\\
&=&\langle\langle\{\Delta\circ E^{i-1}\}\gamma_{\beta}\gamma_{\sigma}G(E^{k-i}\circ\gamma_{\alpha})\rangle\rangle_{0}
+\langle\langle E^{i-1}\Delta\{\gamma_{\beta}\circ\gamma_{\sigma}\}G(E^{k-i}\circ\gamma_{\alpha})\rangle\rangle_{0}\\
&&-\langle\langle E^{i-1}\{\Delta\circ\gamma_{\sigma}\}\gamma_{\beta}G(E^{k-i}\circ\gamma_{\alpha})\rangle\rangle_{0},
\een
\ben
&&\langle\langle\{E^{k-i}
\circ\gamma_{\alpha}\}G(\gamma_{\beta}\circ E^{i-1})\gamma_{\sigma}\Delta\rangle\rangle_{0}\\
&=&\langle\langle\{\Delta\circ E^{k-i}\}\gamma_{\alpha}\gamma_{\sigma}G(E^{i-1}\circ\gamma_{\beta})\rangle\rangle_{0}
+\langle\langle E^{k-i}\Delta\{\gamma_{\alpha}\circ\gamma_{\sigma}\}G(E^{i-1}\circ\gamma_{\beta})\rangle\rangle_{0}\\
&&-\langle\langle E^{k-i}\{\Delta\circ\gamma_{\sigma}\}\gamma_{\alpha}G(E^{k-i}\circ\gamma_{\beta})\rangle\rangle_{0},
\een
\ben
&&\langle\langle\{E^{k-i}
\circ\gamma_{\alpha}\}\{\gamma_{\beta}\circ E^{i-1}\}G(\gamma_{\sigma})\Delta\rangle\rangle_{0}\\
&=&\langle\langle\{\Delta\circ E^{k-i}
\}\gamma_{\alpha}\{\gamma_{\beta}\circ E^{i-1}\}G(\gamma_{\sigma})\rangle\rangle_{0}
+\langle\langle E^{k-i}\Delta
\{\gamma_{\alpha}\circ\gamma_{\beta}\circ E^{i-1}\}G(\gamma_{\sigma})\rangle\rangle_{0}\\
&&-\langle\langle E^{k-i}\gamma_{\alpha}
\{\Delta\circ\gamma_{\beta}\circ E^{i-1}\}G(\gamma_{\sigma})\rangle\rangle_{0},
\een
\ben
&&\langle\langle\{E^{k-i}
\circ\gamma_{\alpha}\}\{\gamma_{\beta}\circ E^{i-1}\}\gamma_{\sigma}G(\Delta)\rangle\rangle_{0}\\
&=&\langle\langle\{G(\Delta)\circ E^{k-i}
\}\gamma_{\alpha}\{\gamma_{\beta}\circ E^{i-1}\}\gamma_{\sigma}\rangle\rangle_{0}
+\langle\langle E^{k-i}\gamma_{\sigma}
\{\gamma_{\alpha}\circ\gamma_{\beta}\circ E^{i-1}\}G(\Delta)\rangle\rangle_{0}\\
&&-\langle\langle E^{k-i}\gamma_{\alpha}\gamma_{\sigma}
\{G(\Delta)\circ\gamma_{\beta}\circ E^{i-1}\}\rangle\rangle_{0},
\een
\ben
&&\langle\langle\{E^{k-i}
\circ\gamma_{\alpha}\}\{\gamma_{\beta}\circ E^{i-1}\}\gamma_{\sigma}\Delta\rangle\rangle_{0}\\
&=&\langle\langle\{\Delta\circ E^{k-i}
\}\gamma_{\alpha}\{\gamma_{\beta}\circ E^{i-1}\}\gamma_{\sigma}\rangle\rangle_{0}
+\langle\langle E^{k-i}\gamma_{\sigma}
\{\gamma_{\alpha}\circ\gamma_{\beta}\circ E^{i-1}\}\Delta\rangle\rangle_{0}\\
&&-\langle\langle E^{k-i}\gamma_{\alpha}\gamma_{\sigma}
\{\Delta\circ\gamma_{\beta}\circ E^{i-1}\}\rangle\rangle_{0}.
\een
The proof is completed by a tedious computation, i.e., by firstly collecting all the above results and using Lemma \ref{4point-3point} to reduce 4-point functions to 3-point functions, and secondly using Lemma \ref{simplication} to simplify the resulting expression,
and then symmetrizing the result.
\end{proof}

\begin{lem}\label{G0part2}For any $\alpha,\beta,\sigma$,
\ben
&&\frac{1}{2}\sum_{g\in S_{3}}\langle\langle E^k\gamma_{\varsigma_{g(1)}}\gamma_{\varsigma_{g(2)}}\gamma^{\mu}\rangle\rangle_{0}
\langle\langle\gamma_{\mu}\gamma_{\varsigma_{g(3)}}\gamma_{\rho}\gamma^{\rho}\rangle\rangle_{0}\\
&=&-\frac{1}{2}\sum_{g\in S_{3}}\sum_{i=1}^k\langle\langle\{G(E^{k-i})\circ\gamma_{\varsigma_{g(1)}}\circ\gamma_{\varsigma_{g(2)}}\}
\gamma_{\mu}\gamma^{\mu}\{E^{i-1}\circ\gamma_{\varsigma_{g(3)}}\}\rangle\rangle_{0}\\
&&-\frac{1}{2}\sum_{g\in S_{3}}\sum_{i=1}^k\langle\langle G(E^{k-i}\circ\gamma_{\varsigma_{g(1)}}\circ\gamma_{\varsigma_{g(2)}})
\gamma_{\mu}\gamma^{\mu}\{E^{i-1}\circ\gamma_{\varsigma_{g(3)}}\}\rangle\rangle_{0}\\
&&+\sum_{g\in S_{3}}\sum_{i=1}^k\langle\langle\{G(E^{k-i}\circ\gamma_{\varsigma_{g(1)}})\circ\gamma_{\varsigma_{g(2)}}\}
\gamma_{\mu}\gamma^{\mu}\{E^{i-1}\circ\gamma_{\varsigma_{g(3)}}\}\rangle\rangle_{0}\\
&&-\frac{1}{2}\sum_{g\in S_{3}}\sum_{i=1}^k\sum_{j=1}^{i-1}\langle\langle\{G(E^{k-i})\circ\Delta\circ E^{j-1}\}G(E^{i-j-1}\circ\gamma_{\varsigma_{g(1)}})\{\gamma_{\varsigma_{g(2)}}\circ\gamma_{\varsigma_{g(3)}}\}\rangle\rangle_{0}\\
&&-\frac{1}{2}\sum_{g\in S_{3}}\sum_{i=1}^k\sum_{j=1}^{i-1}\langle\langle\{G(E^{k-i})\circ E^{j-1}\}G(\Delta\circ E^{i-j-1}\circ\gamma_{\varsigma_{g(1)}})\{\gamma_{\varsigma_{g(2)}}\circ\gamma_{\varsigma_{g(3)}}\}\rangle\rangle_{0}\\
&&+3\sum_{i=1}^{k}(i-1)\langle\langle\{G(E^{k-i})\circ E^{i-2}\}\Delta\{\gamma_{\alpha}\circ\gamma_{\beta}\circ\gamma_{\sigma}\}\rangle\rangle_{0}\\
&&-\frac{1}{2}\sum_{g\in S_{3}}\sum_{i=1}^k\sum_{j=1}^{k-i}\langle\langle\{\Delta\circ E^{j-1}\}G(E^{i-1}\circ\gamma_{\varsigma_{g(1)}}\circ\gamma_{\varsigma_{g(2)}})G(E^{k-i-j}\circ\gamma_{\varsigma_{g(3)}})\rangle\rangle_{0}\\
&&-\frac{1}{2}\sum_{g\in S_{3}}\sum_{i=1}^k\sum_{j=1}^{k-i}\langle\langle E^{k-i-j}G(E^{i-1}\circ\gamma_{\varsigma_{g(1)}}\circ\gamma_{\varsigma_{g(2)}})G(\Delta\circ E^{j-1}\circ\gamma_{\varsigma_{g(3)}})\rangle\rangle_{0}\\
&&+\frac{1}{2}\sum_{g\in S_{3}}\sum_{i=1}^k(k-i)\langle\langle\{\Delta\circ E^{k-i-1}\}G(E^{i-1}\circ\gamma_{\varsigma_{g(1)}}\circ\gamma_{\varsigma_{g(2)}})\gamma_{\varsigma_{g(3)}}\rangle\rangle_{0}\\
&&+\sum_{g\in S_{3}}\sum_{i=1}^k\sum_{j=1}^{i-1}\langle\langle\{\Delta\circ E^{j-1}\circ\gamma_{\varsigma_{g(1)}}\}G(E^{k-i}\circ\gamma_{\varsigma_{g(2)}})G(E^{i-j-1}\circ\gamma_{\varsigma_{g(3)}})\rangle\rangle_{0}\\
&&+\sum_{g\in S_{3}}\sum_{i=1}^k\sum_{j=1}^{i-1}\langle\langle\{ E^{i-j-1}\circ\gamma_{\varsigma_{g(1)}}\}G(E^{k-i}\circ\gamma_{\varsigma_{g(2)}})G(\Delta\circ E^{j-1}\circ\gamma_{\varsigma_{g(3)}})\rangle\rangle_{0}\\
&&-\sum_{g\in S_{3}}\sum_{i=1}^k\sum_{j=1}^{i-1}\langle\langle\{\Delta\circ E^{i-2}\}G(E^{k-i}\circ\gamma_{\varsigma_{g(1)}})\{\gamma_{\varsigma_{g(2)}}\circ\gamma_{\varsigma_{g(3)}}\}\rangle\rangle_{0},
\een
where $k\geq0$ and $\{\varsigma_{1},\varsigma_{2},\varsigma_{3}\}=\{\alpha,\beta,\sigma\}$.
\end{lem}

\begin{proof}
It follows from equation \eqref{4point-3point1} that
\ben
&&\langle\langle E^k\gamma_{\varsigma_{g(1)}}\gamma_{\varsigma_{g(2)}}\gamma^{\mu}\rangle\rangle_{0}
\langle\langle\gamma_{\mu}\gamma_{\varsigma_{g(3)}}\gamma_{\rho}\gamma^{\rho}\rangle\rangle_{0}\\
&=&-\sum_{i=1}^{k}\langle\langle\{G(E^{k-i})\circ E^{i-1}\circ\gamma_{\varsigma_{g(1)}}\circ\gamma_{\varsigma_{g(2)}}\}\gamma_{\varsigma_{g(3)}}\gamma_{\rho}\gamma^{\rho}\rangle\rangle_{0}\\
&&-\sum_{i=1}^{k}\langle\langle\{G( E^{i-1}\circ\gamma_{\varsigma_{g(1)}}\circ\gamma_{\varsigma_{g(2)}})\circ E^{k-i}\}\gamma_{\varsigma_{g(3)}}\gamma_{\rho}\gamma^{\rho}\rangle\rangle_{0}\\
&&+\sum_{i=1}^{k}\langle\langle\{G(E^{k-i}\circ\gamma_{\varsigma_{g(1)}})\circ E^{i-1}\circ\gamma_{\varsigma_{g(2)}}\}\gamma_{\varsigma_{g(3)}}\gamma_{\rho}\gamma^{\rho}\rangle\rangle_{0}\\
&&+\sum_{i=1}^{k}\langle\langle\{G(E^{i-1}\circ\gamma_{\varsigma_{g(2)}})\circ E^{k-i}\circ\gamma_{\varsigma_{g(1)}} \}\gamma_{\varsigma_{g(3)}}\gamma_{\rho}\gamma^{\rho}\rangle\rangle_{0}.
\een
Using equation \eqref{WDVV2}, we have
\ben
&&\langle\langle\{G(E^{k-i})\circ E^{i-1}\circ\gamma_{\varsigma_{g(1)}}\circ\gamma_{\varsigma_{g(2)}}\}\gamma_{\varsigma_{g(3)}}\gamma_{\rho}\gamma^{\rho}\rangle\rangle_{0}\\
&=&\langle\langle\{G(E^{k-i})\circ\gamma_{\varsigma_{g(1)}}\circ\gamma_{\varsigma_{g(2)}} \}\gamma_{\rho}\gamma^{\rho}\{E^{i-1}\circ \gamma_{\varsigma_{g(3)}}\}\rangle\rangle_{0}\\
&&+\langle\langle E^{i-1}\gamma_{\varsigma_{g(3)}}\gamma_{\rho}\{G(E^{k-i})\circ\gamma_{\varsigma_{g(1)}}\circ\gamma_{\varsigma_{g(2)}}\circ \gamma^{\rho}\}\rangle\rangle_{0}\\
&&-\langle\langle E^{i-1}\gamma_{\rho}\{G(E^{k-i})\circ\gamma_{\varsigma_{g(1)}}\circ\gamma_{\varsigma_{g(2)}}\}\{\gamma_{\varsigma_{g(3)}}\circ \gamma^{\rho}\}\rangle\rangle_{0},
\een
\ben
&&\langle\langle\{G( E^{i-1}\circ\gamma_{\varsigma_{g(1)}}\circ\gamma_{\varsigma_{g(2)}})\circ E^{k-i}\}\gamma_{\varsigma_{g(3)}}\gamma_{\rho}\gamma^{\rho}\rangle\rangle_{0}\\
&=&\langle\langle G( E^{i-1}\circ\gamma_{\varsigma_{g(1)}}\circ\gamma_{\varsigma_{g(2)}}) \gamma_{\rho}\gamma^{\rho}\{E^{k-i}\circ\gamma_{\varsigma_{g(3)}}\rangle\rangle_{0}\\
&&+\langle\langle E^{k-i}\gamma_{\varsigma_{g(3)}}\gamma_{\rho}\{G( E^{i-1}\circ\gamma_{\varsigma_{g(1)}}\circ\gamma_{\varsigma_{g(2)}})\circ \gamma^{\rho}\}\rangle\rangle_{0}\\
&&-\langle\langle E^{k-i}\gamma_{\rho}\{\gamma_{\varsigma_{g(3)}}\circ \gamma^{\rho}\}G(E^{i-1}\circ\gamma_{\varsigma_{g(1)}}\circ\gamma_{\varsigma_{g(2)}})\rangle\rangle_{0},
\een
\ben
&&\langle\langle\{G(E^{k-i}\circ\gamma_{\varsigma_{g(1)}})\circ E^{i-1}\circ\gamma_{\varsigma_{g(2)}}\}\gamma_{\varsigma_{g(3)}}\gamma_{\rho}\gamma^{\rho}\rangle\rangle_{0}\\
&=&\langle\langle E^{i-1}\gamma_{\varsigma_{g(3)}}\gamma_{\rho}\{G(E^{k-i}\circ\gamma_{\varsigma_{g(1)}})\circ \gamma_{\varsigma_{g(2)}}\circ\gamma^{\rho}\}\rangle\rangle_{0}\\
&&+\langle\langle \{G(E^{k-i}\circ\gamma_{\varsigma_{g(1)}})\circ \gamma_{\varsigma_{g(2)}}\}\gamma_{\rho}\gamma^{\rho}\{E^{i-1}\circ\gamma_{\varsigma_{g(3)}}\}\rangle\rangle_{0}\\
&&-\langle\langle E^{i-1}\gamma_{\rho}\{\gamma^{\rho}\circ\gamma_{\varsigma_{g(3)}}\}\{G(E^{k-i}\circ\gamma_{\varsigma_{g(1)}})\circ \gamma_{\varsigma_{g(2)}}\}\rangle\rangle_{0},
\een
and
\ben
&&\langle\langle\{G(E^{i-1}\circ\gamma_{\varsigma_{g(2)}})\circ E^{k-i}\circ\gamma_{\varsigma_{g(1)}} \}\gamma_{\varsigma_{g(3)}}\gamma_{\rho}\gamma^{\rho}\rangle\rangle_{0}\\
&=&\langle\langle E^{k-i}\gamma_{\varsigma_{g(3)}}\gamma_{\rho}\{G(E^{i-1}\circ\gamma_{\varsigma_{g(2)}})\circ \gamma_{\varsigma_{g(1)}}\circ\gamma^{\rho}\}\rangle\rangle_{0}\\
&&+\langle\langle \{G(E^{i-1}\circ\gamma_{\varsigma_{g(2)}})\circ \gamma_{\varsigma_{g(1)}}\}\gamma_{\rho}\gamma^{\rho}\{E^{k-i}\circ\gamma_{\varsigma_{g(3)}}\}\rangle\rangle_{0}\\
&&-\langle\langle E^{k-i}\gamma_{\rho}\{\gamma^{\rho}\circ\gamma_{\varsigma_{g(3)}}\}\{G(E^{i-1}\circ\gamma_{\varsigma_{g(2)}})\circ \gamma_{\varsigma_{g(1)}}\}\rangle\rangle_{0}.
\een
The remainder of the argument is analogous to that in  Lemma \ref{G0part1}.
\end{proof}
\begin{lem}\label{G0part3}For any $\alpha,\beta,\sigma$,
\ben
&&\sum_{g\in S_{3}}\langle\langle E^k\gamma_{\varsigma_{g(1)}}\gamma^{\mu}\gamma^{\rho}\rangle\rangle_{0}
\langle\langle \gamma_{\mu}\gamma_{\rho}\gamma_{\varsigma_{g(2)}}\gamma_{\varsigma_{g(3)}}\rangle\rangle_{0}\\
&=&-\sum_{g\in S_{3}}\sum_{i=1}^k\langle\langle \{G(E^{k-i})\circ\gamma_{\varsigma_{g(1)}}\circ\gamma^{\mu}\}\gamma_{\varsigma_{g(2)}}\gamma_{\varsigma_{g(3)}}\{E^{i-1}\circ\gamma_{\mu}\}\rangle\rangle_{0}\\
&&-\sum_{g\in S_{3}}\sum_{i=1}^k\langle\langle G(E^{i-1}\circ\gamma_{\varsigma_{g(1)}}\circ\gamma^{\mu})\gamma_{\varsigma_{g(2)}}\gamma_{\varsigma_{g(3)}}\{E^{k-i}\circ\gamma_{\mu}\}\rangle\rangle_{0}\\
&&+\sum_{g\in S_{3}}\sum_{i=1}^k\langle\langle \{G(E^{k-i}\circ\gamma_{\varsigma_{g(1)}})\circ\gamma^{\mu}\}\gamma_{\varsigma_{g(2)}}\gamma_{\varsigma_{g(3)}}\{E^{i-1}\circ\gamma_{\mu}\}\rangle\rangle_{0}\\
&&+\sum_{g\in S_{3}}\sum_{i=1}^k\langle\langle \{G(E^{i-1}\circ\gamma^{\mu})\circ\gamma_{\varsigma_{g(1)}}\}\gamma_{\varsigma_{g(2)}}\gamma_{\varsigma_{g(3)}}\{E^{k-i}\circ\gamma_{\mu}\}\rangle\rangle_{0}\\
&&-6\sum_{i=1}^k\sum_{j=1}^{k-i}\langle\langle G(E^{i+j-2}\circ\gamma_{\alpha}\circ\gamma_{\beta}\circ\gamma_{\sigma}\circ\gamma^{\mu})G(\gamma_{\mu})E^{k-i-j}\rangle\rangle_{0}\\
&&+3\sum_{g\in S_{3}}\sum_{i=1}^k\sum_{j=1}^{k-i}\langle\langle\{G(E^{k-j-1}\circ\gamma_{\varsigma_{g(1)}}\circ\gamma_{\varsigma_{g(2)}}\circ\gamma^{\mu})\circ G(\gamma_{\mu})\}E^{j-1}\gamma_{\varsigma_{g(3)}}\rangle\rangle_{0}\\
&&-2\sum_{g\in S_{3}}\sum_{i=1}^k\sum_{j=1}^{k-i}\langle\langle\{G(E^{k-j-1}\circ\gamma_{\varsigma_{g(1)}}\circ\gamma^{\mu})\circ G(\gamma_{\mu})\}E^{j-1}\{\gamma_{\varsigma_{g(2)}}\circ\gamma_{\varsigma_{g(3)}}\}\rangle\rangle_{0}\\
&&-\sum_{g\in S_{3}}\sum_{i=1}^k\sum_{j=1}^{k-i}\langle\langle\{G(E^{i+j-2}\circ\gamma_{\varsigma_{g(1)}}\circ\gamma^{\mu})\circ G(\gamma_{\mu})\}E^{k-i-j}\{\gamma_{\varsigma_{g(2)}}\circ\gamma_{\varsigma_{g(3)}}\}\rangle\rangle_{0}\\
&&+6\sum_{i=1}^k\sum_{j=1}^{k-i}\langle\langle\{G(E^{k-j-1}\circ\gamma^{\mu})\circ G(\gamma_{\mu})\}E^{j-1}\{\gamma_{\alpha}\circ\gamma_{\beta}\circ\gamma_{\sigma}\}\rangle\rangle_{0},
\een
where $k\geq0$ and $\{\varsigma_{1},\varsigma_{2},\varsigma_{3}\}=\{\alpha,\beta,\sigma\}$.
\end{lem}
\begin{proof}
The proof is completed by the same argument in that of Lemma \ref{G0part2}.
\end{proof}

\begin{lem}\label{G0part4}For $k\geq0$, and any $\alpha,\beta,\sigma$,
\ben
&&\langle\langle E^k\gamma_{\rho}\gamma^{\rho}\gamma^{\mu}\rangle\rangle_{0}\langle\langle \gamma_{\mu}\gamma_{\alpha}\gamma_{\beta}\gamma_{\sigma}\rangle\rangle_{0}\\
&=&-\sum_{i=1}^{k}\langle\langle \{G(E^{k-i})\circ\Delta\circ E^{i-1}\}\gamma_{\alpha}\gamma_{\beta}\gamma_{\sigma}\rangle\rangle_{0}-\sum_{i=1}^{k}\langle\langle \{G(\Delta\circ E^{i-1})\circ E^{k-i}\}\gamma_{\alpha}\gamma_{\beta}\gamma_{\sigma}\rangle\rangle_{0}\\
&&+\sum_{i=1}^{k}\langle\langle \{\Delta\circ E^{k-1}\}\gamma_{\alpha}\gamma_{\beta}\gamma_{\sigma}\rangle\rangle_{0}.
\een
\end{lem}
\begin{proof}
It follows from equations \eqref{4point-3point1} and \eqref{simplication2}.
\end{proof}
By equation \eqref{G0formula1} and Lemmas \ref{G0part1}, \ref{G0part2}, \ref{G0part3}, \ref{G0part4}, one can obtain an  expression for $G_{0}(E^{k},\gamma_{\alpha},\gamma_{\beta},\gamma_{\sigma})$.  To simplify this expression, we have to reduce the number of 4-point functions.
In order to cancel redundant 4-point functions, the strategy is to repeatedly use  equation \eqref{WDVV2} to obtain some common 4-point functions, and
then use Lemma \ref{4point-3point} and Lemma \ref{simplication} for possible reduction. The following results are used to maximally reduce the number of 4-point functions.

\begin{lem}\label{G0part11}
\ben
&&\langle\langle\{G(E^{k-i})\circ\Delta\}\gamma_{\varsigma_{g(1)}}\gamma_{\varsigma_{g(2)}}\{E^{i-1}\circ\gamma_{\varsigma_{g(3)}}\}\rangle\rangle_{0}\\
&=&\langle\langle\{G(E^{k-i})\circ\Delta\circ E^{i-1}\}\gamma_{\alpha}\gamma_{\beta}\gamma_{\sigma}\rangle\rangle_{0}\\
&&-\sum_{j=1}^{i-1}\langle\langle \{G(E^{k-i})\circ E^{i-j-1}\}\Delta G(E^{j-1}\circ\gamma_{\alpha}\circ\gamma_{\beta}\circ\gamma_{\sigma})\rangle\rangle_{0}\\
&&+\sum_{j=1}^{i-1}\langle\langle \{G(E^{k-i})\circ\Delta\circ E^{j-1}\}G(E^{i-j-1}\circ\gamma_{\varsigma_{g(2)}}\circ\gamma_{\varsigma_{g(3)}})\gamma_{\varsigma_{g(1)}}\rangle\rangle_{0}\\
&&+\sum_{j=1}^{i-1}\langle\langle \{G(E^{k-i})\circ\Delta\circ E^{i-j-1}\}G(E^{j-1}\circ\gamma_{\varsigma_{g(1)}}\circ\gamma_{\varsigma_{g(3)}})\gamma_{\varsigma_{g(2)}}\rangle\rangle_{0}\\
&&-\sum_{j=1}^{i-1}\langle\langle \{G(E^{k-i})\circ\Delta\circ E^{j-1}\}G(E^{i-j-1}\circ\gamma_{\varsigma_{g(3)}})\{\gamma_{\varsigma_{g(1)}}\circ\gamma_{\varsigma_{g(2)}}\}\rangle\rangle_{0}.
\een
\end{lem}

\begin{proof}
By equation \eqref{WDVV2}, we have
\ben
&&\langle\langle\{G(E^{k-i})\circ\Delta\}\gamma_{\varsigma_{g(1)}}\gamma_{\varsigma_{g(2)}}\{E^{i-1}\circ\gamma_{\varsigma_{g(3)}}\}\rangle\rangle_{0}\\
&=&\langle\langle\{G(E^{k-i})\circ\Delta\circ E^{i-1}\}\gamma_{\alpha}\gamma_{\beta}\gamma_{\sigma}\rangle\rangle_{0}
+\langle\langle E^{i-1}\{\gamma_{\varsigma_{g(2)}}\circ\gamma_{\varsigma_{g(3)}}\}\gamma_{\varsigma_{g(1)}}\{G(E^{k-i})\circ\Delta\}\rangle\rangle_{0}\\
&&-\langle\langle E^{i-1}\gamma_{\varsigma_{g(1)}}\gamma_{\varsigma_{g(3)}}\{G(E^{k-i})\circ\Delta\circ\gamma_{\varsigma_{g(2)}}\}\rangle\rangle_{0}.
\een
The proof is completed by using  Lemma \ref{4point-3point}.
\end{proof}
By the same argument as in Lemma \ref{G0part11}, we have the following results, i.e., Lemmas \ref{G0part12}, \ref{G0part13}, \ref{G0part14}, \ref{G0part15}.
\begin{lem}\label{G0part12}
\ben
&&\langle\langle G(\Delta\circ E^{k-i})\gamma_{\varsigma_{g(1)}}\gamma_{\varsigma_{g(2)}}\{E^{i-1}\circ\gamma_{\varsigma_{g(3)}}\}\rangle\rangle_{0}\\
&=&\langle\langle\{G(\Delta\circ E^{k-i})\circ E^{i-1}\}\gamma_{\alpha}\gamma_{\beta}\gamma_{\sigma}\rangle\rangle_{0}\\
&&-\sum_{j=1}^{i-1}\langle\langle G(\Delta\circ E^{k-i}) E^{i-j-1} G(E^{j-1}\circ\gamma_{\alpha}\circ\gamma_{\beta}\circ\gamma_{\sigma})\rangle\rangle_{0}\\
&&+\sum_{j=1}^{i-1}\langle\langle \{G(\Delta\circ E^{k-i})\circ E^{j-1}\}G(E^{i-j-1}\circ\gamma_{\varsigma_{g(2)}}\circ\gamma_{\varsigma_{g(3)}})\gamma_{\varsigma_{g(1)}}\rangle\rangle_{0}\\
&&+\sum_{j=1}^{i-1}\langle\langle \{G(\Delta\circ E^{k-i})\circ E^{i-j-1}\}G(E^{j-1}\circ\gamma_{\varsigma_{g(1)}}\circ\gamma_{\varsigma_{g(3)}})\gamma_{\varsigma_{g(2)}}\rangle\rangle_{0}\\
&&-\sum_{j=1}^{i-1}\langle\langle \{G(\Delta\circ E^{k-i})\circ E^{j-1}\}G(E^{i-j-1}\circ\gamma_{\varsigma_{g(3)}})\{\gamma_{\varsigma_{g(1)}}\circ\gamma_{\varsigma_{g(2)}}\}\rangle\rangle_{0}.
\een
\end{lem}
%\begin{proof}
%The proof follows by the same argument as Lemma \ref{G0part12}.
%\ben
%&&\langle\langle G(\Delta\circ E^{k-i})\gamma_{\varsigma_{g(1)}}\gamma_{\varsigma_{g(2)}}\{E^{i-1}\circ\gamma_{\varsigma_{g(3)}}\}\rangle\rangle_{0}\\
%&=&\langle\langle \{G(\Delta\circ E^{k-i})\circ E^{i-1}\}\gamma_{\alpha}\gamma_{\beta}\gamma_{\sigma}\rangle\rangle_{0}
%+\langle\langle E^{i-1}\{\gamma_{\varsigma_{g(2)}}\circ\gamma_{\varsigma_{g(3)}}\}\gamma_{\varsigma_{g(1)}}G(\Delta\circ E^{k-i})\rangle\rangle_{0}\\
%&&-\langle\langle E^{i-1}\gamma_{\varsigma_{g(1)}}\gamma_{\varsigma_{g(3)}}\{G(\Delta\circ E^{k-i})\circ\gamma_{\varsigma_{g(2)}}\}\rangle\rangle_{0},
%\een
% equation \eqref{WDVV1} and Lemma \ref{4point-3point}.
%\end{proof}

\begin{lem}\label{G0part13}
\ben
&&\langle\langle\{\Delta\circ E^{k-i}\}\gamma_{\varsigma_{g(1)}}\gamma_{\varsigma_{g(2)}}\{E^{i-1}\circ\gamma_{\varsigma_{g(3)}}\}\rangle\rangle_{0}\\
&=&\langle\langle\{\Delta\circ E^{k-1}\}\gamma_{\alpha}\gamma_{\beta}\gamma_{\sigma}\rangle\rangle_{0}\\
&&-\sum_{j=1}^{i-1}\langle\langle \Delta  E^{k-j-1} G(E^{j-1}\circ\gamma_{\alpha}\circ\gamma_{\beta}\circ\gamma_{\sigma})\rangle\rangle_{0}\\
&&+\sum_{j=1}^{i-1}\langle\langle \{\Delta\circ E^{k-i+j-1}\}G(E^{i-j-1}\circ\gamma_{\varsigma_{g(2)}}\circ\gamma_{\varsigma_{g(3)}})\gamma_{\varsigma_{g(1)}}\rangle\rangle_{0}\\
&&+\sum_{j=1}^{i-1}\langle\langle \{\Delta\circ E^{k-j-1}\}G(E^{j-1}\circ\gamma_{\varsigma_{g(1)}}\circ\gamma_{\varsigma_{g(3)}})\gamma_{\varsigma_{g(2)}}\rangle\rangle_{0}\\
&&-\sum_{j=1}^{i-1}\langle\langle \{\Delta\circ E^{k-i+j-1}\}G(E^{i-j-1}\circ\gamma_{\varsigma_{g(3)}})\{\gamma_{\varsigma_{g(1)}}\circ\gamma_{\varsigma_{g(2)}}\}\rangle\rangle_{0}.
\een

\end{lem}

%\begin{proof}
%Using equation \eqref{WDVV2}, we have
%\ben
%&&\langle\langle\{\Delta\circ E^{k-i}\}\gamma_{\varsigma_{g(1)}}\gamma_{\varsigma_{g(2)}}\{E^{i-1}\circ\gamma_{\varsigma_{g(3)}}\}\rangle\rangle_{0}\\
%&=&\langle\langle\{\Delta\circ E^{k-1}\}\gamma_{\alpha}\gamma_{\beta}\gamma_{\sigma}\rangle\rangle_{0}
%+\langle\langle E^{i-1}\gamma_{\varsigma_{g(1)}}\{\gamma_{\varsigma_{g(2)}}\circ\gamma_{\varsigma_{g(3)}}\}\{\Delta\circ E^{k-i}\}\rangle\rangle_{0}\\
%&&-\langle\langle E^{i-1}\gamma_{\varsigma_{g(1)}}\gamma_{\varsigma_{g(3)}}\{\Delta\circ E^{k-i}\circ\gamma_{\varsigma_{g(2)}}\}\rangle\rangle
%\een
%The proof is completed by using  Lemma \ref{4point-3point}.

%\end{proof}

\begin{lem}\label{G0part14}
\ben
&&\langle\langle G(E^{k-i})\Delta\gamma_{\varsigma_{g(1)}}\{E^{i-1}\circ\gamma_{\varsigma_{g(2)}}\circ\gamma_{\varsigma_{g(3)}}\}\rangle\rangle_{0}\\
&=&\langle\langle\{\Delta\circ E^{i-1}\}\{\gamma_{\varsigma_{g(2)}}\circ\gamma_{\varsigma_{g(3)}}\}\gamma_{\varsigma_{g(1)}}G(E^{k-i})\rangle\rangle_{0}\\
&&+\sum_{j=1}^{i-1}\langle\langle\{G(\Delta\circ E^{i-j-1})\circ G(E^{k-i})\}E^{j-1}\{\gamma_{\alpha}\circ\gamma_{\beta}\circ\gamma_{\sigma}\}\rangle\rangle_{0}\\
&&+\sum_{j=1}^{i-1}\langle\langle\{G(E^{k-i})\circ E^{i-j-1}\}\Delta G(E^{j-1}\circ\gamma_{\alpha}\circ\gamma_{\beta}\circ\gamma_{\sigma})\rangle\rangle_{0}\\
&&-\sum_{j=1}^{i-1}\langle\langle\{G(E^{k-i})\circ\Delta\circ E^{j-1}\} G(E^{i-j-1}\circ\gamma_{\varsigma_{g(2)}}\circ\gamma_{\varsigma_{g(3)}})\gamma_{\varsigma_{g(1)}}\rangle\rangle_{0}\\
&&-\sum_{j=1}^{i-1}\langle\langle\{G(E^{k-i})\circ E^{i-j-1}\}G(\Delta\circ E^{j-1}\circ\gamma_{\varsigma_{g(1)}})\{\gamma_{\varsigma_{g(2)}}\circ\gamma_{\varsigma_{g(3)}}\}\rangle\rangle_{0}.
\een
\end{lem}

%\begin{proof}
%The proof is completed by
%\ben
%&&\langle\langle G(E^{k-i})\Delta\gamma_{\varsigma_{g(1)}}\{E^{i-1}\circ\gamma_{\varsigma_{g(2)}}\circ\gamma_{\varsigma_{g(3)}}\}\rangle\rangle_{0}\\
%&=&\langle\langle\{\Delta\circ E^{i-1}\}\{\gamma_{\varsigma_{g(2)}}\circ\gamma_{\varsigma_{g(3)}}\}\gamma_{\varsigma_{g(1)}}G(E^{k-i})\rangle\rangle_{0}
%+\langle\langle E^{i-1}\{\gamma_{\alpha}\circ\gamma_{\beta}\circ\gamma_{\sigma}\}\Delta G(E^{k-i})\rangle\rangle_{0}\\
%&&-\langle\langle E^{i-1}\{\Delta\circ\gamma_{\varsigma_{g(1)}}\}\{\gamma_{\varsigma_{g(2)}}\circ\gamma_{\varsigma_{g(3)}}\}G(E^{k-i})\rangle\rangle_{0},
%\een
%equation \eqref{WDVV1} and Lemma \ref{4point-3point}.
%\end{proof}

\begin{lem}\label{G0part15}
\ben
&&\langle\langle\{E^{k-1}\circ\gamma_{\varsigma_{g(2)}}\}\Delta\gamma_{\varsigma_{g(1)}}G(\gamma_{\varsigma_{g(3)}})\rangle\rangle_{0}\\
&=&\langle\langle\{\Delta\circ E^{k-1}\}\gamma_{\varsigma_{g(1)}}\gamma_{\varsigma_{g(2)}}G(\gamma_{\varsigma_{g(3)}})\rangle\rangle_{0}\\
&&+\sum_{i=1}^{k-1}\langle\langle \{G(\Delta\circ E^{k-i-1})\circ E^{i-1}\}\{\gamma_{\varsigma_{g(1)}}\circ\gamma_{\varsigma_{g(2)}}\} G(\gamma_{\varsigma_{g(3)}})\rangle\rangle_{0}\\
&&+\sum_{i=1}^{k-1}\langle\langle \{\Delta\circ E^{k-i-1}\}G( E^{i-1}\circ\gamma_{\varsigma_{g(1)}}\circ\gamma_{\varsigma_{g(2)}}) G(\gamma_{\varsigma_{g(3)}})\rangle\rangle_{0}\\
&&-\sum_{i=1}^{k-1}\langle\langle \{\Delta\circ E^{i-1}\}G( E^{k-i-1}\circ\gamma_{\varsigma_{g(2)}}) \{\gamma_{\varsigma_{g(1)}}\circ G(\gamma_{\varsigma_{g(3)}})\}\rangle\rangle_{0}\\
&&-\sum_{i=1}^{k-1}\langle\langle E^{k-i-1}G(\Delta\circ  E^{i-1}\circ\gamma_{\varsigma_{g(1)}})\{\gamma_{\varsigma_{g(2)}}\circ G(\gamma_{\varsigma_{g(3)}})\}\rangle\rangle_{0}.
\een

\end{lem}

%\begin{proof}
%The proof is completed by
%\ben
%&&\langle\langle\{E^{k-1}\circ\gamma_{\varsigma_{g(2)}}\}\Delta\gamma_{\varsigma_{g(1)}}G(\gamma_{\varsigma_{g(3)}})\rangle\rangle_{0}\\
%&=&\langle\langle\{\Delta\circ E^{k-1}\}\gamma_{\varsigma_{g(1)}}\gamma_{\varsigma_{g(2)}}G(\gamma_{\varsigma_{g(3)}})\rangle\rangle_{0}
%+\langle\langle E^{k-1}\{\gamma_{\varsigma_{g(1)}}\circ\gamma_{\varsigma_{g(2)}}\}\Delta G(\gamma_{\varsigma_{g(3)}})\rangle\rangle_{0}\\
%&&-\langle\langle E^{k-1}\gamma_{\varsigma_{g(2)}}\{\Delta\circ\gamma_{\varsigma_{g(1)}}\}G(\gamma_{\varsigma_{g(3)}})\rangle\rangle_{0}
%\een
%equation \eqref{WDVV1} and Lemma \ref{4point-3point}.
%\end{proof}

Next, the following two lemmas produce the 4-point function appearing in $\{\gamma_{\alpha}\circ\gamma_{\beta}\circ\gamma_{\sigma}\}\Phi_{k}$
which is used for simplification.
\begin{lem}\label{G0part41}
\ben
&&\sum_{i=1}^k\langle\langle \{G(E^{k-i})\circ E^{i-1}\circ\Delta\}\gamma_{\alpha}\gamma_{\beta}\gamma_{\sigma}\rangle\rangle_{0}\\
&=&-\sum_{i=1}^{k-1}\langle\langle\{E^{i-1}\circ\gamma_{\alpha}\circ\gamma_{\beta}\circ\gamma_{\sigma}\}\gamma_{\mu}\gamma^{\mu}G(E^{k-i})\rangle\rangle_{0}\\
&&+\frac{1}{3}\sum_{g\in S_{3}}\sum_{i=1}^{k}\langle\langle\{G(E^{k-i})\circ E^{i-1}\circ\gamma_{\varsigma_{g(1)}}\circ\gamma^{\mu}\}\gamma_{\varsigma_{g(2)}}\gamma_{\varsigma_{g(3)}}\gamma_{\mu}\rangle\rangle_{0}\\
&&+\sum_{i=1}^{k}\langle\langle\{G(E^{k-i})\circ E^{i-1}\}\gamma_{\mu}\gamma^{\mu}\{\gamma_{\alpha}\circ\gamma_{\beta}\circ\gamma_{\sigma}\}\rangle\rangle_{0}\\
&&-\sum_{i=1}^k\sum_{j=1}^{i-1}\langle\langle G(E^{k-i})E^{i-j-1}G(\Delta\circ E^{j-1}\circ\gamma_{\alpha}\circ\gamma_{\beta}\circ\gamma_{\sigma})\rangle\rangle_{0}\\
&&-\frac{1}{6}\sum_{g\in S_{3}}\sum_{i=1}^k\langle\langle\{G(E^{k-i})\circ\gamma^{\mu}\}\{\gamma_{\varsigma_{g(1)}}\circ\gamma_{\varsigma_{g(2)}}\}\gamma_{\varsigma_{g(3)}}\{E^{i-1}\circ\gamma_{\mu}\}\rangle\rangle_{0}\\
&&-\sum_{i=1}^k\sum_{j=1}^{i-1}\langle\langle\{G(E^{k-i})\circ E^{i-j-1}\}\Delta G(E^{j-1}\circ\gamma_{\alpha}\circ\gamma_{\beta}\circ\gamma_{\sigma})\rangle\rangle_{0}\\
&&+\sum_{i=1}^k(i-1)\langle\langle\{G(E^{k-i})\circ E^{i-2}\}\Delta \{\gamma_{\alpha}\circ\gamma_{\beta}\circ\gamma_{\sigma}\}\rangle\rangle_{0}.
\een
\end{lem}

\begin{proof}
By equation \eqref{WDVV2}, we have
\ben
&&\langle\langle \{G(E^{k-i})\circ E^{i-1}\circ\Delta\}\gamma_{\alpha}\gamma_{\beta}\gamma_{\sigma}\rangle\rangle_{0}\\
&=&\langle\langle \{G(E^{k-i})\circ E^{i-1}\circ\gamma^{\mu}\circ\gamma_{\mu}\}\gamma_{\alpha}\gamma_{\beta}\gamma_{\sigma}\rangle\rangle_{0}\\
&=&\langle\langle \{G(E^{k-i})\circ E^{i-1}\circ\gamma^{\mu}\}\gamma_{\alpha}\gamma_{\beta}\{\gamma_{\mu}\circ\gamma_{\sigma}\}\rangle\rangle_{0}\\
&&+\langle\langle \{G(E^{k-i})\circ E^{i-1}\circ\gamma_{\alpha}\circ\gamma^{\mu}\}\gamma_{\beta}\gamma_{\sigma}\gamma_{\mu}\rangle\rangle_{0}\\
&&-\langle\langle\{G(E^{k-i})\circ E^{i-1}\circ\gamma^{\mu}\}\gamma_{\beta}\{\gamma_{\alpha}\circ\gamma_{\sigma}\}\gamma_{\mu}\rangle\rangle_{0}\\
&=&\langle\langle \{G(E^{k-i})\circ E^{i-1}\circ\gamma_{\sigma}\circ\gamma^{\mu}\}\gamma_{\alpha}\gamma_{\beta}\gamma_{\mu}\rangle\rangle_{0}\\
&&+\langle\langle \{G(E^{k-i})\circ E^{i-1}\circ\gamma_{\alpha}\circ\gamma^{\mu}\}\gamma_{\beta}\gamma_{\sigma}\gamma_{\mu}\rangle\rangle_{0}\\
&&-\langle\langle\{G(E^{k-i})\circ E^{i-1}\circ\gamma^{\mu}\}\gamma_{\beta}\{\gamma_{\alpha}\circ\gamma_{\sigma}\}\gamma_{\mu}\rangle\rangle_{0}.
\een
Using equation \eqref{WDVV2} again, we have
\ben
&&\langle\langle\{G(E^{k-i})\circ E^{i-1}\circ\gamma^{\mu}\}\gamma_{\beta}\{\gamma_{\alpha}\circ\gamma_{\sigma}\}\gamma_{\mu}\rangle\rangle_{0}\\
&=&\langle\langle\{G(E^{k-i})\circ E^{i-1}\}\{\gamma_{\beta}\circ\gamma^{\mu}\}\{\gamma_{\alpha}\circ\gamma_{\sigma}\}\gamma_{\mu}\rangle\rangle_{0}\\
&&+\langle\langle\{G(E^{k-i})\circ E^{i-1}\circ\gamma_{\alpha}\circ\gamma_{\sigma}\}\gamma_{\mu}\gamma^{\mu}\gamma_{\beta}\rangle\rangle_{0}\\
&&-\langle\langle\{G(E^{k-i})\circ E^{i-1}\}\gamma_{\mu}\gamma^{\mu}\{\gamma_{\alpha}\circ\gamma_{\beta}\circ\gamma_{\sigma}\}\rangle\rangle_{0},
\een
\ben
&&\langle\langle\{G(E^{k-i})\circ E^{i-1}\}\{\gamma_{\beta}\circ\gamma^{\mu}\}\{\gamma_{\alpha}\circ\gamma_{\sigma}\}\gamma_{\mu}\rangle\rangle_{0}\\
&=&\langle\langle\{E^{i-1}\circ \gamma_{\mu}\}\{\gamma_{\beta}\circ\gamma^{\mu}\}\{\gamma_{\alpha}\circ\gamma_{\sigma}\}G(E^{k-i})\rangle\rangle_{0}\\
&&+\langle\langle E^{i-1}\gamma_{\mu}\{\gamma_{\beta}\circ\gamma^{\mu}\}\{G(E^{k-i})\circ \gamma_{\alpha}\circ\gamma_{\sigma}\}\rangle\rangle_{0}\\
&&-\langle\langle E^{i-1}\{\gamma_{\beta}\circ\gamma^{\mu}\}\{\gamma_{\alpha}\circ\gamma_{\sigma}\circ \gamma_{\mu}\}G(E^{k-i})\rangle\rangle_{0},
\een
\ben
&&\langle\langle\{G(E^{k-i})\circ E^{i-1}\circ\gamma_{\alpha}\circ\gamma_{\sigma}\}\gamma_{\mu}\gamma^{\mu}\gamma_{\beta}\rangle\rangle_{0}\\
&=&\langle\langle\{E^{i-1}\circ\gamma_{\alpha}\circ\gamma_{\sigma}\}\gamma_{\beta}\gamma^{\mu}\{G(E^{k-i})\circ \gamma_{\mu}\}\rangle\rangle_{0}\\
&&+\langle\langle\{E^{i-1}\circ\gamma_{\alpha}\circ\gamma_{\beta}\circ\gamma_{\sigma}\}\gamma_{\mu}\gamma^{\mu}G(E^{k-i})\rangle\rangle_{0}\\
&&-\langle\langle\{E^{i-1}\circ\gamma_{\alpha}\circ\gamma_{\sigma}\}G(E^{k-i})\gamma^{\mu}\{\gamma_{\beta}\circ \gamma_{\mu}\}\rangle\rangle_{0},
\een
\ben
&&\langle\langle\{E^{i-1}\circ\gamma_{\alpha}\circ\gamma_{\sigma}\}\gamma_{\beta}\gamma^{\mu}\{G(E^{k-i})\circ \gamma_{\mu}\}\rangle\rangle_{0}\\
&=&\langle\langle\{E^{i-1}\circ\gamma^{\mu}\}\{\gamma_{\alpha}\circ\gamma_{\sigma}\}\gamma_{\beta}\{G(E^{k-i})\circ \gamma_{\mu}\}\rangle\rangle_{0}\\
&&+\langle\langle E^{i-1}\gamma_{\beta}\gamma^{\mu}\{G(E^{k-i})\circ\gamma_{\alpha}\circ\gamma_{\sigma}\circ \gamma_{\mu}\}\rangle\rangle_{0}\\
&&-\langle\langle E^{i-1}\{\gamma_{\alpha}\circ\gamma_{\sigma}\}\gamma_{\beta}\{G(E^{k-i})\circ \Delta\}\rangle\rangle_{0},
\een
and
\ben
&&\langle\langle\{E^{i-1}\circ\gamma_{\alpha}\circ\gamma_{\sigma}\}G(E^{k-i})\gamma^{\mu}\{\gamma_{\beta}\circ \gamma_{\mu}\}\rangle\rangle_{0}\\
&=&\langle\langle\{E^{i-1}\circ\gamma^{\mu}\}\{\gamma_{\alpha}\circ\gamma_{\sigma}\}\{\gamma_{\beta}\circ \gamma_{\mu}\}G(E^{k-i})\rangle\rangle_{0}\\
&&+\langle\langle E^{i-1}\gamma^{\mu}\{\gamma_{\alpha}\circ\gamma_{\beta}\circ\gamma_{\sigma}\circ\gamma_{\mu}\}G(E^{k-i})\rangle\rangle_{0}\\
&&-\langle\langle E^{i-1}\{\gamma_{\alpha}\circ\gamma_{\sigma}\}\{\gamma_{\beta}\circ\Delta\}G(E^{k-i})\rangle\rangle_{0}.
\een
Collecting all the above equations and using Lemma \ref{4point-3point} and equation \eqref{simplication2}, and then symmetrizing
the resulting expression, the proof follows.
\end{proof}
\begin{lem}\label{G0part21}
\ben
&&\langle\langle\{G(E^{k-i})\circ\gamma_{\varsigma_{g(1)}}\circ\gamma_{\varsigma_{g(2)}}\}\gamma_{\mu}\gamma^{\mu}\{E^{i-1}\circ\gamma_{\varsigma_{g(3)}}\}\rangle\rangle_{0}\\
&=&\langle\langle\{E^{i-1}\circ\gamma_{\alpha}\circ\gamma_{\beta}\circ\gamma_{\sigma}\}\gamma_{\mu}\gamma^{\mu}G(E^{k-i})\rangle\rangle_{0}\\
&&+\langle\langle\{G(E^{k-i})\circ\gamma^{\mu}\}\{\gamma_{\varsigma_{g(1)}}\circ\gamma_{\varsigma_{g(2)}}\}\gamma_{\varsigma_{g(3)}}\{E^{i-1}\circ\gamma_{\mu}\}\rangle\rangle_{0}\\
&&-\langle\langle G(E^{k-i})\{\gamma_{\varsigma_{g(1)}}\circ\gamma_{\varsigma_{g(2)}}\}\{E^{i-1}\circ\gamma_{\mu}\}\{\gamma_{\varsigma_{g(3)}}\circ\gamma^{\mu}\}\rangle\rangle_{0}\\
&&+\sum_{j=1}^{i-1}\langle\langle\{G(E^{k-i})\circ E^{i-j-1}\}\Delta G(E^{j-1}\circ\gamma_{\alpha}\circ\gamma_{\beta}\circ\gamma_{\sigma})\rangle\rangle_{0}\\
&&-\sum_{j=1}^{i-1}\langle\langle\{G(E^{k-i})\circ\Delta \circ E^{j-1}\} G(E^{i-j-1}\circ\gamma_{\varsigma_{g(3)}})\{\gamma_{\varsigma_{g(1)}}\circ\gamma_{\varsigma_{g(2)}}\}\rangle\rangle_{0}.
\een
\end{lem}
\begin{proof}
By equation \eqref{WDVV2}, we have
\ben
&&\langle\langle\{G(E^{k-i})\circ\gamma_{\varsigma_{g(1)}}\circ\gamma_{\varsigma_{g(2)}}\}\gamma_{\mu}\gamma^{\mu}\{E^{i-1}\circ\gamma_{\varsigma_{g(3)}}\}\rangle\rangle_{0}\\
&=&\langle\langle\{E^{i-1}\circ\gamma_{\alpha}\circ\gamma_{\beta}\circ\gamma_{\sigma}\}\gamma_{\mu}\gamma^{\mu}G(E^{k-i})\rangle\rangle_{0}\\
&&+\langle\langle\{G(E^{k-i})\circ\gamma^{\mu}\}\{\gamma_{\varsigma_{g(1)}}\circ\gamma_{\varsigma_{g(2)}}\}\{E^{i-1}\circ\gamma_{\varsigma_{g(3)}}\}\gamma_{\mu}\rangle\rangle_{0}\\
&&-\langle\langle G(E^{k-i})\{\gamma_{\varsigma_{g(1)}}\circ\gamma_{\varsigma_{g(2)}}\}\gamma_{\mu}\{E^{i-1}\circ\gamma_{\varsigma_{g(3)}}\circ\gamma^{\mu}\}\rangle\rangle_{0},
\een
and
\ben
&&\langle\langle\{G(E^{k-i})\circ\gamma^{\mu}\}\{\gamma_{\varsigma_{g(1)}}\circ\gamma_{\varsigma_{g(2)}}\}\{E^{i-1}\circ\gamma_{\varsigma_{g(3)}}\}\gamma_{\mu}\rangle\rangle_{0}\\
&=&\langle\langle\{G(E^{k-i})\circ\gamma^{\mu}\}\{\gamma_{\varsigma_{g(1)}}\circ\gamma_{\varsigma_{g(2)}}\}\gamma_{\varsigma_{g(3)}}\{E^{i-1}\circ\gamma_{\mu}\}\rangle\rangle_{0}\\
&&+\langle\langle E^{i-1}\gamma_{\mu}\{G(E^{k-i})\circ\gamma^{\mu}\}\{\gamma_{\alpha}\circ\gamma_{\beta}\circ\gamma_{\sigma}\}\rangle\rangle_{0}\\
&&-\langle\langle E^{i-1}\gamma_{\varsigma_{g(3)}}\{\gamma_{\varsigma_{g(1)}}\circ\gamma_{\varsigma_{g(2)}}\circ\gamma_{\mu}\}\{G(E^{k-i})\circ\gamma^{\mu}\}\rangle\rangle_{0}.
\een
Notice that
\ben
&&\langle\langle G(E^{k-i})\{\gamma_{\varsigma_{g(1)}}\circ\gamma_{\varsigma_{g(2)}}\}\gamma_{\mu}\{E^{i-1}\circ\gamma_{\varsigma_{g(3)}}\circ\gamma^{\mu}\}\rangle\rangle_{0}\\
&=&\langle\langle G(E^{k-i})\{\gamma_{\varsigma_{g(1)}}\circ\gamma_{\varsigma_{g(2)}}\}\{E^{i-1}\circ\gamma_{\mu}\}
\{\gamma_{\varsigma_{g(3)}}\circ\gamma^{\mu}\}\rangle\rangle_{0}.
\een
The proof is completed by using Lemma \ref{4point-3point} and equation \eqref{simplication2}.
\end{proof}

%----------The second stage for computing $G_{0}$
The following results are useful for further simplification.
\begin{lem}\label{4pointreduction1}
\ben
&&-\sum_{g\in S_{3}}\sum_{i=1}^k\langle\langle\{G(E^{k-i}\circ\gamma_{\varsigma_{g(1)}})\circ E^{i-1}\circ\gamma^{\mu}\}\gamma_{\varsigma_{g(2)}}\gamma_{\varsigma_{g(3)}}\gamma_{\mu}\rangle\rangle_{0}\\
&&-\frac{1}{6}\sum_{g\in S_{3}}\sum_{i=1}^k\langle\langle\{G(E^{k-i})\circ E^{i-1}\circ\gamma^{\mu}\}\{\gamma_{\varsigma_{g(1)}}\circ\gamma_{\varsigma_{g(2)}}\}\gamma_{\varsigma_{g(3)}}\gamma_{\mu}\rangle\rangle_{0}\\
&&+\frac{1}{3}\sum_{g\in S_{3}}\sum_{i=1}^k\langle\langle\{G(E^{k-i})\circ E^{i-1}\circ\gamma_{\varsigma_{g(1)}}\circ \gamma^{\mu}\}\gamma_{\varsigma_{g(2)}}\gamma_{\varsigma_{g(3)}}\gamma_{\mu}\rangle\rangle_{0}\\
&=&-\sum_{g\in S_{3}}\sum_{i=1}^k\langle\langle\{G(E^{k-i}\circ\gamma_{\varsigma_{g(1)}})\circ\gamma_{\varsigma_{g(2)}}\}\gamma_{\mu}\gamma^{\mu}\{E^{i-1}\circ\gamma_{\varsigma_{g(3)}}\}\rangle\rangle_{0}\\
&&-\sum_{g\in S_{3}}\sum_{i=1}^k\langle\langle\{\Delta\circ E^{i-1}\}G(E^{k-i}\circ\gamma_{\varsigma_{g(1)}})\gamma_{\varsigma_{g(2)}}\gamma_{\varsigma_{g(3)}}\rangle\rangle_{0}\\
&&+\sum_{g\in S_{3}}\sum_{i=1}^k\langle\langle G(E^{k-i}\circ\gamma_{\varsigma_{g(1)}})\gamma_{\varsigma_{g(2)}}\{\gamma_{\varsigma_{g(3)}}\circ\gamma^{\mu}\}\{E^{i-1}\circ\gamma_{\mu}\}\rangle\rangle_{0}\\
&&-\frac{1}{6}\sum_{g\in S_{3}}\sum_{i=1}^k\langle\langle\{\Delta\circ E^{i-1}\}\{\gamma_{\varsigma_{g(1)}}\circ\gamma_{\varsigma_{g(2)}}\}\gamma_{\varsigma_{g(3)}}G(E^{k-i})\rangle\rangle_{0}\\
&&+\frac{1}{3}\sum_{g\in S_{3}}\sum_{i=1}^k\langle\langle G(E^{k-i})\{\gamma_{\varsigma_{g(1)}}\circ\gamma_{\varsigma_{g(2)}}\circ\gamma^{\mu}\}\gamma_{\varsigma_{g(3)}}\{E^{i-1}\circ\gamma_{\mu}\}\rangle\rangle_{0}\\
&&+\frac{1}{6}\sum_{g\in S_{3}}\sum_{i=1}^k\langle\langle\{\Delta\circ E^{i-1}\}\gamma_{\varsigma_{g(1)}}\gamma_{\varsigma_{g(2)}}\{G(E^{k-i})\circ\gamma_{\varsigma_{g(3)}}\}\rangle\rangle_{0}\\
&&-\frac{1}{6}\sum_{g\in S_{3}}\sum_{i=1}^k\langle\langle G(E^{k-i})\{\gamma_{\varsigma_{g(1)}}\circ\gamma_{\varsigma_{g(2)}}\}\{\gamma_{\varsigma_{g(3)}}\circ\gamma^{\mu}\}\{E^{i-1}\circ\gamma_{\mu}\}\rangle\rangle_{0}\\
&&-\sum_{g\in S_{3}}\sum_{i=1}^k\sum_{j=1}^{i-1}\langle\langle\{G(\Delta\circ E^{i-j-1})\circ E^{j-1}\}G(E^{k-i}\circ\gamma_{\varsigma_{g(1)}})\{\gamma_{\varsigma_{g(2)}}\circ\gamma_{\varsigma_{g(3)}}\}\rangle\rangle_{0}\\
&&-\sum_{g\in S_{3}}\sum_{i=1}^k\sum_{j=1}^{i-1}\langle\langle\{\Delta\circ E^{i-j-1}\circ\gamma_{\varsigma_{g(1)}}\}G(E^{k-i}\circ\gamma_{\varsigma_{g(2)}})G(E^{j-1}\circ\gamma_{\varsigma_{g(3)}})\rangle\rangle_{0}\\
&&+\sum_{g\in S_{3}}\sum_{i=1}^k\sum_{j=1}^{i-1}\langle\langle\{\Delta\circ E^{i-2}\}G(E^{k-i}\circ\gamma_{\varsigma_{g(1)}})\{\gamma_{\varsigma_{g(2)}}\circ\gamma_{\varsigma_{g(3)}}\}\rangle\rangle_{0}.
\een
\end{lem}

\begin{proof}
Using equation \eqref{WDVV2}, we have
\ben
&&\langle\langle\{G(E^{k-i}\circ\gamma_{\varsigma_{g(1)}})\circ E^{i-1}\circ\gamma^{\mu}\}\gamma_{\varsigma_{g(2)}}\gamma_{\varsigma_{g(3)}}\gamma_{\mu}\rangle\rangle_{0}\\
&=&\langle\langle\{\Delta\circ E^{i-1}\}G(E^{k-i}\circ\gamma_{\varsigma_{g(1)}})\gamma_{\varsigma_{g(2)}}\gamma_{\varsigma_{g(3)}}\rangle\rangle_{0}\\
&&+\langle\langle\{G(E^{k-i}\circ\gamma_{\varsigma_{g(1)}})\circ \gamma_{\varsigma_{g(3)}}\}\{E^{i-1}\circ\gamma^{\mu}\}\gamma_{\mu}\gamma_{\varsigma_{g(2)}}\rangle\rangle_{0}\\
&&-\langle\langle\{E^{i-1}\circ\gamma^{\mu}\}G(E^{k-i}\circ\gamma_{\varsigma_{g(1)}})\gamma_{\varsigma_{g(2)}}
\{\gamma_{\varsigma_{g(3)}}\circ\gamma_{\mu}\}\rangle\rangle_{0},
\een
\ben
&&\langle\langle\{G(E^{k-i}\circ\gamma_{\varsigma_{g(1)}})\circ \gamma_{\varsigma_{g(3)}}\}\{E^{i-1}\circ\gamma^{\mu}\}\gamma_{\mu}\gamma_{\varsigma_{g(2)}}\rangle\rangle_{0}\\
&=&\langle\langle E^{i-1}\Delta\gamma_{\varsigma_{g(2)}}\{G(E^{k-i}\circ\gamma_{\varsigma_{g(1)}})\circ\gamma_{\varsigma_{g(3)}}\}\rangle\rangle_{0}\\
&&+\langle\langle\{G(E^{k-i}\circ\gamma_{\varsigma_{g(1)}})\circ\gamma_{\varsigma_{g(3)}}\}\gamma_{\mu}\gamma^{\mu}\{E^{i-1}\circ\gamma_{\varsigma_{g(2)}}\}\rangle\rangle_{0}\\
&&-\langle\langle E^{i-1}\gamma^{\mu}\{\gamma_{\varsigma_{g(3)}}\circ\gamma_{\mu}\}
\{G(E^{k-i}\circ\gamma_{\varsigma_{g(1)}})\circ\gamma_{\varsigma_{g(3)}}\}\rangle\rangle_{0},
\een
and
\ben
&&\langle\langle\{G(E^{k-i})\circ E^{i-1}\circ\gamma^{\mu}\}\{\gamma_{\varsigma_{g(1)}}\circ\gamma_{\varsigma_{g(2)}}\}\gamma_{\varsigma_{g(3)}}\gamma_{\mu}\rangle\rangle_{0}\\
&=&\langle\langle\{G(E^{k-i})\circ \gamma_{\varsigma_{g(1)}}\circ\gamma_{\varsigma_{g(2)}}\}\gamma_{\varsigma_{g(3)}}\gamma_{\mu}\{ E^{i-1}\circ\gamma^{\mu}\}\rangle\rangle_{0}\\
&&+\langle\langle\{\Delta\circ E^{i-1}\}G(E^{k-i})\{\gamma_{\varsigma_{g(1)}}\circ\gamma_{\varsigma_{g(2)}}\}\gamma_{\varsigma_{g(3)}}\rangle\rangle_{0}\\
&&-\langle\langle\{E^{i-1}\circ\gamma^{\mu}\}G(E^{k-i})
\{\gamma_{\varsigma_{g(1)}}\circ\gamma_{\varsigma_{g(2)}}\circ\gamma_{\mu}\}\gamma_{\varsigma_{g(3)}}\rangle\rangle_{0}.
\een
Since by equation \eqref{WDVV2} again, we have
\ben
&&\langle\langle\{G(E^{k-i})\circ E^{i-1}\circ\gamma_{\varsigma_{g(1)}}\circ \gamma^{\mu}\}\gamma_{\varsigma_{g(2)}}\gamma_{\varsigma_{g(3)}}\gamma_{\mu}\rangle\rangle_{0}\\
&=&\langle\langle\{E^{i-1}\circ\gamma^{\mu}\}\gamma_{\varsigma_{g(2)}}\gamma_{\mu}\{G(E^{k-i})\circ\gamma_{\varsigma_{g(1)}}\circ\gamma_{\varsigma_{g(3)}}\}\rangle\rangle_{0}\\
&&+\langle\langle\{\Delta\circ E^{i-1}\}\{G(E^{k-i})\circ\gamma_{\varsigma_{g(1)}}\}\gamma_{\varsigma_{g(2)}}\gamma_{\varsigma_{g(3)}}\rangle\rangle_{0}\\
&&-\langle\langle\{G(E^{k-i})\circ\gamma_{\varsigma_{g(1)}}\}\{E^{i-1}\circ\gamma^{\mu}\}
\gamma_{\varsigma_{g(2)}}\{\gamma_{\varsigma_{g(3)}}\circ\gamma_{\mu}\}\rangle\rangle_{0},
\een
and
\ben
&&\langle\langle\{G(E^{k-i})\circ\gamma_{\varsigma_{g(1)}}\}\{E^{i-1}\circ\gamma^{\mu}\}
\gamma_{\varsigma_{g(2)}}\{\gamma_{\varsigma_{g(3)}}\circ\gamma_{\mu}\}\rangle\rangle_{0}\\
&=&\langle\langle\{G(E^{k-i})\circ E^{i-1}\circ\gamma^{\mu}\}\gamma_{\varsigma_{g(1)}}
\gamma_{\varsigma_{g(2)}}\{\gamma_{\varsigma_{g(3)}}\circ\gamma_{\mu}\}\rangle\rangle_{0}\\
&&+\langle\langle G(E^{k-i})\{\gamma_{\varsigma_{g(1)}}\circ
\gamma_{\varsigma_{g(2)}}\}\{\gamma_{\varsigma_{g(3)}}\circ\gamma_{\mu}\} \{E^{i-1}\circ\gamma^{\mu}\}\rangle\rangle_{0}\\
&&-\langle\langle G(E^{k-i})\gamma_{\varsigma_{g(1)}}\{\gamma_{\varsigma_{g(3)}}\circ\gamma_{\mu}\} \{E^{i-1}\circ
\gamma_{\varsigma_{g(2)}}\circ\gamma^{\mu}\}\rangle\rangle_{0}\\
&=&\langle\langle\{G(E^{k-i})\circ\gamma_{\varsigma_{g(3)}}\circ E^{i-1}\circ\gamma^{\mu}\}\gamma_{\varsigma_{g(1)}}
\gamma_{\varsigma_{g(2)}}\gamma_{\mu}\rangle\rangle_{0}\\
&&+\langle\langle G(E^{k-i})\{\gamma_{\varsigma_{g(1)}}\circ
\gamma_{\varsigma_{g(2)}}\}\{\gamma_{\varsigma_{g(3)}}\circ\gamma_{\mu}\} \{E^{i-1}\circ\gamma^{\mu}\}\rangle\rangle_{0}\\
&&-\langle\langle G(E^{k-i})\gamma_{\varsigma_{g(1)}}\{\gamma_{\varsigma_{g(2)}}\circ\gamma_{\varsigma_{g(3)}}\circ\gamma_{\mu}\} \{E^{i-1}\circ
\gamma^{\mu}\}\rangle\rangle_{0}.
\een
Hence we have
\ben
&&\langle\langle\{G(E^{k-i})\circ\gamma_{\varsigma_{g(1)}}\circ E^{i-1}\circ\gamma^{\mu}\}\gamma_{\varsigma_{g(2)}}\gamma_{\varsigma_{g(3)}}\gamma_{\mu}\rangle\rangle_{0}\\
&&+\langle\langle\{G(E^{k-i})\circ\gamma_{\varsigma_{g(3)}}\circ E^{i-1}\circ\gamma^{\mu}\}\gamma_{\varsigma_{g(1)}}\gamma_{\varsigma_{g(2)}}\gamma_{\mu}\rangle\rangle_{0}\\
&=&\langle\langle\{G(E^{k-i})\circ\gamma_{\varsigma_{g(1)}}\circ\gamma_{\varsigma_{g(3)}}\}\gamma_{\varsigma_{g(2)}}\gamma_{\mu}\{E^{i-1}\circ\gamma^{\mu}\}\rangle\rangle_{0}\\
&&+\langle\langle\{\Delta\circ E^{i-1}\}\gamma_{\varsigma_{g(2)}}\gamma_{\varsigma_{g(3)}}\{G(E^{k-i})\circ\gamma_{\varsigma_{g(1)}}\}\rangle\rangle_{0}\\
&&-\langle\langle G(E^{k-i})\{\gamma_{\varsigma_{g(1)}}\circ\gamma_{\varsigma_{g(2)}}\}\{\gamma_{\varsigma_{g(3)}}\circ\gamma_{\mu}\}\{E^{i-1}\circ\gamma^{\mu}\}\rangle\rangle_{0}\\
&&+\langle\langle G(E^{k-i})\gamma_{\varsigma_{g(1)}}\{\gamma_{\varsigma_{g(2)}}\circ\gamma_{\varsigma_{g(2)}}\circ\gamma_{\mu}\}\{E^{i-1}\circ\gamma^{\mu}\}\rangle\rangle_{0}.
\een
The proof follows by using Lemma \ref{4point-3point} and equation \eqref{simplication2}.
\end{proof}
\begin{lem}\label{4pointreduction2}
\ben
&&-2\sum_{i=1}^k\langle\langle\{G(E^{k-i})\circ E^{i-1}\}\gamma_{\mu}\gamma^{\mu}\{\gamma_{\alpha}\circ\gamma_{\beta}\circ\gamma_{\sigma}\}\rangle\rangle_{0}\\
&=&2\sum_{i=1}^k\sum_{j=1}^{i-1}\langle\langle\{G(\Delta\circ E^{i-j-1})\circ G(E^{k-i})\}E^{j-1}\{\gamma_{\alpha}\circ\gamma_{\beta}\circ\gamma_{\sigma}\}\rangle\rangle_{0}\\
&&-2\sum_{i=1}^k\sum_{j=1}^{i-1}\langle\langle G(E^{k-i})E^{i-j-1}G(\Delta\circ E^{j-1}\circ\gamma_{\alpha}\circ\gamma_{\beta}\circ\gamma_{\sigma})\rangle\rangle_{0}\\
&&-\frac{1}{3}\sum_{g\in S_{3}}\sum_{i=1}^k\langle\langle G(E^{k-i})\gamma_{\varsigma_{g(1)}}\{\gamma_{\varsigma_{g(2)}}\circ\gamma_{\varsigma_{g(3)}}\circ\gamma^{\mu}\}\{E^{i-1}\circ\gamma_{\mu}\}\rangle\rangle_{0}\\
&&-\frac{1}{3}\sum_{g\in S_{3}}\sum_{i=1}^k\langle\langle G(E^{k-i})\{\gamma_{\varsigma_{g(1)}}\circ\gamma_{\varsigma_{g(2)}}\}\{\gamma_{\varsigma_{g(3)}}\circ\gamma^{\mu}\}\{E^{i-1}\circ\gamma_{\mu}\}\rangle\rangle_{0}\\
&&+\frac{1}{3}\sum_{g\in S_{3}}\sum_{i=1}^k\langle\langle G(E^{k-i})\{\Delta\circ E^{i-1}\}\{\gamma_{\varsigma_{g(1)}}\circ\gamma_{\varsigma_{g(2)}}\}\gamma_{\varsigma_{g(3)}}\rangle\rangle_{0}.
\een

\end{lem}
\begin{proof}
By equation \eqref{WDVV2}, we have
\ben
&&\langle\langle\{G(E^{k-i})\circ E^{i-1}\}\gamma_{\mu}\gamma^{\mu}\{\gamma_{\alpha}\circ\gamma_{\beta}\circ\gamma_{\sigma}\}\rangle\rangle_{0}\\
&=&\langle\langle\{G(E^{k-i})\circ E^{i-1}\}\gamma_{\mu}\gamma_{\alpha}\{\gamma_{\beta}\circ\gamma_{\sigma}\circ\gamma^{\mu}\}\rangle\rangle_{0}\\
&&+\langle\langle\{G(E^{k-i})\circ E^{i-1}\}\gamma^{\mu}\{\gamma_{\alpha}\circ\gamma_{\mu}\}\{\gamma_{\beta}\circ\gamma_{\sigma}\}\rangle\rangle_{0}\\
&&-\langle\langle\{G(E^{k-i})\circ E^{i-1}\}\Delta\gamma_{\alpha}\{\gamma_{\beta}\circ\gamma_{\sigma}\}\rangle\rangle_{0},
\een
\ben
&&\langle\langle\{G(E^{k-i})\circ E^{i-1}\}\gamma_{\mu}\gamma_{\alpha}\{\gamma_{\beta}\circ\gamma_{\sigma}\circ\gamma^{\mu}\}\rangle\rangle_{0}\\
&=&\langle\langle G(E^{k-i})\gamma_{\alpha}\{\gamma_{\beta}\circ\gamma_{\sigma}\circ\gamma^{\mu}\}\{E^{i-1}\circ\gamma_{\mu}\}\rangle\rangle_{0}\\
&&+\langle\langle E^{i-1}\{G(E^{k-i})\circ\gamma_{\alpha}\}\gamma_{\mu}\{\gamma_{\beta}\circ\gamma_{\sigma}\circ\gamma^{\mu}\}\rangle\rangle_{0}\\
&&-\langle\langle E^{i-1}\{\gamma_{\alpha}\circ\gamma_{\mu}\}\{\gamma_{\beta}\circ\gamma_{\sigma}\circ\gamma^{\mu}\}G(E^{k-i})\rangle\rangle_{0},
\een
\ben
&&\langle\langle\{G(E^{k-i})\circ E^{i-1}\}\gamma^{\mu}\{\gamma_{\alpha}\circ\gamma_{\mu}\}\{\gamma_{\beta}\circ\gamma_{\sigma}\}\rangle\rangle_{0}\\
&=&\langle\langle G(E^{k-i})\{\gamma_{\beta}\circ\gamma_{\sigma}\}\{\gamma_{\alpha}\circ\gamma_{\mu}\}\{E^{i-1}\circ\gamma^{\mu}\}\rangle\rangle_{0}\\
&&+\langle\langle E^{i-1}\{\gamma_{\beta}\circ\gamma_{\sigma}\}\gamma^{\mu}\{G(E^{k-i})\circ\gamma_{\alpha}\circ\gamma_{\mu}\}\rangle\rangle_{0}\\
&&-\langle\langle E^{i-1}\{\gamma_{\beta}\circ\gamma_{\sigma}\}\{\gamma_{\alpha}\circ\Delta\}G(E^{k-i})\rangle\rangle_{0},
\een
and
\ben
&&\langle\langle\{G(E^{k-i})\circ E^{i-1}\}\Delta\gamma_{\alpha}\{\gamma_{\beta}\circ\gamma_{\sigma}\}\rangle\rangle_{0}\\
&=&\langle\langle G(E^{k-i})\{\Delta\circ E^{i-1}\}\{\gamma_{\beta}\circ\gamma_{\sigma}\}\gamma_{\alpha}\rangle\rangle_{0}\\
&&+\langle\langle E^{i-1}\Delta\{\gamma_{\beta}\circ\gamma_{\sigma}\}\{G(E^{k-i})\circ\gamma_{\alpha}\}\rangle\rangle_{0}\\
&&-\langle\langle E^{i-1}\{\Delta\circ \gamma_{\alpha}\}\{\gamma_{\beta}\circ\gamma_{\sigma}\}G(E^{k-i})\rangle\rangle_{0}.
\een
%\ben
%&&\langle\langle\{G(E^{k-i})\circ E^{i-1}\}\gamma_{\mu}\gamma^{\mu}\{\gamma_{\alpha}\circ\gamma_{\beta}\circ\gamma_{\sigma}\}\rangle\rangle_{0}\\
%&=&\langle\langle G(E^{k-i})\gamma_{\alpha}\{\gamma_{\beta}\circ\gamma_{\sigma}\circ\gamma^{\mu}\}\{E^{i-1}\circ\gamma_{\mu}\}\rangle\rangle_{0}\\
%&&+\langle\langle G(E^{k-i})\{\gamma_{\beta}\circ\gamma_{\sigma}\}\{\gamma_{\alpha}\circ\gamma_{\mu}\}\{E^{i-1}\circ\gamma^{\mu}\}\rangle\rangle_{0}\\
%&&+\sum_{j=1}^{i-1}\langle\langle G(E^{k-i})E^{i-j-1}G(\Delta\circ E^{j-1}\circ\gamma_{\alpha}\circ\gamma_{\beta}\circ\gamma_{\sigma})\rangle\rangle_{0}\\
%&&-\langle\langle G(E^{k-i})\{\Delta\circ E^{i-1}\}\{\gamma_{\beta}\circ\gamma_{\sigma}\}\gamma_{\alpha}\rangle\rangle_{0}\\
%&&-\sum_{j=1}^{i-1}\langle\langle\{G(\Delta\circ E^{i-j-1})\circ G(E^{k-i}) \} E^{j-1}\{\gamma_{\alpha}\circ\gamma_{\beta}\circ\gamma_{\sigma}\}\rangle\rangle_{0}.
%\een
The proof is completed by using Lemma \ref{4point-3point} and equation \eqref{simplication2}.
\end{proof}

\begin{lem}\label{4pointreduction3}
\ben
&&\sum_{g\in S_{3}}\sum_{i=1}^k\langle\langle G(E^{k-i}\circ\gamma_{\varsigma_{g(1)}})\gamma_{\varsigma_{g(2)}}\{\gamma_{\varsigma_{g(3)}}\circ\gamma^{\mu}\}\{E^{i-1}\circ\gamma_{\mu}\}\rangle\rangle_{0}\\
&=&\frac{1}{2}\sum_{g\in S_{3}}\sum_{i=1}^k\langle\langle \{\Delta\circ E^{i-1}\} G(E^{k-i}\circ\gamma_{\varsigma_{g(1)}})\gamma_{\varsigma_{g(2)}}\gamma_{\varsigma_{g(3)}}\rangle\rangle_{0}\\
&&+\frac{1}{2}\sum_{g\in S_{3}}\sum_{i=1}^k\langle\langle G(E^{k-i}\circ\gamma_{\varsigma_{g(1)}})\gamma_{\mu}\gamma^{\mu}\{E^{i-1}\circ\gamma_{\varsigma_{g(2)}}\circ\gamma_{\varsigma_{g(3)}}\}\rangle\rangle_{0}\\
&&+\frac{1}{2}\sum_{g\in S_{3}}\sum_{i=1}^k\sum_{j=1}^{i-1}\langle\langle \{G(\Delta\circ E^{i-j-1})\circ E^{j-1}\} G(E^{k-i}\circ\gamma_{\varsigma_{g(1)}})\{\gamma_{\varsigma_{g(2)}}\circ\gamma_{\varsigma_{g(3)}}\}\rangle\rangle_{0}\\
&&+\frac{1}{2}\sum_{g\in S_{3}}\sum_{i=1}^k\sum_{j=1}^{i-1}\langle\langle \{\Delta\circ E^{i-j-1}\} G(E^{k-i}\circ\gamma_{\varsigma_{g(1)}})G(E^{j-1}\circ\gamma_{\varsigma_{g(2)}}\circ\gamma_{\varsigma_{g(3)}})\rangle\rangle_{0}\\
&&-\frac{1}{2}\sum_{g\in S_{3}}\sum_{i=1}^k\sum_{j=1}^{i-1}\langle\langle \{\Delta\circ E^{i-2}\} G(E^{k-i}\circ\gamma_{\varsigma_{g(1)}})\{\gamma_{\varsigma_{g(2)}}\circ\gamma_{\varsigma_{g(3)}}\}\rangle\rangle_{0}.
\een

\end{lem}

\begin{proof}
By equation \eqref{WDVV2}, we have
\ben
&&\langle\langle G(E^{k-i}\circ\gamma_{\varsigma_{g(1)}})\gamma_{\varsigma_{g(2)}}\{\gamma_{\varsigma_{g(3)}}\circ\gamma^{\mu}\}\{E^{i-1}\circ\gamma_{\mu}\}\rangle\rangle_{0}\\
&&+\langle\langle G(E^{k-i}\circ\gamma_{\varsigma_{g(1)}})\gamma_{\varsigma_{g(3)}}\{\gamma_{\varsigma_{g(2)}}\circ\gamma^{\mu}\}\{E^{i-1}\circ\gamma_{\mu}\}\rangle\rangle_{0}\\
&=&\langle\langle\{\Delta\circ E^{i-1}\}G(E^{k-i}\circ\gamma_{\varsigma_{g(1)}})\gamma_{\varsigma_{g(2)}}\gamma_{\varsigma_{g(3)}}\rangle\rangle_{0}\\
&&+\langle\langle G(E^{k-i}\circ\gamma_{\varsigma_{g(1)}})\{\gamma_{\varsigma_{g(2)}}\circ\gamma_{\varsigma_{g(3)}}\}\gamma^{\mu}\{E^{i-1}\circ\gamma_{\mu}\}\rangle\rangle_{0},
\een
and
\ben
&&\langle\langle G(E^{k-i}\circ\gamma_{\varsigma_{g(1)}})\{\gamma_{\varsigma_{g(2)}}\circ\gamma_{\varsigma_{g(3)}}\}\gamma^{\mu}\{E^{i-1}\circ\gamma_{\mu}\}\rangle\rangle_{0}\\
&=&\langle\langle G(E^{k-i}\circ\gamma_{\varsigma_{g(1)}})\gamma_{\mu}\gamma^{\mu}\{E^{i-1}\circ\gamma_{\varsigma_{g(2)}}\circ\gamma_{\varsigma_{g(3)}}\}\rangle\rangle_{0}\\
&&+\langle\langle  E^{i-1}\Delta \{\gamma_{\varsigma_{g(2)}}\circ\gamma_{\varsigma_{g(3)}}\}G(E^{k-i}\circ\gamma_{\varsigma_{g(1)}})\rangle\rangle_{0}\\
&&-\langle\langle E^{i-1} \gamma_{\mu}\{\gamma_{\varsigma_{g(2)}}\circ\gamma_{\varsigma_{g(3)}}\circ\gamma^{\mu}\}G(E^{k-i}\circ\gamma_{\varsigma_{g(1)}})\rangle\rangle_{0}.
\een
The proof is completed by using Lemma \ref{4point-3point} and equation \eqref{simplication2}.
\end{proof}

\begin{lem}\label{4pointreduction4}
\ben
&&-\sum_{g\in S_{3}}\langle\langle\{G(E^{k-1}\circ\gamma^{\mu})\circ\gamma_{\varsigma_{g(1)}}\}\gamma_{\varsigma_{g(2)}}\gamma_{\varsigma_{g(3)}}\gamma_{\mu}\rangle\rangle_{0}\\
&=&-\frac{1}{2}\sum_{g\in S_{3}}\langle\langle\{E^{k-1}\circ\gamma_{\varsigma_{g(1)}}\}\gamma_{\mu}\gamma^{\mu}\{\gamma_{\varsigma_{g(2)}}\circ\gamma_{\varsigma_{g(3)}}\}\rangle\rangle_{0}\\
&&+\frac{1}{2}\sum_{g\in S_{3}}\sum_{j=1}^{k-1}\langle\langle E^{k-j-1}G(\Delta\circ E^{j-1}\circ\gamma_{\varsigma_{g(1)}})\{\gamma_{\varsigma_{g(2)}}\circ\gamma_{\varsigma_{g(3)}}\}\rangle\rangle_{0}\\
&&-3\sum_{j=1}^{k-1}\langle\langle G(\Delta\circ E^{k-j-1})E^{j-1}\{\gamma_{\alpha}\circ\gamma_{\beta}\circ\gamma_{\sigma}\}\rangle\rangle_{0}\\
&&-\frac{1}{2}\sum_{g\in S_{3}}\sum_{j=1}^{k-1}\langle\langle \{\Delta\circ E^{k-j-1}\}G( E^{j-1}\circ\gamma_{\varsigma_{g(1)}})\{\gamma_{\varsigma_{g(2)}}\circ\gamma_{\varsigma_{g(3)}}\}\rangle\rangle_{0}\\
&&-\sum_{g\in S_{3}}\sum_{j=1}^{k-1}\langle\langle\{G(E^{k-j-1}\circ\gamma_{\varsigma_{g(1)}}\circ\gamma^{\mu})\circ G(\gamma_{\mu})\}E^{j-1}\{\gamma_{\varsigma_{g(2)}}\circ\gamma_{\varsigma_{g(3)}}\}\rangle\rangle_{0}\\
&&+6\sum_{j=1}^{k-1}\langle\langle\{G(E^{k-j-1}\circ\gamma^{\mu})\circ G(\gamma_{\mu})\}E^{j-1}\{\gamma_{\alpha}\circ\gamma_{\beta}\circ\gamma_{\sigma}\}\rangle\rangle_{0}\\
&&+\frac{1}{2}\sum_{g\in S_{3}}\sum_{j=1}^{k-1}\langle\langle \{\Delta\circ E^{k-j-1}\}G( E^{j-1}\circ\gamma_{\varsigma_{g(1)}}\circ\gamma_{\varsigma_{g(2)}})\gamma_{\varsigma_{g(3)}}\rangle\rangle_{0}\\
&&-3\langle\langle\{\Delta\circ E^{k-1}\}\gamma_{\alpha}\gamma_{\beta}\gamma_{\sigma}\rangle\rangle_{0}\\
&&+\sum_{g\in S_{3}}\langle\langle G(E^{k-1}\circ\gamma_{\varsigma_{g(1)}}\circ\gamma^{\mu})\gamma_{\varsigma_{g(2)}}\gamma_{\varsigma_{g(3)}}\gamma_{\mu}\rangle\rangle_{0}.
\een

\end{lem}

\begin{proof}
Using equation \eqref{WDVV2}, we have
\ben
&&\langle\langle\{G(E^{k-1}\circ\gamma^{\mu})\circ\gamma_{\varsigma_{g(1)}}\}\gamma_{\varsigma_{g(2)}}\gamma_{\varsigma_{g(3)}}\gamma_{\mu}\rangle\rangle_{0}\\
&=&\langle\langle G(E^{k-1}\circ\gamma^{\mu})\{\gamma_{\varsigma_{g(1)}}\circ\gamma_{\varsigma_{g(3)}}\}\gamma_{\varsigma_{g(2)}}\gamma_{\mu}\rangle\rangle_{0}
+\frac{1}{2}\langle\langle\{\Delta\circ E^{k-1}\}\gamma_{\alpha}\gamma_{\beta}\gamma_{\sigma}\rangle\rangle_{0}\\
&&-\langle\langle G(E^{k-1}\circ\gamma^{\mu})\gamma_{\varsigma_{g(1)}}\gamma_{\varsigma_{g(2)}}\{\gamma_{\mu}\circ\gamma_{\varsigma_{g(3)}}\}\rangle\rangle_{0}\\
&=&\langle\langle G(\gamma_{\mu})\{\gamma_{\varsigma_{g(1)}}\circ\gamma_{\varsigma_{g(3)}}\}\gamma_{\varsigma_{g(2)}}\{E^{k-1}\circ\gamma^{\mu}\}\rangle\rangle_{0}+\frac{1}{2}\langle\langle\{\Delta\circ E^{k-1}\}\gamma_{\alpha}\gamma_{\beta}\gamma_{\sigma}\rangle\rangle_{0}\\
&&-\langle\langle\{G(E^{k-1}\circ\gamma_{\varsigma_{g(3)}}\circ\gamma^{\mu})\}
\gamma_{\varsigma_{g(1)}}\gamma_{\varsigma_{g(2)}}\gamma_{\mu}\rangle\rangle_{0},
\een
and
\ben
&&\langle\langle G(\gamma_{\mu})\{\gamma_{\varsigma_{g(1)}}\circ\gamma_{\varsigma_{g(3)}}\}\gamma_{\varsigma_{g(2)}}\{E^{k-1}\circ\gamma^{\mu}\}\rangle\rangle_{0}\\
&=&\frac{1}{2}\langle\langle E^{k-1}\Delta \gamma_{\varsigma_{g(2)}}\{\gamma_{\varsigma_{g(1)}}\circ\gamma_{\varsigma_{g(3)}}\}\rangle\rangle_{0}
+\langle\langle\{E^{k-1}\circ\gamma_{\varsigma_{g(2)}}\}\gamma^{\mu}G(\gamma_{\mu})\{\gamma_{\varsigma_{g(1)}}\circ\gamma_{\varsigma_{g(3)}}\}\rangle\rangle_{0}\\
&&-\langle\langle E^{k-1}\gamma^{\mu}\{\gamma_{\varsigma_{g(1)}}\circ\gamma_{\varsigma_{g(3)}}\}\{\gamma_{\varsigma_{g(2)}}\circ G(\gamma_{\mu})\}\rangle\rangle_{0}\\
&=&\frac{1}{2}\langle\langle\{E^{k-1}\circ\gamma_{\varsigma_{g(2)}}\}\gamma_{\mu}\gamma^{\mu}\{\gamma_{\varsigma_{g(1)}}\circ\gamma_{\varsigma_{g(3)}}\}\rangle\rangle_{0}
+\frac{1}{2}\langle\langle E^{k-1}\Delta\gamma_{\varsigma_{g(2)}}\{\gamma_{\varsigma_{g(1)}}\circ\gamma_{\varsigma_{g(3)}}\}\rangle\rangle_{0}\\
&&-\langle\langle E^{k-1}\gamma^{\mu}\{\gamma_{\varsigma_{g(1)}}\circ\gamma_{\varsigma_{g(3)}}\}\{\gamma_{\varsigma_{g(2)}}\circ G(\gamma_{\mu})\}\rangle\rangle_{0}.
\een
The proof is completed by using Lemma \ref{4point-3point}.
\end{proof}

%---------------------------------------------------------------------------------------------------
%\section*{Conflict of interest statement}
%On behalf of all authors, the corresponding author states that there is no conflict of interest.

\end{document}